\documentclass[11pt]{preprint}
\usepackage[full]{textcomp}
\usepackage[osf]{newtxtext} 
\usepackage[cal=boondoxo]{mathalfa}
\usepackage{colortbl}

\usepackage{comment}

\usepackage{amssymb}
\usepackage{mathtools}
\usepackage{hyperref}
\usepackage{breakurl}
\usepackage{mhenvs}
\usepackage{mhequ} 
\usepackage{mhsymb}
\usepackage{booktabs}
\usepackage{tikz}
\usepackage{tcolorbox}
\usepackage{mathrsfs}
\usepackage[utf8]{inputenc}
\usepackage{longtable}
\usepackage{wrapfig}
\usepackage{subcaption}
\usepackage{mathrsfs}
\usepackage{epsfig}
\usepackage{microtype}
\usepackage{comment}
\usepackage{wasysym}
\usepackage{centernot}
\usepackage{enumitem}
\usepackage{bm}
\usepackage{stackrel}
\usepackage{graphicx}

\makeatletter
\newcommand{\globalcolor}[1]{%
  \color{#1}\global\let\default@color\current@color
}
\makeatother

\usetikzlibrary{calc}
\usetikzlibrary{decorations}
\usetikzlibrary{positioning}
\usetikzlibrary{shapes}
\usetikzlibrary{external}

\definecolor{blush}{rgb}{0.87, 0.36, 0.51}
	\definecolor{brightcerulean}{rgb}{0.11, 0.67, 0.84}
	\definecolor{greenryb}{rgb}{0.4, 0.69, 0.2}

\newif\ifdark
\darkfalse

\ifdark
\definecolor{darkred}{rgb}{0.9,0.2,0.2}
\definecolor{darkblue}{rgb}{0.7,0.3,1}
\definecolor{darkgreen}{rgb}{0.1,0.9,0.1}
\definecolor{franck}{rgb}{0,0.8,1}
\definecolor{pagebackground}{rgb}{.15,.21,.18}
\definecolor{pageforeground}{rgb}{.84,.84,.85}
\pagecolor{pagebackground}
\AtBeginDocument{\globalcolor{pageforeground}}
\definecolor{symbols}{rgb}{0,0.7,1}
\colorlet{connection}{red!80!black}
\colorlet{boxcolor}{blue!50}

\else

\definecolor{darkred}{rgb}{0.7,0.1,0.1}
\definecolor{darkblue}{rgb}{0.4,0.1,0.8}
\definecolor{darkgreen}{rgb}{0.1,0.7,0.1}
\definecolor{franck}{rgb}{0,0,1}
\definecolor{pagebackground}{rgb}{1,1,1}
\definecolor{pageforeground}{rgb}{0,0,0}
\colorlet{symbols}{blue!90!black}
\colorlet{connection}{red!30!black}
\colorlet{boxcolor}{blue!50!black}

\fi

\def\slash{\leavevmode\unskip\kern0.18em/\penalty\exhyphenpenalty\kern0.18em}
\def\dash{\leavevmode\unskip\kern0.18em--\penalty\exhyphenpenalty\kern0.18em}

\DeclareMathAlphabet{\mathbbm}{U}{bbm}{m}{n}

\DeclareFontFamily{U}{BOONDOX-calo}{\skewchar\font=45 }
\DeclareFontShape{U}{BOONDOX-calo}{m}{n}{
  <-> s*[1.05] BOONDOX-r-calo}{}
\DeclareFontShape{U}{BOONDOX-calo}{b}{n}{
  <-> s*[1.05] BOONDOX-b-calo}{}
\DeclareMathAlphabet{\mcb}{U}{BOONDOX-calo}{m}{n}
\SetMathAlphabet{\mcb}{bold}{U}{BOONDOX-calo}{b}{n}

\setlist{noitemsep,topsep=4pt,leftmargin=1.5em}

\DeclareMathAlphabet{\mathbbm}{U}{bbm}{m}{n}

\DeclareMathAlphabet{\mcb}{U}{BOONDOX-calo}{m}{n}
\SetMathAlphabet{\mcb}{bold}{U}{BOONDOX-calo}{b}{n}
\DeclareFontFamily{U}{mathx}{\hyphenchar\font45}
\DeclareFontShape{U}{mathx}{m}{n}{
      <5> <6> <7> <8> <9> <10>
      <10.95> <12> <14.4> <17.28> <20.74> <24.88>
      mathx10
      }{}
\DeclareSymbolFont{mathx}{U}{mathx}{m}{n}
\DeclareMathSymbol{\bigtimes}{1}{mathx}{"91}

\def\emptyset{{\centernot\ocircle}}

\newcommand{\z}{\mathsf{z}}
\newcommand{\rr}{\mathsf{r}}

\providecommand{\figures}{false}
{ \ifthenelse{\equal{\figures}{false}} {#1}{\[ {\rm Figure \ missing !} \]} }{}
\def\id{\mathrm{id}}

\def\proj{\mathbf{p}}

\tikzstyle{tinydots}=[dash pattern=on \pgflinewidth off \pgflinewidth]
\tikzstyle{superdense}=[dash pattern=on 4pt off 1pt]

\newcommand{\mcH}{\mathcal{H}}

\newcommand{\mcR}{\mathcal{R}}

\newcommand{\mcB}{\mathcal{B}}

\newcommand{\mcI}{\mathcal{I}}

\newcommand{\mcN}{\mathcal{N}}

\newcommand{\mcP}{\mathcal{P}}
\newcommand{\mcG}{\mathcal{G}}

\newcommand{\beq}{\begin{equation}}
\newcommand{\eeq}{\end{equation}}

\usepackage{empheq}

\newcommand{\mfL}{\mathfrak{L}}

\newcommand{\mfl}{\mathfrak{l}}

\def\${|\!|\!|}

\def\proj{\mathfrak{p}}

\newcounter{theorem}

\newtheorem{ex}[theorem]{Example}

\newenvironment{DIFnomarkup}{}{} 

\newcommand{\rmT}{{\rm T}}

\theorembodyfont{\rmfamily}

\newfont{\indic}{bbmss12}

\def\Nabla_#1{\nabla_{\!#1}}

\makeatletter
\pgfdeclareshape{crosscircle}
{
  \inheritsavedanchors[from=circle] 
  \inheritanchorborder[from=circle]
  \inheritanchor[from=circle]{north}
  \inheritanchor[from=circle]{north west}
  \inheritanchor[from=circle]{north east}
  \inheritanchor[from=circle]{center}
  \inheritanchor[from=circle]{west}
  \inheritanchor[from=circle]{east}
  \inheritanchor[from=circle]{mid}
  \inheritanchor[from=circle]{mid west}
  \inheritanchor[from=circle]{mid east}
  \inheritanchor[from=circle]{base}
  \inheritanchor[from=circle]{base west}
  \inheritanchor[from=circle]{base east}
  \inheritanchor[from=circle]{south}
  \inheritanchor[from=circle]{south west}
  \inheritanchor[from=circle]{south east}
  \inheritbackgroundpath[from=circle]
  \foregroundpath{
    \centerpoint%
    \pgf@xc=\pgf@x%
    \pgf@yc=\pgf@y%
    \pgfutil@tempdima=\radius%
    \pgfmathsetlength{\pgf@xb}{\pgfkeysvalueof{/pgf/outer xsep}}%
    \pgfmathsetlength{\pgf@yb}{\pgfkeysvalueof{/pgf/outer ysep}}%
    \ifdim\pgf@xb<\pgf@yb%
      \advance\pgfutil@tempdima by-\pgf@yb%
    \else%
      \advance\pgfutil@tempdima by-\pgf@xb%
    \fi%
    \pgfpathmoveto{\pgfpointadd{\pgfqpoint{\pgf@xc}{\pgf@yc}}{\pgfqpoint{-0.707107\pgfutil@tempdima}{-0.707107\pgfutil@tempdima}}}
    \pgfpathlineto{\pgfpointadd{\pgfqpoint{\pgf@xc}{\pgf@yc}}{\pgfqpoint{0.707107\pgfutil@tempdima}{0.707107\pgfutil@tempdima}}}
    \pgfpathmoveto{\pgfpointadd{\pgfqpoint{\pgf@xc}{\pgf@yc}}{\pgfqpoint{-0.707107\pgfutil@tempdima}{0.707107\pgfutil@tempdima}}}
    \pgfpathlineto{\pgfpointadd{\pgfqpoint{\pgf@xc}{\pgf@yc}}{\pgfqpoint{0.707107\pgfutil@tempdima}{-0.707107\pgfutil@tempdima}}}
  }
}
\makeatother

\def\symbol#1{\textcolor{symbols}{#1}}

\def\decorate#1#2{
        \ifnum#2>0
    		\foreach \count in {1,...,#2}{
	       	let
				\p1 = (sourcenode.center),
                \p2 = (sourcenode.east),
				\n1 = {\x2-\x1},
				\n2 = {1mm},
				\n3 = {(1.3+0.6*(\count-1))*\n1},
				\n4 = {0.7*\n1}
			in 
        		node[rectangle,fill=symbols,rotate=30,inner sep=0pt,minimum width=0.2*\n2,minimum height=\n2] at ($(sourcenode.center) + (\n3,\n4)$) {}
				}
		\fi
        \ifnum#1>0
    		\foreach \count in {1,...,#1}{
	       	let
				\p1 = (sourcenode.center),
                \p2 = (sourcenode.east),
				\n1 = {\x2-\x1},
				\n2 = {1mm},
				\n3 = {(1.3+0.6*(\count-1))*\n1},
				\n4 = {0.7*\n1}
			in 
        		node[rectangle,fill=symbols,rotate=-30,inner sep=0pt,minimum width=0.2*\n2,minimum height=\n2] at ($(sourcenode.center) + (-\n3,\n4)$) {}
				}
		\fi
}

\tikzset{
    dectriangle/.style 2 args={
        triangle,
        alias=sourcenode,
        append after command={\decorate{#1}{#2}}
    },
    dectriangle/.default={0}{0},
}

\tikzset{
	cross/.style={path picture={ 
  		\draw[symbols]
			(path picture bounding box.south east) -- (path picture bounding box.north west) (path picture bounding box.south west) -- (path picture bounding box.north east);
		}},
root/.style={circle,fill=green!50!black,inner sep=0pt, minimum size=1.2mm},
        dot/.style={circle,fill=pageforeground,inner sep=0pt, minimum size=1mm},
        dotred/.style={circle,fill=pageforeground!50!pagebackground,inner sep=0pt, minimum size=2mm},
        var/.style={circle,fill=pageforeground!10!pagebackground,draw=pageforeground,inner sep=0pt, minimum size=3mm},
        kernel/.style={semithick,shorten >=2pt,shorten <=2pt},
        kernels/.style={snake=zigzag,shorten >=2pt,shorten <=2pt,segment amplitude=1pt,segment length=4pt,line before snake=2pt,line after snake=5pt,},
        rho/.style={densely dashed,semithick,shorten >=2pt,shorten <=2pt},
           testfcn/.style={dotted,semithick,shorten >=2pt,shorten <=2pt},
        renorm/.style={shape=circle,fill=pagebackground,inner sep=1pt},
        labl/.style={shape=rectangle,fill=pagebackground,inner sep=1pt},
        xic/.style={very thin,circle,draw=symbols,fill=symbols,inner sep=0pt,minimum size=1.2mm},
        g/.style={very thin,rectangle,draw=symbols,fill=symbols!10!pagebackground,inner sep=0pt,minimum width=2.5mm,minimum height=1.2mm},
        xi/.style={very thin,circle,draw=symbols,fill=symbols!10!pagebackground,inner sep=0pt,minimum size=1.2mm},
	xies/.style={very thin,rectangle,fill=green!50!black!25,draw=symbols,inner sep=0pt,minimum size=1.1mm},
	xiesf/.style={very thin,rectangle,fill=green!50!black,draw=symbols,inner sep=0pt,minimum size=1.1mm},
        xix/.style={very thin,crosscircle,fill=symbols!10!pagebackground,draw=symbols,inner sep=0pt,minimum size=1.2mm},
        X/.style={very thin,cross,rectangle,fill=pagebackground,draw=symbols,inner sep=0pt,minimum size=1.2mm},
	xib/.style={thin,circle,fill=symbols!10!pagebackground,draw=symbols,inner sep=0pt,minimum size=1.6mm},
	xie/.style={thin,circle,fill=green!50!black,draw=symbols,inner sep=0pt,minimum size=1.6mm},
	xid/.style={thin,circle,fill=symbols,draw=symbols,inner sep=0pt,minimum size=1.6mm},
	xibx/.style={thin,crosscircle,fill=symbols!10!pagebackground,draw=symbols,inner sep=0pt,minimum size=1.6mm},
	kernels2/.style={very thick,draw=connection,segment length=12pt},
	keps/.style={thin,draw=symbols,->},
	kepspr/.style={thick,draw=connection,->},
	krho/.style={thin,draw=symbols,superdense,->},
	krhopr/.style={thick,draw=connection,superdense},
	triangle/.style = { regular polygon, regular polygon sides=3},
	not/.style={thin,circle,draw=connection,fill=connection,inner sep=0pt,minimum size=0.5mm},
	diff/.style = {very thin,draw=symbols,triangle,fill=red!50!black,inner sep=0pt,minimum size=1.6mm},
	diff1/.style = {very thin,dectriangle={1}{0},fill=red!50!black,draw=symbols,inner sep=0pt,minimum size=1.6mm},
	diff2/.style = {very thin,dectriangle={1}{1},fill=red!50!black,draw=symbols,inner sep=0pt,minimum size=1.6mm},
		diffmini/.style = {very thin,rectangle,fill=black,draw=black,inner sep=0pt,minimum size=0.75mm},
	 kernelsmod/.style={very thick,draw=connection,segment length=12pt},
	 rec/.style = {very thin,rectangle,fill=black,draw=black,inner sep=0pt,minimum size=2mm},
	cerc/.style={very thin,circle,draw=black,fill=symbols,inner sep=0pt,minimum size=2mm},
	stars/.style={very thin,star,star points=6,star point ratio=0.5, draw=black,fill=red,inner sep=0pt,minimum size=0.7mm},
	>=stealth,
        }
        \tikzset{
root/.style={circle,fill=black!50,inner sep=0pt, minimum size=3mm},
        circ/.style={circle,fill=white,draw=black,very thin,inner sep=.5pt, minimum size=1.2mm},
        round1/.style={fill=white,outer sep = 0,inner sep=2pt,rounded corners=1mm,draw,text=black,thin,minimum size=1.2mm},
          circ1/.style={circle,fill=red!10,draw=red,very thin,inner sep=.5pt, minimum size=1.2mm},
        rect/.style={fill=white,outer sep = 0,inner sep=2pt,rectangle,draw,text=black,thin,minimum size=1.2mm},
        rect1/.style={fill=white,outer sep = 0,inner sep=2pt,rectangle,draw,text=black,thin,minimum size=1.2mm},
        round2/.style={fill=red!10,outer sep = 0,inner sep=2pt,rounded corners=1mm,draw,text=black,thin,minimum size=1.2mm},
       round3/.style={fill=blue!10,outer sep = 0,inner sep=2pt,rounded corners=1mm,draw,text=black,thin,minimum size=1.2mm}, 
        rect2/.style={fill=black!10,outer sep = 0,inner sep=2pt,rectangle,draw,text=black,thin,minimum size=1.2mm},
        dot/.style={circle,fill=black,inner sep=0pt, minimum size=1.2mm},
        dotred/.style={circle,fill=black!50,inner sep=0pt, minimum size=2mm},
        var/.style={circle,fill=black!10,draw=black,inner sep=0pt, minimum size=3mm},
        kernel/.style={semithick,shorten >=2pt,shorten <=2pt},
         diag/.style={thin,shorten >=4pt,shorten <=4pt},
        kernel1/.style={thick},
        kernels/.style={snake=zigzag,shorten >=2pt,shorten <=2pt,segment amplitude=1pt,segment length=4pt,line before snake=2pt,line after snake=5pt,},
		kernels1/.style={snake=zigzag,segment amplitude=0.5pt,segment length=2pt},
		rho1/.style={densely dotted,semithick},
        rho/.style={densely dashed,semithick,shorten >=2pt,shorten <=2pt},
           testfcn/.style={dotted,semithick,shorten >=2pt,shorten <=2pt},
           visible/.style={draw, circle, fill, inner sep=0.25ex},
        renorm/.style={shape=circle,fill=white,inner sep=1pt},
        labl/.style={shape=rectangle,fill=white,inner sep=1pt},
        xic/.style={very thin,circle,fill=symbols,draw=black,inner sep=0pt,minimum size=1.2mm},
        xi/.style={very thin,circle,fill=blue!10,draw=black,inner sep=0pt,minimum size=1.2mm},
	xib/.style={very thin,circle,fill=blue!10,draw=black,inner sep=0pt,minimum size=1.6mm},
	xie/.style={very thin,circle,fill=green!50!black,draw=black,inner sep=0pt,minimum size=1mm},
	xid/.style={very thin,circle,fill=symbols,draw=black,inner sep=0pt,minimum size=1.6mm},
	edgetype/.style={very thin,circle,draw=black,inner sep=0pt,minimum size=5mm},
	nodetype/.style={very thick,circle,draw=black,inner sep=0pt,minimum size=5mm},
	kernels2/.style={very thick,draw=connection,segment length=12pt},
clean/.style={thin,circle,fill=black,inner sep=0pt,minimum size=1mm},	not/.style={thin,circle,fill=symbols,draw=connection,fill=connection,inner sep=0pt,minimum size=0.8mm},
	>=stealth,
        }

\makeatletter
\def\DeclareSymbol#1#2#3{%
	\expandafter\gdef\csname MH@symb@#1\endcsname{\tikzsetnextfilename{symbol#1}%
	\tikz[baseline=#2,scale=0.15,draw=symbols,line join=round]{#3}}%
	\expandafter\gdef\csname MH@symb@#1s\endcsname{\scalebox{0.75}{\tikzsetnextfilename{symbol#1}%
	\tikz[baseline=#2,scale=0.15,draw=symbols,line join=round]{#3}}}%
	\expandafter\gdef\csname MH@symb@#1ss\endcsname{\scalebox{0.65}{\tikzsetnextfilename{symbol#1}%
	\tikz[baseline=#2,scale=0.15,draw=symbols,line join=round]{#3}}}%
	}
\def\<#1>{\ifthenelse{\boolean{mmode}}{\mathchoice{\csname MH@symb@#1\endcsname}{\csname MH@symb@#1\endcsname}{\csname MH@symb@#1s\endcsname}{\csname MH@symb@#1ss\endcsname}}{\csname MH@symb@#1\endcsname}}
\makeatother

\DeclareSymbol{Xi22}{0.5}{\draw (0,0) node[xi] {} -- (-1,1) node[not] {} -- (0,2) node[xi] {};} 
\DeclareSymbol{Xi2}{-2}{\draw (-1,-0.25) node[xi] {} -- (0,1) node[xi] {};} 
\DeclareSymbol{Xi2b}{-2}{\draw (-1,-0.25) node[xic] {} -- (0,1) node[xic] {};} 
\DeclareSymbol{Xi2g}{-2}{\draw (-1,-0.25) node[xies] {} -- (0,1) node[xi] {};} 
\DeclareSymbol{Xi2g2}{-2}{\draw (-1,-0.25) node[xi] {} -- (0,1) node[xies] {};} 
\DeclareSymbol{cXi2}{-2}{\draw (0,-0.25) node[xi] {} -- (-1,1) node[xic] {};}
\DeclareSymbol{Xi3}{0}{\draw (0,0) node[xi] {} -- (-1,1) node[xi] {} -- (0,2) node[xi] {};}
\DeclareSymbol{XiIIXi}{0}{\draw (0,0) node[xi] {} -- (-1,1); \draw[kernels2] (-1,1) node[not] {} -- (0,2) node[xi] {};}

\DeclareSymbol{Xi4}{2}{\draw (0,0) node[xi] {} -- (-1,1) node[xi] {} -- (0,2) node[xi] {} -- (-1,3) node[xi] {};}
\DeclareSymbol{Xi4_1}{2}{\draw (0,0) node[xic] {} -- (-1,1) node[xic] {} -- (0,2) node[xi] {} -- (-1,3) node[xi] {};}
\DeclareSymbol{Xi4_2}{2}{\draw (0,0) node[xic] {} -- (-1,1) node[xi] {} -- (0,2) node[xi] {} -- (-1,3) node[xic] {};}
\DeclareSymbol{Xi2X}{-2}{\draw (0,-0.25) node[xi] {} -- (-1,1) node[xix] {};}
\DeclareSymbol{XXi2}{-2}{\draw (0,-0.25) node[xix] {} -- (-1,1) node[xi] {};}
\DeclareSymbol{IIXi}{0}{\draw (0,-0.25) node[not] {} -- (-1,1) node[xi] {} -- (0,2) node[xi] {};}
\DeclareSymbol{IXi^2}{-1}{\draw (-1,1) node[xi] {} -- (0,0) node[not] {} -- (1,1) node[xi] {};}
\DeclareSymbol{IIXi^2}{-4}{\draw (0,-1.5) node[not] {} -- (0,0);
\draw[kernels2] (-1,1) node[xi] {} -- (0,0) node[not] {} -- (1,1) node[xi] {};}
\DeclareSymbol{XiX}{-2.8}{\node[xibx] {};}
\DeclareSymbol{tauX}{-2.8}{ \node[X] {};}
\DeclareSymbol{Xi}{-2.8}{\node[xib] {};}

\DeclareSymbol{IXiX}{-1}{\draw (0,-0.25) node[not] {} -- (-1,1) node[xix] {};}
\DeclareSymbol{IXi3}{2}{\draw (0,-0.25) node[not] {} -- (-1,1) node[xi] {} -- (0,2) node[xi] {} -- (-1,3) node[xi] {};}
\DeclareSymbol{IXi}{-2}{\draw (0,-0.25) node[not] {} -- (-1,1) node[xi] {};}
\DeclareSymbol{XiI}{-2}{\draw (0,-0.25) node[xi] {} -- (-1,1) node[not] {};}

\DeclareSymbol{Xi4b}{0}{\draw(0,1.5) node[xi] {} -- (0,0); \draw (-1,1) node[xi] {} -- (0,0) node[xi] {} -- (1,1) node[xi] {};}
\DeclareSymbol{Xi4b'}{0}{\draw(0,1.5) node[xi] {} -- (0,-0.2); \draw (-1,1) node[xi] {} -- (0,-0.2) node[not] {} -- (1,1) node[xi] {};}
\DeclareSymbol{Xi4c}{0}{\draw (0,1) -- (0.8,2.2) node[xi] {};\draw (0,-0.25) node[xi] {} -- (0,1) node[xi] {} -- (-0.8,2.2) node[xi] {};}
\DeclareSymbol{Xi4d}{-4.5}{\draw (0,-1.5) node[not] {} -- (0,0); \draw (-1,1) node[xi] {} -- (0,0) node[xi] {} -- (1,1) node[xi] {};}
\DeclareSymbol{Xi4e}{0}{\draw (0,2) node[xi] {} -- (-1,1) node[xi] {} -- (0,0) node[xi] {} -- (1,1) node[xi] {};}
\DeclareSymbol{Xi4e'}{0}{\draw (0,2) node[xi] {} -- (-1,1) node[xi] {} -- (0,-0.2) node[not] {} -- (1,1) node[xi] {};}

\DeclareSymbol{Xitwo}
{0}{\draw[kernels2] (0,0) node[not] {} -- (-1,1) node[not] {}
-- (-2,2) node[not]{} -- (-3,3) node[xi]  {};
\draw[kernels2] (0,0) -- (1,1) node[xi] {};
\draw[kernels2] (-1,1) -- (0,2) node[xi] {};
\draw[kernels2] (-2,2) -- (-1,3) node[xi] {};}

\DeclareSymbol{IXitwo}
{0}{\draw (-.7,1.2) node[xi] {} -- (0,-0.2) -- (.7,1.2) node[xi] {};}
\DeclareSymbol{I1Xitwo}
{0}{\draw[kernels2] (0,0) node[not] {} -- (-1,1) node[xi] {};
\draw[kernels2] (0,0) -- (1,1) node[xi] {};}

\DeclareSymbol{I1Xitwobis}
{0}{\draw[kernels2] (0,0) node[not] {} -- (-1,1) node[xies] {};
\draw[kernels2] (0,0) -- (1,1) node[xies] {};}

\DeclareSymbol{I1Xitwog}
{0}{\draw[kernels2] (0,0) node[not] {} -- (-1,1) node[xies] {};
\draw[kernels2] (0,0) -- (1,1) node[xi] {};}

\DeclareSymbol{cI1Xitwo}
{0}{\draw[kernels2] (0,0) node[not] {} -- (-1,1) node[xic] {};
\draw[kernels2] (0,0) -- (1,1) node[xi] {};}

\DeclareSymbol{I1IXi3}{0}{\draw (0,0) node[xi] {} -- (-1,1) ; 
\draw[kernels2] (-1,1) node[not] {} -- (0,2) node[xi] {};
\draw[kernels2] (-1,1) node[not] {} -- (-2,2) node[xi] {};}

\DeclareSymbol{I1Xi3c}{-1}{\draw[kernels2](0,1.5) node[xi] {} -- (0,0) node[not] {}; \draw (-1,1) node[xi] {} -- (0,0) ; \draw[kernels2] (0,0) -- (1,1) node[xi] {};}

\DeclareSymbol{I1Xi3cbis}{-1}{\draw[kernels2](0,1.5) node[xies] {} -- (0,0) node[not] {}; \draw (-1,1) node[xies] {} -- (0,0) ; \draw[kernels2] (0,0) -- (1,1) node[xies] {};}

\DeclareSymbol{I1IXi3b}{0}{\draw[kernels2] (0,0) node[not] {} -- (-1,1) ; \draw[kernels2] (0,0)   -- (1,1) node[xi] {} ;
\draw (-1,1) node[xi] {} -- (0,2) node[xi] {};
}

\DeclareSymbol{I1IXi3c}{0}{\draw[kernels2] (0,0) node[not] {} -- (-1,1) ; \draw[kernels2] (0,0)   -- (1,1) node[xi] {} ;
\draw[kernels2] (-1,1) node[not] {} -- (0,2) node[xi] {};
\draw[kernels2] (-1,1) node[not] {} -- (-2,2) node[xi] {};}

\DeclareSymbol{I1IXi3cbis}{0}{\draw[kernels2] (0,0) node[not] {} -- (-1,1) ; \draw[kernels2] (0,0)   -- (1,1) node[xies] {} ;
\draw[kernels2] (-1,1) node[not] {} -- (0,2) node[xies] {};
\draw[kernels2] (-1,1) node[not] {} -- (-2,2) node[xies] {};}

\DeclareSymbol{I1Xi}{0}{\draw[kernels2] (0,0) node[not] {} -- (-1,1)  node[xi] {} ;}

\DeclareSymbol{I1Xi4a}{2}{\draw[kernels2] (0,0) node[not] {} -- (-1,1) ; \draw[kernels2] (0,0) node[not] {} -- (1,1) node[xi] {} ;
\draw (-1,1) node[xi] {} -- (0,2) node[xi] {} -- (-1,3) node[xi] {};}

\DeclareSymbol{cI1Xi4a}{2}{\draw[kernels2] (0,0) node[not] {} -- (-1,1) ; \draw[kernels2] (0,0) node[not] {} -- (1,1) node[xic] {} ;
\draw (-1,1) node[xic] {} -- (0,2) node[xi] {} -- (-1,3) node[xi] {};}

\DeclareSymbol{I1Xi4b}{2}{\draw (0,0) node[xi] {} -- (-1,1) node[xi] {} -- (0,2) ; \draw[kernels2] (0,2) node[not] {} -- (-1,3) node[xi] {};\draw[kernels2] (0,2)  -- (1,3) node[xi] {};
}

\DeclareSymbol{cI1Xi4b}{2}{\draw (0,0) node[xic] {} -- (-1,1) node[xic] {} -- (0,2) ; \draw[kernels2] (0,2) node[not] {} -- (-1,3) node[xi] {};\draw[kernels2] (0,2)  -- (1,3) node[xi] {};
}

\DeclareSymbol{I1Xi4c}{2}{\draw (0,0) node[xi] {} -- (-1,1) node[not] {}; \draw[kernels2] (-1,1) -- (0,2) ; 
\draw[kernels2] (-1,1) -- (-2,2) node[xi] {} ;
\draw (0,2) node[xi] {} -- (-1,3) node[xi] {};}

\DeclareSymbol{cI1Xi4c}{2}{\draw (0,0) node[xic] {} -- (-1,1) node[not] {}; \draw[kernels2] (-1,1) -- (0,2) ; 
\draw[kernels2] (-1,1) -- (-2,2) node[xic] {} ;
\draw (0,2) node[xi] {} -- (-1,3) node[xi] {};}

\DeclareSymbol{I1Xi4ab}{2}{\draw[kernels2] (0,0) node[not] {} -- (-1,1) ; \draw[kernels2] (0,0) node[not] {} -- (1,1) node[xi] {};\draw (-1,1) node[xi] {} -- (0,2) ; \draw[kernels2] (0,2) node[not] {} -- (-1,3) node[xi] {};\draw[kernels2] (0,2)  -- (1,3) node[xi] {}; }

\DeclareSymbol{cI1Xi4ab}{2}{\draw[kernels2] (0,0) node[not] {} -- (-1,1) ; \draw[kernels2] (0,0) node[not] {} -- (1,1) node[xic] {};\draw (-1,1) node[xic] {} -- (0,2) ; \draw[kernels2] (0,2) node[not] {} -- (-1,3) node[xi] {};\draw[kernels2] (0,2)  -- (1,3) node[xi] {}; }

\DeclareSymbol{I1Xi4bc}{2}{\draw (0,0) node[xi] {} -- (-1,1) node[not] {}; \draw[kernels2] (-1,1) -- (0,2) ; 
\draw[kernels2] (-1,1) -- (-2,2) node[xi] {} ; \draw[kernels2] (0,2) node[not] {} -- (-1,3) node[xi] {};\draw[kernels2] (0,2)  -- (1,3) node[xi] {};
}

\DeclareSymbol{cI1Xi4bc}{2}{\draw (0,0) node[xic] {} -- (-1,1) node[not] {}; \draw[kernels2] (-1,1) -- (0,2) ; 
\draw[kernels2] (-1,1) -- (-2,2) node[xic] {} ; \draw[kernels2] (0,2) node[not] {} -- (-1,3) node[xi] {};\draw[kernels2] (0,2)  -- (1,3) node[xi] {};
}

\DeclareSymbol{I1Xi4abcc1}{2}{\draw[kernels2] (0,0) node[not] {} -- (-1,1) node[not] {}
-- (-2,2) node[not]{} -- (-3,3) node[xic]  {};
\draw[kernels2] (0,0) -- (1,1) node[xic] {};
\draw[kernels2] (-1,1) -- (0,2) node[xi] {};
\draw[kernels2] (-2,2) -- (-1,3) node[xi] {};
}

\DeclareSymbol{I1Xi4abcc1b}{2}{\draw[kernels2] (0,0) node[not] {} -- (-1,1) node[not] {}
-- (-2,2) node[not]{} -- (-3,3) node[xi]  {};
\draw[kernels2] (0,0) -- (1,1) node[xic] {};
\draw[kernels2] (-1,1) -- (0,2) node[xic] {};
\draw[kernels2] (-2,2) -- (-1,3) node[xi] {};
}

\DeclareSymbol{I1Xi4abcc2}{2}{\draw[kernels2] (0,0) node[not] {} -- (-1,1) node[not] {}
-- (-2,2) node[not]{} -- (-3,3) node[xic]  {};
\draw[kernels2] (0,0) -- (1,1) node[xi] {};
\draw[kernels2] (-1,1) -- (0,2) node[xi] {};
\draw[kernels2] (-2,2) -- (-1,3) node[xic] {};
}

\DeclareSymbol{I1Xi4ac}{2}{\draw[kernels2] (0,0) node[not] {} -- (-1,1) ; \draw[kernels2] (0,0) node[not] {} -- (1,1) node[xi] {}; 
\draw[kernels2] (-1,1) node[not] {} -- (0,2) ; 
\draw[kernels2] (-1,1) -- (-2,2) node[xi] {} ;
\draw (0,2) node[xi] {} -- (-1,3) node[xi] {};}

\DeclareSymbol{cI1Xi4ac}{2}{\draw[kernels2] (0,0) node[not] {} -- (-1,1) ; \draw[kernels2] (0,0) node[not] {} -- (1,1) node[xic] {}; 
\draw[kernels2] (-1,1) node[not] {} -- (0,2) ; 
\draw[kernels2] (-1,1) -- (-2,2) node[xic] {} ;
\draw (0,2) node[xi] {} -- (-1,3) node[xi] {};}

\DeclareSymbol{I1Xi4acc1}{2}{\draw[kernels2] (0,0) node[not] {} -- (-1,1) ; \draw[kernels2] (0,0) node[not] {} -- (1,1) node[xic] {}; 
\draw[kernels2] (-1,1) node[not] {} -- (0,2) ; 
\draw[kernels2] (-1,1) -- (-2,2) node[xi] {} ;
\draw (0,2) node[xic] {} -- (-1,3) node[xi] {};}

\DeclareSymbol{I1Xi4acc2}{2}{\draw[kernels2] (0,0) node[not] {} -- (-1,1) ; \draw[kernels2] (0,0) node[not] {} -- (1,1) node[xic] {}; 
\draw[kernels2] (-1,1) node[not] {} -- (0,2) ; 
\draw[kernels2] (-1,1) -- (-2,2) node[xi] {} ;
\draw (0,2) node[xi] {} -- (-1,3) node[xic] {};}

\DeclareSymbol{2I1Xi4}{2}{\draw[kernels2] (0,0) node[not] {} -- (-1,1) node[not] {};
\draw[kernels2] (0,0) -- (1,1) node[not] {};
\draw[kernels2] (-1,1) -- (-1.5,2.5) node[xi] {};
\draw[kernels2] (-1,1) -- (-0.5,2.5) node[xi] {};
\draw[kernels2] (1,1) -- (0.5,2.5) node[xi] {};
\draw[kernels2] (1,1) -- (1.5,2.5) node[xi] {};
}

\DeclareSymbol{2I1Xi4dis}{2}{\draw[kernels2] (0,0) node[not] {} -- (-1,1) node[not] {};
\draw[kernels2] (0,0) -- (1,1) node[not] {};
\draw[kernels2] (-1,1) -- (-1.5,2.5) node[xies] {};
\draw[kernels2] (-1,1) -- (-0.5,2.5) node[xies] {};
\draw[kernels2] (1,1) -- (0.5,2.5) node[xies] {};
\draw[kernels2] (1,1) -- (1.5,2.5) node[xies] {};
}

\DeclareSymbol{2I1Xi4c1}{2}{\draw[kernels2] (0,0) node[not] {} -- (-1,1) node[not] {};
\draw[kernels2] (0,0) -- (1,1) node[not] {};
\draw[kernels2] (-1,1) -- (-1.5,2.5) node[xic] {};
\draw[kernels2] (-1,1) -- (-0.5,2.5) node[xi] {};
\draw[kernels2] (1,1) -- (0.5,2.5) node[xic] {};
\draw[kernels2] (1,1) -- (1.5,2.5) node[xi] {};
}

\DeclareSymbol{2I1Xi4c2}{2}{\draw[kernels2] (0,0) node[not] {} -- (-1,1) node[not] {};
\draw[kernels2] (0,0) -- (1,1) node[not] {};
\draw[kernels2] (-1,1) -- (-1.5,2.5) node[xic] {};
\draw[kernels2] (-1,1) -- (-0.5,2.5) node[xic] {};
\draw[kernels2] (1,1) -- (0.5,2.5) node[xi] {};
\draw[kernels2] (1,1) -- (1.5,2.5) node[xi] {};
}

\DeclareSymbol{2I1Xi4b}{2}{\draw[kernels2] (0,0) node[not] {} -- (-1,1) ;
\draw[kernels2] (0,0) -- (1,1);
\draw (-1,1) node[xi] {} -- (-1,2.5) node[xi] {};
\draw (1,1)  node[xi] {} -- (1,2.5) node[xi] {};
}

\DeclareSymbol{2I1Xi4bb}{2}{\draw[kernels2] (0,0) node[not] {} -- (-1,1) ;
\draw[kernels2] (0,0) -- (1,1);
\draw (-1,1) node[xi] {} -- (-1,2.5) node[xiesf] {};
\draw (1,1)  node[xi] {} -- (1,2.5) node[xic] {};
}

\DeclareSymbol{2I1Xi4c}{2}{\draw[kernels2] (0,0) node[not] {} -- (-1,1);
\draw[kernels2] (0,0) -- (1,1) node[not] {};
\draw (-1,1)  node[xi] {} -- (-1,2.5) node[xi] {};
\draw[kernels2] (1,1) -- (0.4,2.5) node[xi] {};
\draw[kernels2] (1,1) -- (1.6,2.5) node[xi] {};
}

\DeclareSymbol{2I1Xi4cc1}{2}{\draw[kernels2] (0,0) node[not] {} -- (-1,1);
\draw[kernels2] (0,0) -- (1,1) node[not] {};
\draw (-1,1)  node[xic] {} -- (-1,2.5) node[xi] {};
\draw[kernels2] (1,1) -- (0.4,2.5) node[xic] {};
\draw[kernels2] (1,1) -- (1.6,2.5) node[xi] {};
}

\DeclareSymbol{2I1Xi4cc2}{2}{\draw[kernels2] (0,0) node[not] {} -- (-1,1);
\draw[kernels2] (0,0) -- (1,1) node[not] {};
\draw (-1,1)  node[xic] {} -- (-1,2.5) node[xic] {};
\draw[kernels2] (1,1) -- (0.4,2.5) node[xi] {};
\draw[kernels2] (1,1) -- (1.6,2.5) node[xi] {};
}

\DeclareSymbol{Xi4ba}{0}{\draw(-0.5,1.5) node[xi] {} -- (0,0); \draw (-1.5,1) node[xi] {} -- (0,0) node[not] {}; \draw[kernels2] (0,0) -- (1.5,1) node[xi] {};
\draw[kernels2] (0,0) -- (0.5,1.5) node[xi] {} ;}

\DeclareSymbol{Xi4badis}{0}{\draw(-0.5,1.5) node[xies] {} -- (0,0); \draw (-1.5,1) node[xies] {} -- (0,0) node[not] {}; \draw[kernels2] (0,0) -- (1.5,1) node[xies] {};
\draw[kernels2] (0,0) -- (0.5,1.5) node[xies] {} ;}

\DeclareSymbol{Xi4ba1}{0}{\draw(-0.5,1.5) node[xi] {} -- (0,0); \draw (-1.5,1) node[xi] {} -- (0,0) node[not] {}; \draw[kernels2] (0,0) -- (1.5,1) node[xic] {};
\draw[kernels2] (0,0) -- (0.5,1.5) node[xic] {} ;}

\DeclareSymbol{Xi4ba1b}{0}{\draw(-0.5,1.5) node[xic] {} -- (0,0); \draw (-1.5,1) node[xic] {} -- (0,0) node[not] {}; \draw[kernels2] (0,0) -- (1.5,1) node[xi] {};
\draw[kernels2] (0,0) -- (0.5,1.5) node[xi] {} ;}

\DeclareSymbol{Xi4ba1bdiff}{0}{\draw(-0.5,1.5) node[xic] {} -- (0,0); \draw (-1.5,1) node[xic] {} -- (0,0) node[not] {}; \draw (0,0) -- (1.5,1) node[xi] {};
\draw (0,0) -- (0.5,1.5) node[xi] {};
\draw(0,0) node[diff] {};}

\DeclareSymbol{Xi4ba1bb}{0}{\draw(-0.5,1.5) node[xic] {} -- (0,0); \draw (-1.5,1) node[xiesf] {} -- (0,0) node[not] {}; \draw[kernels2] (0,0) -- (1.5,1) node[xi] {};
\draw[kernels2] (0,0) -- (0.5,1.5) node[xi] {} ;}

\DeclareSymbol{Xi4ba2}{0}{\draw(-0.5,1.5) node[xi] {} -- (0,0); \draw (-1.5,1) node[xic] {} -- (0,0) node[not] {}; \draw[kernels2] (0,0) -- (1.5,1) node[xi] {};
\draw[kernels2] (0,0) -- (0.5,1.5) node[xic] {} ;}

\DeclareSymbol{Xi4ba2b}{0}{\draw(-0.5,1.5) node[xi] {} -- (0,0); \draw (-1.5,1) node[xic] {} -- (0,0) node[not] {}; \draw[kernels2] (0,0) -- (1.5,1) node[xi] {};
\draw[kernels2] (0,0) -- (0.5,1.5) node[xiesf] {} ;}

\DeclareSymbol{Xi4ca}{0}{\draw (0,1) -- (-1,2.2) node[xi] {};\draw (0,-0.25) node[xi] {} -- (0,1) ; \draw[kernels2] (0,1) node[not] {} -- (1,2.2) node[xi] {};
\draw[kernels2] (0,1) {} -- (0,2.7) node[xi] {};
}

\DeclareSymbol{Xi4cb}{0}{\draw (-1,1) -- (-2,2) node[xi] {};\draw[kernels2] (0,0)  -- (-1,1) node[xi] {} ; \draw[kernels2] (0,0) node[not] {} -- (1,1) node[xi] {} ; 
\draw (-1,1) node[xi] {} -- (0,2) node[xi] {};}

\DeclareSymbol{Xi4cbb}{0}{\draw (-1,1) -- (-2,2) node[xiesf] {};\draw[kernels2] (0,0)  -- (-1,1) node[xi] {} ; \draw[kernels2] (0,0) node[not] {} -- (1,1) node[xi] {} ; 
\draw (-1,1) node[xi] {} -- (0,2) node[xic] {};}

\DeclareSymbol{Xi4cbc1}{0}{\draw (-1,1) -- (-2,2) node[xic] {};\draw[kernels2] (0,0)  -- (-1,1) node[xic] {} ; \draw[kernels2] (0,0) node[not] {} -- (1,1) node[xi] {} ; 
\draw (-1,1) node[xic] {} -- (0,2) node[xi] {};}

\DeclareSymbol{Xi4cbc2}{0}{\draw (-1,1) -- (-2,2) node[xi] {};\draw[kernels2] (0,0)  -- (-1,1) node[xi] {} ; \draw[kernels2] (0,0) node[not] {} -- (1,1) node[xic] {} ; 
\draw (-1,1) node[xic] {} -- (0,2) node[xi] {};}

\DeclareSymbol{Xi4cab}{0}{\draw (-1,1) -- (-2,2) node[xi] {};\draw[kernels2] (0,0)  -- (-1,1); \draw[kernels2] (0,0) node[not] {} -- (1,1) node[xi] {} ; 
\draw[kernels2] (-1,1)  {} -- (0,2) node[xi] {};
\draw[kernels2] (-1,1) node[not] {} -- (-1,2.5) node[xi] {};
}

\DeclareSymbol{Xi4cabdis}{0}{\draw (-1,1) -- (-2,2) node[xies] {};\draw[kernels2] (0,0)  -- (-1,1); \draw[kernels2] (0,0) node[not] {} -- (1,1) node[xies] {} ; 
\draw[kernels2] (-1,1)  {} -- (0,2) node[xies] {};
\draw[kernels2] (-1,1) node[not] {} -- (-1,2.5) node[xies] {};
}

\DeclareSymbol{Xi4cabc1}{0}{\draw (-1,1) -- (-2,2) node[xi] {};\draw[kernels2] (0,0)  -- (-1,1); \draw[kernels2] (0,0) node[not] {} -- (1,1) node[xic] {} ; 
\draw[kernels2] (-1,1)  {} -- (0,2) node[xic] {};
\draw[kernels2] (-1,1) node[not] {} -- (-1,2.5) node[xi] {};
}

\DeclareSymbol{Xi4cabc2}{0}{\draw (-1,1) -- (-2,2) node[xic] {};\draw[kernels2] (0,0)  -- (-1,1); \draw[kernels2] (0,0) node[not] {} -- (1,1) node[xic] {} ; 
\draw[kernels2] (-1,1)  {} -- (0,2) node[xi] {};
\draw[kernels2] (-1,1) node[not] {} -- (-1,2.5) node[xi] {};
}

\DeclareSymbol{Xi4ea}{1.5}{\draw (-1,2.5) node[xi] {} -- (-1,1) node[xi] {} -- (0,0); 
 \draw[kernels2] (0,0)  -- (1,1) node[xi] {};
\draw[kernels2] (0,0) node[not] {} -- (0,1.5) node[xi] {}; }

\DeclareSymbol{Xi4eac1}{1.5}{\draw (-1,2.5) node[xic] {} -- (-1,1) node[xi] {} -- (0,0); 
 \draw[kernels2] (0,0)  -- (1,1) node[xic] {};
\draw[kernels2] (0,0) node[not] {} -- (0,1.5) node[xi] {}; }

\DeclareSymbol{Xi4eac1b}{1.5}{\draw (-1,2.5) node[xic] {} -- (-1,1) node[xi] {} -- (0,0); 
 \draw[kernels2] (0,0)  -- (1,1) node[xiesf] {};
\draw[kernels2] (0,0) node[not] {} -- (0,1.5) node[xi] {}; }

\DeclareSymbol{Xi4eac2}{1.5}{\draw (-1,2.5) node[xic] {} -- (-1,1) node[xic] {} -- (0,0); 
 \draw[kernels2] (0,0)  -- (1,1) node[xi] {};
\draw[kernels2] (0,0) node[not] {} -- (0,1.5) node[xi] {}; }

\DeclareSymbol{Xi4eact1}{1.5}{\draw (-1,2.5) node[xic] {} -- (-1,1) node[xi] {} -- (0,0); 
 \draw (0,0)  -- (1,1) node[xic] {};
\draw[rho] (0,0) node[not] {} -- (0,1.5) node[xi] {}; }

\DeclareSymbol{Xi4eact2}{1.5}{\draw[rho] (-1,2.5) node[xic] {} -- (-1,1) node[xi] {} -- (0,0); 
 \draw (0,0)  -- (1,1) node[xic] {};
\draw (0,0) node[not] {} -- (0,1.5) node[xi] {}; }

\DeclareSymbol{Xi4eabis}{1.5}{\draw (-1,2.5) node[xi] {} -- (-1,1) ; \draw[kernels2] (-1,1) node[xi] {} -- (0,0); 
 \draw (0,0)  -- (1,1) node[xi] {};
\draw[kernels2] (0,0) node[not] {} -- (0,1.5) node[xi] {}; }

\DeclareSymbol{Xi4eabisc1}{1.5}{\draw (-1,2.5) node[xic] {} -- (-1,1) ; \draw[kernels2] (-1,1) node[xi] {} -- (0,0); 
 \draw (0,0)  -- (1,1) node[xi] {};
\draw[kernels2] (0,0) node[not] {} -- (0,1.5) node[xic] {}; }

\DeclareSymbol{Xi4eabisc1b}{1.5}{\draw (-1,2.5) node[xic] {} -- (-1,1) ; \draw[kernels2] (-1,1) node[xi] {} -- (0,0); 
 \draw (0,0)  -- (1,1) node[xi] {};
\draw[kernels2] (0,0) node[not] {} -- (0,1.5) node[xiesf] {}; }

\DeclareSymbol{Xi4eabisc1bis}{1.5}{\draw (-1,2.5) node[xi] {} -- (-1,1) ; \draw[kernels2] (-1,1) node[xi] {} -- (0,0); 
 \draw (0,0)  -- (1,1) node[xi] {};
\draw[kernels2] (0,0) node[not] {} -- (0,1.5) node[xi] {};
\draw (-2,1) node[] {\tiny{$i$}};
\draw (-2,2.5) node[] {\tiny{$\ell$}};
\draw (2,1) node[] {\tiny{$k$}};
\draw (0,2.5) node[] {\tiny{$j$}};
 }

\DeclareSymbol{Xi4eabisc1tris}{1.5}{\draw (-1,2.5) node[xi] {} -- (-1,1) ; \draw[kernels2] (-1,1) node[xi] {} -- (0,0); 
 \draw (0,0)  -- (1,1) node[xi] {};
\draw[kernels2] (0,0) node[not] {} -- (0,1.5) node[xi] {};
\draw (-2,1) node[] {\tiny{i}};
\draw (-2,2.5) node[] {\tiny{j}};
\draw (2,1) node[] {\tiny{j}};
\draw (0,2.5) node[] {\tiny{i}};
 }

\DeclareSymbol{Xi4eabisc1quater}{1.5}{\draw (-1,2.5) node[xic] {} -- (-1,1) ; \draw[kernels2] (-1,1) node[xi] {} -- (0,0); 
 \draw (0,0)  -- (1,1) node[xic] {};
\draw[kernels2] (0,0) node[not] {} -- (0,1.5) node[xi] {};
 }

\DeclareSymbol{Xi4eabisc2}{1.5}{\draw (-1,2.5) node[xic] {} -- (-1,1) ; \draw[kernels2] (-1,1) node[xi] {} -- (0,0); 
 \draw (0,0)  -- (1,1) node[xic] {};
\draw[kernels2] (0,0) node[not] {} -- (0,1.5) node[xi] {}; }

\DeclareSymbol{Xi4eabisc2l}{1.5}{\draw (-1,2.5) node[xiesf] {} -- (-1,1) ; \draw[kernels2] (-1,1) node[xi] {} -- (0,0); 
 \draw (0,0)  -- (1,1) node[xic] {};
\draw[kernels2] (0,0) node[not] {} -- (0,1.5) node[xi] {}; }

\DeclareSymbol{Xi4eabisc2r}{1.5}{\draw (-1,2.5) node[xic] {} -- (-1,1) ; \draw[kernels2] (-1,1) node[xi] {} -- (0,0); 
 \draw (0,0)  -- (1,1) node[xiesf] {};
\draw[kernels2] (0,0) node[not] {} -- (0,1.5) node[xi] {}; }

\DeclareSymbol{Xi4eabisc3}{1.5}{\draw (-1,2.5) node[xic] {} -- (-1,1) ; \draw[kernels2] (-1,1) node[xic] {} -- (0,0); 
 \draw (0,0)  -- (1,1) node[xi] {};
\draw[kernels2] (0,0) node[not] {} -- (0,1.5) node[xi] {}; }

\DeclareSymbol{Xi4eb}{0}{
\draw[kernels2] (0,2) node[xi] {} -- (-1,1) ; \draw[kernels2] (-2,2)  node[xi] {} -- (-1,1) ; \draw (-1,1)  node[not] {} -- (0,0); 
 \draw (0,0) node[xi] {}  -- (1,1) node[xi] {};
}

\DeclareSymbol{Xi4eab}{1.5}{\draw[kernels2] (-1,2.5) node[xi] {} -- (-1,1) ; \draw[kernels2] (-2,2)  node[xi] {} -- (-1,1) ; \draw (-1,1)  node[not] {} -- (0,0); 
 \draw[kernels2] (0,0)  -- (1,1) node[xi] {};
\draw[kernels2] (0,0) node[not] {} -- (0,1.5) node[xi] {}; 
}

\DeclareSymbol{Xi4eabdis}{1.5}{\draw[kernels2] (-1,2.5) node[xies] {} -- (-1,1) ; \draw[kernels2] (-2,2)  node[xies] {} -- (-1,1) ; \draw (-1,1)  node[not] {} -- (0,0); 
 \draw[kernels2] (0,0)  -- (1,1) node[xies] {};
\draw[kernels2] (0,0) node[not] {} -- (0,1.5) node[xies] {}; 
}

\DeclareSymbol{Xi4eabc1}{1.5}{\draw[kernels2] (-1,2.5) node[xic] {} -- (-1,1) ; \draw[kernels2] (-2,2)  node[xi] {} -- (-1,1) ; \draw (-1,1)  node[not] {} -- (0,0); 
 \draw[kernels2] (0,0)  -- (1,1) node[xic] {};
\draw[kernels2] (0,0) node[not] {} -- (0,1.5) node[xi] {}; 
}

\DeclareSymbol{Xi4eabc2}{1.5}{\draw[kernels2] (-1,2.5) node[xi] {} -- (-1,1) ; \draw[kernels2] (-2,2)  node[xi] {} -- (-1,1) ; \draw (-1,1)  node[not] {} -- (0,0); 
 \draw[kernels2] (0,0)  -- (1,1) node[xic] {};
\draw[kernels2] (0,0) node[not] {} -- (0,1.5) node[xic] {}; 
}

\DeclareSymbol{Xi4eabbis}{1.5}{\draw[kernels2] (-1,2.5) node[xi] {} -- (-1,1) ; \draw[kernels2] (-2,2)  node[xi] {} -- (-1,1) ; \draw[kernels2] (-1,1)  node[not] {} -- (0,0); 
 \draw (0,0)  -- (1,1) node[xi] {};
\draw[kernels2] (0,0) node[not] {} -- (0,1.5) node[xi] {}; 
}

\DeclareSymbol{Xi4eabbisc1}{1.5}{\draw[kernels2] (-1,2.5) node[xic] {} -- (-1,1) ; \draw[kernels2] (-2,2)  node[xi] {} -- (-1,1) ; \draw[kernels2] (-1,1)  node[not] {} -- (0,0); 
 \draw (0,0)  -- (1,1) node[xic] {};
\draw[kernels2] (0,0) node[not] {} -- (0,1.5) node[xi] {}; 
}

\DeclareSymbol{Xi4eabbisc1perm}{1.5}{\draw[kernels2] (-1,2.5) node[xi] {} -- (-1,1) ; \draw[kernels2] (-2,2)  node[xic] {} -- (-1,1) ; \draw[kernels2] (-1,1)  node[not] {} -- (0,0); 
 \draw (0,0)  -- (1,1) node[xic] {};
\draw[kernels2] (0,0) node[not] {} -- (0,1.5) node[xi] {}; 
}

\DeclareSymbol{Xi4eabbisc2}{1.5}{\draw[kernels2] (-1,2.5) node[xi] {} -- (-1,1) ; \draw[kernels2] (-2,2)  node[xi] {} -- (-1,1) ; \draw[kernels2] (-1,1)  node[not] {} -- (0,0); 
 \draw (0,0)  -- (1,1) node[xic] {};
\draw[kernels2] (0,0) node[not] {} -- (0,1.5) node[xic] {}; 
}

\DeclareSymbol{Xi2cbis}{0}{\draw[kernels2] (0,1) -- (0.8,2.2) node[xi] {};\draw[kernels2] (0,-0.25) node[not] {} -- (0,1); \draw[kernels2] (0,1) node[not] {} -- (-0.8,2.2) node[xi] {};}

\DeclareSymbol{Xi2cbis1}{0}{\draw (0,1) -- (-0.8,2.2) node[xi] {};\draw[kernels2] (0,-0.25) node[not] {} -- (0,1) node[xi] {}; }

\DeclareSymbol{Xi2Xbis}{-2}{\draw[kernels2] (0,-0.25)  -- (-1,1) ; \draw (-1,1) node[xix] {};
\draw[kernels2] (0,-0.25) node[not] {} -- (1,1) node[xi] {};}

\DeclareSymbol{XXi2bis}{-2}{\draw[kernels2] (0,-0.25) -- (-1,1) node[xi] {};
\draw[kernels2] (0,-0.25) node[X] {} -- (1,1) node[xi] {};}

\DeclareSymbol{I1XiIXi}{0}{\draw[kernels2] (0,-0.25) -- (1,1) node[xi] {};
\draw (0,-0.25) node[not] {} -- (-1,1) node[xi] {};}

\DeclareSymbol{I1XiIXib}{0}{\draw  (0,-0.25) node[xi] {} -- (0,1) node[not] {};
\draw[kernels2] (0,1) -- (0,2.25) ; \draw (0,2.25) node[xi]{}; }

\DeclareSymbol{I1XiIXic}{0}{
\draw[kernels2] (0,0) -- (1,1) node[xi] {} ; 
\draw[kernels2] (0,0) node[not] {}  -- (-1,1) node[not] {} -- (0,2) node[xi] {};
}

\DeclareSymbol{thin}{1.4}{\draw[pagebackground] (-0.3,0) -- (0.3,0); \draw  (0,0) -- (0,2);}
\DeclareSymbol{thin2}{1.4}{\draw[pagebackground] (-0.3,0) -- (0.3,0); \draw[tinydots]  (0,0) -- (0,2);}

\DeclareSymbol{thick}{1.4}{\draw[pagebackground] (-0.3,0) -- (0.3,0); \draw[kernels2]  (0,0) -- (0,2);}

\DeclareSymbol{thick2}{1.4}{\draw[pagebackground] (-0.3,0) -- (0.3,0); \draw[kernels2,tinydots]  (0,0) -- (0,2);}

\DeclareSymbol{Xi4ind}{2}{\draw (0,0) node[xi,label={[label distance=-0.2em]right: \scriptsize  $ i $}]  { } -- (-1,1) node[xi,label={[label distance=-0.2em]left: \scriptsize  $ j $}] {} -- (0,2) node[xi,label={[label distance=-0.2em]right: \scriptsize  $ k $}] {} -- (-1,3) node[xi,label={[label distance=-0.2em]left: \scriptsize  $ \ell $}] {};}

\DeclareSymbol{Xi4c1}{2}{\draw (0,0) node[xic] {} -- (-1,1) node[xi] {} -- (0,2) node[xic] {} -- (-1,3) node[xi] {};} 
\DeclareSymbol{IXi2ex}{0}{\draw (0,-0.25) node[xie] {} -- (-1,1) node[xi] {} ; \draw (0,-0.25)-- (1,1) node[xi] {};}
\DeclareSymbol{IXi2ex1}{0}{\draw (0,-0.25) node[xie] {} -- (-1,1) node[xi] {} -- (0,2) node[xi] {};}

\DeclareSymbol{Xi4b1}{0}{\draw(0,1.5) node[xic] {} -- (0,0); \draw (-1,1) node[xic] {} -- (0,0) node[xi] {} -- (1,1) node[xi] {};}

\DeclareSymbol{Xi4ec1}{0}{\draw (0,2) node[xi] {} -- (-1,1) node[xic] {} -- (0,0) node[xic] {} -- (1,1) node[xi] {};}
\DeclareSymbol{Xi4ec2}{0}{\draw (0,2) node[xic] {} -- (-1,1) node[xi] {} -- (0,0) node[xic] {} -- (1,1) node[xi] {};}
\DeclareSymbol{Xi4ec3}{0}{\draw (0,2) node[xic] {} -- (-1,1) node[xic] {} -- (0,0) node[xi] {} -- (1,1) node[xi] {};}

\DeclareSymbol{I1Xi4ac1}{2}{\draw[kernels2] (0,0) node[not] {} -- (-1,1) ; \draw[kernels2] (0,0) node[not] {} -- (1,1) node[xic] {} ;
\draw (-1,1) node[xi] {} -- (0,2) node[xic] {} -- (-1,3) node[xi] {};}

\DeclareSymbol{I1Xi4ac2}{2}{\draw[kernels2] (0,0) node[not] {} -- (-1,1) ; \draw[kernels2] (0,0) node[not] {} -- (1,1) node[xic] {} ;
\draw (-1,1) node[xi] {} -- (0,2) node[xi] {} -- (-1,3) node[xic] {};}

\DeclareSymbol{I1Xi4bp}{2}{\draw (0,0) node[not] {} -- (-1,1) node[xi] {} -- (0,2) ; \draw[kernels2] (0,2) node[not] {} -- (-1,3) node[xi] {};\draw[kernels2] (0,2)  -- (1,3) node[xi] {};
}

\DeclareSymbol{I1Xi4bc1}{2}{\draw (0,0) node[xic] {} -- (-1,1) node[xi] {} -- (0,2) ; \draw[kernels2] (0,2) node[not] {} -- (-1,3) node[xi] {};\draw[kernels2] (0,2)  -- (1,3) node[xic] {};
}

\DeclareSymbol{I1Xi4bc2}{2}{\draw (0,0) node[xic] {} -- (-1,1) node[xi] {} -- (0,2) ; \draw[kernels2] (0,2) node[not] {} -- (-1,3) node[xic] {};\draw[kernels2] (0,2)  -- (1,3) node[xi] {};
}

\DeclareSymbol{I1Xi4cp}{2}{\draw (0,0) node[not] {} -- (-1,1) node[not] {}; \draw[kernels2] (-1,1) -- (0,2) ; 
\draw[kernels2] (-1,1) -- (-2,2) node[xi] {} ;
\draw (0,2) node[xi] {} -- (-1,3) node[xi] {};}

\DeclareSymbol{I1Xi4cc1}{2}{\draw (0,0) node[xic] {} -- (-1,1) node[not] {}; \draw[kernels2] (-1,1) -- (0,2) ; 
\draw[kernels2] (-1,1) -- (-2,2) node[xi] {} ;
\draw (0,2) node[xic] {} -- (-1,3) node[xi] {};}

\DeclareSymbol{I1Xi4cc2}{2}{\draw (0,0) node[xic] {} -- (-1,1) node[not] {}; \draw[kernels2] (-1,1) -- (0,2) ; 
\draw[kernels2] (-1,1) -- (-2,2) node[xi] {} ;
\draw (0,2) node[xi] {} -- (-1,3) node[xic] {};}

\DeclareSymbol{I1Xi4abc1}{2}{\draw[kernels2] (0,0) node[not] {} -- (-1,1) ; \draw[kernels2] (0,0) node[not] {} -- (1,1) node[xic] {};\draw (-1,1) node[xi] {} -- (0,2) ; \draw[kernels2] (0,2) node[not] {} -- (-1,3) node[xic] {};\draw[kernels2] (0,2)  -- (1,3) node[xi] {}; }

\DeclareSymbol{I1Xi4abc2}{2}{\draw[kernels2] (0,0) node[not] {} -- (-1,1) ; \draw[kernels2] (0,0) node[not] {} -- (1,1) node[xic] {};\draw (-1,1) node[xi] {} -- (0,2) ; \draw[kernels2] (0,2) node[not] {} -- (-1,3) node[xi] {};\draw[kernels2] (0,2)  -- (1,3) node[xic] {}; }

\DeclareSymbol{R1}{0}{\draw (-1,1) node[xi] {} -- (0,0) node[not] {};
\draw[kernels2] (0,1.5) node[xic] {} -- (0,0) -- (1,1) node[xic] {};}
\DeclareSymbol{R2}{0}{\draw (-1,1) node[xic] {} -- (0,0) node[not] {};
\draw[kernels2] (0,1.5)  {} -- (0,0) -- (1,1)  {};
\draw (0,1.5) node[xi] {};
\draw (1,1) node[xic] {};
}
\DeclareSymbol{R3}{1}{\draw[kernels2] (-1,1.5)  {} -- (0,0) node[not] {} -- (1,1.5);
\draw (-1,1.5) node[xi] {};
\draw[kernels2] (0,3) {} -- (1,1.5) -- (2,3)  {};
\draw  (0,3) node[xic] {} ;
\draw (2,3) node[xic] {};}
\DeclareSymbol{R4}{1}{\draw[kernels2] (-1,1.5) node[xic] {} -- (0,0) node[not] {} -- (1,1.5);
\draw[kernels2] (0,3) {} -- (1,1.5) -- (2,3) node[xic] {};
\draw (0,3) node[xi] {};}

\DeclareSymbol{I1Xi4bcp}{2}{\draw (0,0) node[not] {} -- (-1,1) node[not] {}; \draw[kernels2] (-1,1) -- (0,2) ; 
\draw[kernels2] (-1,1) -- (-2,2) node[xi] {} ; \draw[kernels2] (0,2) node[not] {} -- (-1,3) node[xi] {};\draw[kernels2] (0,2)  -- (1,3) node[xi] {};
}

\DeclareSymbol{I1Xi4bcc1}{2}{\draw (0,0) node[xic] {} -- (-1,1) node[not] {}; \draw[kernels2] (-1,1) -- (0,2) ; 
\draw[kernels2] (-1,1) -- (-2,2) node[xi] {} ; \draw[kernels2] (0,2) node[not] {} -- (-1,3) node[xi] {};\draw[kernels2] (0,2)  -- (1,3) node[xic] {};
}

\DeclareSymbol{I1Xi4bcc2}{2}{\draw (0,0) node[xic] {} -- (-1,1) node[not] {}; \draw[kernels2] (-1,1) -- (0,2) ; 
\draw[kernels2] (-1,1) -- (-2,2) node[xi] {} ; \draw[kernels2] (0,2) node[not] {} -- (-1,3) node[xic] {};\draw[kernels2] (0,2)  -- (1,3) node[xi] {};
} 

\DeclareSymbol{2I1Xi4bc1}{2}{\draw[kernels2] (0,0) node[not] {} -- (-1,1) ;
\draw[kernels2] (0,0) -- (1,1);
\draw (-1,1) node[xic] {} -- (-1,2.5) node[xi] {};
\draw (1,1)  node[xic] {} -- (1,2.5) node[xi] {};
}

\DeclareSymbol{2I1Xi4bc2}{2}{\draw[kernels2] (0,0) node[not] {} -- (-1,1) ;
\draw[kernels2] (0,0) -- (1,1);
\draw (-1,1) node[xi] {} -- (-1,2.5) node[xic] {};
\draw (1,1)  node[xic] {} -- (1,2.5) node[xi] {};
}

\DeclareSymbol{diff2I1Xi4bc2}{2}{\draw (0,0) node[diff] {} -- (-1,1) ;
\draw (0,0) -- (1,1);
\draw (-1,1) node[xi] {} -- (-1,2.5) node[xic] {};
\draw (1,1)  node[xic] {} -- (1,2.5) node[xi] {};
}

\DeclareSymbol{2I1Xi4bc3}{2}{\draw[kernels2] (0,0) node[not] {} -- (-1,1) ;
\draw[kernels2] (0,0) -- (1,1);
\draw (-1,1) node[xic] {} -- (-1,2.5) node[xic] {};
\draw (1,1)  node[xi] {} -- (1,2.5) node[xi] {};
}

\DeclareSymbol{Xi41}{0}{\draw (0,1) -- (0.8,2.2) node[xic] {};\draw (0,-0.25) node[xi] {} -- (0,1) node[xi] {} -- (-0.8,2.2) node[xic] {};} 

\DeclareSymbol{Xi42}{0}{\draw (0,1) -- (0.8,2.2) node[xi] {};\draw (0,-0.25) node[xic] {} -- (0,1) node[xi] {} -- (-0.8,2.2) node[xic] {};}

\DeclareSymbol{Xi4ca1}{0}{\draw (0,1) -- (-1,2.2) node[xic] {};\draw (0,-0.25) node[xi] {} -- (0,1) ; \draw[kernels2] (0,1) node[not] {} -- (1,2.2) node[xic] {};
\draw[kernels2] (0,1) {} -- (0,2.7) node[xi] {};
}

\DeclareSymbol{Xi4ca2}{0}{\draw (0,1) -- (-1,2.2) node[xi] {};\draw (0,-0.25) node[xi] {} -- (0,1) ; \draw[kernels2] (0,1) node[not] {} -- (1,2.2) node[xic] {};
\draw[kernels2] (0,1) {} -- (0,2.7) node[xic] {};
}

\DeclareSymbol{Xi4cap}{0}{\draw (0,1) -- (-1,2.2) node[xi] {};\draw (0,-0.25) node[not] {} -- (0,1) ; \draw[kernels2] (0,1) node[not] {} -- (1,2.2) node[xi] {};
\draw[kernels2] (0,1) {} -- (0,2.7) node[xi] {};
}

\DeclareSymbol{Xi3a}{0}{
 \draw (-1,1)  node[xi] {} -- (0,0); 
 \draw (0,0) node[xi] {}  -- (1,1) node[xi] {};
}

\DeclareSymbol{Xi4ebc1}{0}{
\draw[kernels2] (0,2) node[xi] {} -- (-1,1) ; \draw[kernels2] (-2,2)  node[xic] {} -- (-1,1) ; \draw (-1,1)  node[not] {} -- (0,0); 
 \draw (0,0) node[xic] {}  -- (1,1) node[xi] {};
}

\DeclareSymbol{Xi4ebc2}{0}{
\draw[kernels2] (0,2) node[xi] {} -- (-1,1) ; \draw[kernels2] (-2,2)  node[xi] {} -- (-1,1) ; \draw (-1,1)  node[not] {} -- (0,0); 
 \draw (0,0) node[xic] {}  -- (1,1) node[xic] {};
}

\DeclareSymbol{Xi2cbispex}{0}{\draw[kernels2] (0,1) -- (0.8,2.2) node[xi] {};\draw (0,-0.25) node[xie] {} -- (0,1); \draw[kernels2] (0,1) node[not] {} -- (-0.8,2.2) node[xi] {};}

\DeclareSymbol{Xi2cbis1p}{0}{\draw (0,1) -- (-0.8,2.2) node[xi] {};\draw (0,-0.25) node[not] {} -- (0,1) node[xi] {}; }

\DeclareSymbol{Xi2Xp}{-2}{\draw (0,-0.25) node[not] {} -- (-1,1) node[xix] {};} 

\DeclareSymbol{I1XiIXib}{0}{\draw  (0,-0.25) node[xi] {} -- (0,1) node[not] {};
\draw[kernels2] (0,1) -- (0,2.25) ; \draw (0,2.25) node[xi]{}; }

\DeclareSymbol{IXi2b}{0}{\draw  (0,-0.25) node[xi] {} -- (0,1) node[not] {};
\draw (0,1) -- (0,2.25) ; \draw (0,2.25) node[xi]{}; }

\DeclareSymbol{IXi2bex}{0}{\draw  (0,-0.25) node[xi] {} -- (0,1) node[xie] {};
\draw (0,1) -- (0,2.25) ; \draw (0,2.25) node[xi]{}; }

 \def\1{\mathbf{\symbol{1}}}

\DeclareSymbol{diff}{0}{
\draw (0,0.5) node[diff] {};
}

\DeclareSymbol{diff1}{0}{
\draw (0,0.5) node[diff1] {};
}

\DeclareSymbol{diff2}{0}{
\draw (0,0.5) node[diff2] {};
}

\DeclareSymbol{geo}{0}{
\draw (0,0) node[diff] {};
\draw (0.3,0) node[diff] {};
}

\DeclareSymbol{generic}{0}{
\draw (0,0.6) node[xi] {};
}

\DeclareSymbol{g}{0}{
\draw (0,0.6) node[g] {};
}

\DeclareSymbol{Ito}{0}{
\draw (0,0.6) node[xies] {};
}

\DeclareSymbol{Itob}{0}{
\draw (0,0.6) node[xiesf] {};
}

\DeclareSymbol{greycirc}{0}{
\draw (0,0.3) node[xi] {};
}

\DeclareSymbol{not}{0}{
\draw (0,0.6) node[not] {};
\draw[tinydots] (0,0.6) circle (0.8);
}

\DeclareSymbol{genericb}{0}{
\draw (0,0.6) node[xic] {};
}

\DeclareSymbol{bluecirc}{0}{
\draw (0,0.3) node[xic] {};
}

\DeclareSymbol{genericxix}{0}{
\draw (0,0.6) node[xix] {};
}

\DeclareSymbol{genericX}{0}{
\draw (0,0.6) node[X] {};
}

\DeclareSymbol{diffIto}{1}{
\draw  (0,2.5) -- (0,0) ;
\draw (0,-0.1) node[diff] {};
\draw (0,2.5) node[xies] {};
}
\DeclareSymbol{Itodiff}{2}{
\draw(0,2.9) -- (0,-0.2);
\draw (0,2.9) node[diff] {};
\draw (0,-0.1) node[xies] {};
}

\DeclareSymbol{diffgeneric}{1}{
\draw  (0,2.5) -- (0,0) ;
\draw (0,-0.1) node[diff] {};
\draw (0,2.5) node[xi] {};
}

\DeclareSymbol{genericdiff}{2}{
\draw(0,2.9) -- (0,-0.2);
\draw (0,2.9) node[diff] {};
\draw (0,-0.1) node[xi] {};
}

\DeclareSymbol{diffdot}{2}{
\draw  (0,3) -- (0,-0.1) ;
\draw (0,3) node[not] {};
\draw (0,-0.1) node[diff] {};
}

\DeclareSymbol{diffdotmini}{0}{
\draw  (0,0) -- (0,1.2) ;
\draw (0,1.2) node[not] {};
\draw (0,0) node[diffmini] {};
}

\DeclareSymbol{dotdiff}{2}{
\draw[kernelsmod]  (0,3) -- (0,-0.1) ;
\draw (0,3) node[diff] {};
\draw (0,-0.1) node[not] {};
}

\DeclareSymbol{dotdiff1}{2}{
\draw[kernelsmod]  (0,3) -- (0,-0.1) ;
\draw (0,3) node[diff1] {};
\draw (0,-0.1) node[not] {};
}

\DeclareSymbol{dotdiff1mini}{0}{
\draw[kernelsmod]  (0,1.2) -- (0,0) ;
\draw (0,1.2) node[diffmini] {};
\draw (0,0) node[not] {};
}

\DeclareSymbol{dotdiff2}{2}{
\draw (0,3) -- (0,-0.1) ;
\draw (0,3) node[diff] {};
\draw (0,-0.1) node[not] {};
}

\DeclareSymbol{dotdiff2mini}{0}{
\draw (0,1.2) -- (0,0) ;
\draw (0,1.2) node[diffmini] {};
\draw (0,0) node[not] {};
}

\DeclareSymbol{dotdiffstraight}{0}{
\draw  (0,3) -- (0,-0.1) ;
\draw (0,3) node[diff] {};
\draw (0,-0.1) node[not] {};
}

\DeclareSymbol{arbre1}{0}{
\draw  (0,0) -- (1.5,1.5) ;
\draw (1.5,1.5) node[not] {};
\draw (0,0) node[not] {};
}

\DeclareSymbol{arbre2}{0}{
\draw  (0,0) -- (1.5,1.5) ;
\draw[kernelsmod] (0,0) -- (-1.5,1.5);
\draw (1.5,1.5) node[not] {};
\draw (0,0) node[not] {};
\draw (-1.5,1.5) node[xi] {};
}

\DeclareSymbol{arbre3}{0}{
\draw  (0,0) -- (1.5,1.5) ;
\draw[kernelsmod] (1.5,1.5) -- (0,3);
\draw (0,0) node[not] {};
\draw (1.5,1.5) node[not] {};
\draw (0,3) node[xi] {};
}

\DeclareSymbol{treeeval}{0}{
\draw (0,0) -- (1,1);
\draw (0,0) node[xi] {};
\draw (1.25,1.25) node[xi] {};
\draw (-0.6,0.6) node[]{\tiny{$i$}};
\draw (0.65,1.85) node[]{\tiny{$j$}};
}

\DeclareSymbol{testeval}{0}{
\draw (0,0) -- (1,1);
\draw (0,0) -- (-1,1);
\draw (0,0) node[xi] {};
\draw (1.25,1.25) node[xi] {};
\draw (-1.25,1.25) node[xi] {};
\draw (-0.6,-0.6) node[]{\tiny{$i$}};
\draw (0.65,1.85) node[]{\tiny{$j$}};
\draw (-1.95,1.85) node[]{\tiny{$k$}};
}

\DeclareSymbol{treeeval2}{0}{
\draw[kernelsmod] (-0.25,-1) -- (1,0.5) ;
\draw[kernelsmod] (1,0.5) -- (-0.25,2);
\draw (1,0.5) node[diff2] {};
\draw (-0.25,-1) node[not] {};
\draw (-0.25,2) node[xi] {};
\draw (-0.6,1.2) node[]{\tiny{1}};
}

\DeclareSymbol{arbreact}{1}{
\draw (0,0) node[not] {};
\draw[kernelsmod] (0,0) -- (1,1);
\draw[kernelsmod] (0,0) -- (-1,1);
\draw (-1,1) node[xic] {};
\draw  (0,2) -- (1,1) ;
\draw (0,2) node[xic] {};
\draw (1,1) node[xi] {};
}

\DeclareSymbol{arbreact1}{0}{
\draw (0,-1.5) -- (0,0);
\draw[kernelsmod] (0,0) -- (1,1);
\draw[kernelsmod] (0,0) -- (-1,1);
\draw  (0,2) -- (1,1) ;
\draw (0,-1.5) node[diff] {};
\draw (0,0) node[not] {};
\draw (-1,1) node[xic] {};
\draw (0,2) node[xic] {};
\draw (1,1) node[xi] {};
}

\DeclareSymbol{arbreact2}{0}{
\draw (0,-0.75) -- (-1,0.5); 
\draw (0,-0.75) -- (1,0.5);
\draw (0,1.5) -- (1,0.5);
\draw (0,1.5) node[xic] {};
\draw (1,0.5) node[xi] {};
\draw (-1,0.5) node[xic] {};
\draw (0,-0.75) node[diff] {};
}

\DeclareSymbol{arbreact3}{0}{
\draw[kernelsmod] (0,-0.75) -- (-1,0.5); 
\draw[kernelsmod] (0,-0.75) -- (1,0.5);
\draw (0,1.75) -- (1,0.5);
\draw (2,1.75) -- (1,0.5);
\draw (0,1.75) node[xic] {};
\draw (1,0.5) node[diff] {};
\draw (-1,0.5) node[xic] {};
\draw (2,1.75) node[xi] {};
\draw (0,-0.75) node[not] {};
}

\DeclareSymbol{pre_im_I1Xitwo}{0}{
\draw[kernels2] (0,-0.3) node[not] {} -- (-0.6,0.7) ;
\draw[kernels2] (0,-0.3) -- (0.6,0.7);
\draw (0,0.9) node[g] {};
}

\DeclareSymbol{pre_im_cI1Xi4ab}{2}{
\draw[kernels2] (0,-1) node[not] {} -- (-0.6,0) ;
\draw[kernels2] (0,-1) -- (0.6,0);
\draw (0,0.2) node[g] {};
\draw (0,0.6) -- (0,1.5);
\draw[kernels2] (0,1.5) node[not] {} -- (-0.6,2.5) ;
\draw[kernels2] (0,1.5) -- (0.6,2.5);
\draw (0,2.7) node[g] {};
}

\DeclareSymbol{pre_im_I1Xi4acc2}{0}{
\draw[kernels2] (-1,-0.5) node[not] {} -- (-1.6,0.5) ;
\draw[kernels2] (-1,-0.5) -- (-0.4,0.5);
\draw[kernels2] (-1,-0.5) -- (0.2,-1.5) node[not] {} ;
\draw (-1,1.1) -- (-1,2);
\draw[kernels2] (0.2,-1.5) -- (0.2,2);
\draw (-1,0.7) node[g] {};
\draw (-0.3,2.2) node[g] {};
}

\DeclareSymbol{pre_im_I1Xi4abcc2}{2}{
\draw[kernels2] (0,-1) node[not] {} -- (-1,0) node[not] {};
\draw[kernels2] (-1,1.2) node[not] {} -- (-1,0);
\draw[kernels2] (-1,1.2) -- (-1.5,2.5);
\draw[kernels2] (-1,1.2) -- (-0.5,2.5);
\draw[kernels2] (-1,0) -- (0.7,2.5);
\draw[kernels2] (0,-1) -- (1.5,2.5);
\draw (-1,2.7) node[g] {};
\draw (1,2.7) node[g] {};
}

\DeclareSymbol{pre_im_2I1Xi4c1}{2}{
\draw[kernels2] (0,-0.5) node[not] {} -- (-1,0.5) node[not] {};
\draw[kernels2] (0,-0.5) -- (1,0.5) node[not] {};
\draw[kernels2] (-1,0.5) node[not] {}-- (-1.7,2);
\draw[kernels2]  (-1,2) -- (1,0.5);
\draw[kernels2] (-1,0.5) -- (1,2);
\draw[kernels2] (1,0.5) -- (1.7,2);
\draw (-1.2,2.2) node[g] {};
\draw (1.2,2.2) node[g] {};
}

\DeclareSymbol{pre_im_Xi4eabisc2}{2}{
\draw[kernels2] (1.2,-0.5) node[not] {} -- (-0.7,0.8) ;
\draw[kernels2] (1.2,-0.5) -- (0.4,0.8);
\draw (0,1.4)  -- (0,2.2);
\draw (1.2,2.2) -- (1.2,-0.6);
\draw (0,1) node[g] {};
\draw (0.6,2.4) node[g] {};
}

\DeclareSymbol{pre_im_Xi4eabisc22}{2}{
\draw (1.2,-0.5) node[not] {} -- (-0.7,0.8) ;
\draw[kernels2] (1.2,-0.5) -- (0.4,0.8);
\draw (0,1.4)  -- (0,2.2);
\draw[kernels2] (1.2,2.2) -- (1.2,-0.6);
\draw (0,1) node[g] {};
\draw (0.6,2.4) node[g] {};
}

\DeclareSymbol{pre_im_Xi4eabisc222}{2}{
\draw[kernels2] (0.4,-0.5) node[not] {} -- (-0.6,1) ;
\draw[kernels2] (1.2,0) -- (0.3,1);
\draw (0,1.1)  -- (0,2.5);
\draw[kernels2] (1.2,2.5) -- (1.2,0) node[not] {} -- (0.4,-0.6);
\draw (0,1.2) node[g] {};
\draw (0.6,2.5) node[g] {};
}

\DeclareSymbol{pre_im_Xi4eabc2}{2}{
\draw (0,-0.5) node[not] {} -- (-1,0.5) node[not] {};
\draw[kernels2] (-1,0.5) -- (-1.5,2);
\draw[kernels2] (0,-0.5)  -- (0.7,2);
\draw[kernels2] (-1,0.5) -- (-0.5,2);
\draw[kernels2] (0,-0.5) -- (1.5,2);
\draw (-1,2.2) node[g] {};
\draw (1,2.2) node[g] {};
}

\DeclareSymbol{pre_im_Xi4eabbisc2}{2}{
\draw[kernels2] (0,-0.5) node[not] {} -- (-1,0.5) node[not] {};
\draw[kernels2] (-1,0.5) -- (-1.5,2);
\draw[kernels2] (0,-0.5)  -- (0.7,2);
\draw[kernels2] (-1,0.5) -- (-0.5,2);
\draw (0,-0.5) -- (1.5,2);
\draw (-1,2.2) node[g] {};
\draw (1,2.2) node[g] {};
}

\DeclareSymbol{pre_im_I1Xi4abcc1}{2}{
\draw[kernels2] (0,-1) node[not] {} -- (-1,0) node[not] {};
\draw[kernels2] (0,1.1) node[not] {} -- (-1,0);
\draw[kernels2] (-1,0) -- (-1.5,2.5);
\draw[kernels2] (0,1.1) node[not] {} -- (-0.5,2.5);
\draw[kernels2] (0,1.1) -- (0.5,2.5);
\draw[kernels2] (0,-1) -- (1.5,2.5);
\draw (-1,2.7) node[g] {};
\draw (1,2.7) node[g] {};
}

\DeclareSymbol{pre_im_Xi4eabc1}{2}{
\draw (0,-0.5) node[not] {} -- (-1,0.5) node[not] {};
\draw[kernels2] (-1,0.5) -- (-1.5,2);
\draw[kernels2] (0,-0.5)  -- (-0.5,2);
\draw[kernels2] (-1,0.5) -- (0.8,2);
\draw[kernels2] (0,-0.5) -- (1.5,2);
\draw (-1,2.2) node[g] {};
\draw (1,2.2) node[g] {};
}

\DeclareSymbol{pre_im_Xi4ba1b}{2}{
\draw[kernels2] (0,0) node[not] {}  -- (1.8,1.5);
\draw[kernels2] (0,0) -- (0.8,1.5);
\draw (0,-0.1) -- (-1.8,1.5);
\draw (0,-0.1) -- (-0.8,1.5);
\draw (-1,1.7) node[g] {};
\draw (1,1.7) node[g] {};
}

\DeclareSymbol{pre_im_Xi4ba2}{2}{
\draw (0,-0.1) node[not] {}  -- (1.8,1.5);
\draw[kernels2] (0,0) -- (0.8,1.5);
\draw (0,-0.1) -- (-1.8,1.5);
\draw[kernels2] (0,0) -- (-0.8,1.5);
\draw (-1,1.7) node[g] {};
\draw (1,1.7) node[g] {};
}

\DeclareSymbol{pre_im_Xi4cabc2}{2}{
\draw[kernels2] (0,-0.5) node[not] {} -- (-1,0.5) node[not] {};
\draw[kernels2] (-1,0.5) -- (-1.5,2);
\draw (-1,0.5)  -- (0.7,2);
\draw[kernels2] (-1,0.5) -- (-0.5,2);
\draw[kernels2] (0,-0.5) -- (1.5,2);
\draw (-1,2.2) node[g] {};
\draw (1,2.2) node[g] {};
}

\DeclareSymbol{pre_im_Xi4cabc1}{2}{
\draw[kernels2] (0,-0.5) node[not] {} -- (-1,0.5) node[not] {};
\draw (-1,0.5) -- (-1.5,2);
\draw[kernels2] (-1,0.5)  -- (0.7,2);
\draw[kernels2] (-1,0.5) -- (-0.5,2);
\draw[kernels2] (0,-0.5) -- (1.5,2);
\draw (-1,2.2) node[g] {};
\draw (1,2.2) node[g] {};
}

\DeclareSymbol{pre_im_Xi4eabbisc1}{2}{
\draw[kernels2] (0,-0.5) node[not] {} -- (-1,0.5) node[not] {};
\draw[kernels2] (-1,0.5) -- (-1.5,2);
\draw[kernels2] (0,-0.5)  -- (-0.5,2);
\draw[kernels2] (-1,0.5) -- (0.8,2);
\draw (0,-0.5) -- (1.5,2);
\draw (-1,2.2) node[g] {};
\draw (1,2.2) node[g] {};
}

\DeclareSymbol{pre_im_1}{0}{
\draw[kernels2] (0,-0.5) node[not] {} -- (-0.6,0.5) ;
\draw[kernels2] (0,-0.5) -- (0.6,0.5);
\draw (0,1.1)  -- (-0.55,2);
\draw (0,1.1)  -- (0.55,2);
\draw (0,0.7) node[g] {};
\draw (0,2.2) node[g] {};
}

\DeclareSymbol{disconnect}{0}{
\draw[kernels2] (0,-0.5) node[not] {} -- (-0.6,0.5) ;
\draw[kernels2] (0,-0.5) -- (0.6,0.5);
\draw (-0.55,1.1)  -- (-0.55,2.3);
\draw (0.55,2.3) -- (0.55,1.5) -- (1.2,1.5) -- (1.2,3.5) -- (0.55,3.5) -- (0.55,2.7);
\draw (0,0.7) node[g] {};
\draw (0,2.5) node[g] {};
}

\DeclareSymbol{pre_im_2}{2}{\draw[kernels2] (0,0) node[not] {} -- (-1,1) node[not] {};
\draw[kernels2] (0,0) -- (1,1) node[not] {};
\draw[kernels2] (-1,1) -- (-1.5,2.5);
\draw[kernels2] (-1,1) -- (-0.5,2.5);
\draw[kernels2] (1,1) -- (0.5,2.5);
\draw[kernels2] (1,1) -- (1.5,2.5);
\draw (-1,2.7) node[g] {};
\draw (1,2.7) node[g] {};
}

\DeclareSymbol{CX_rec}{0}{
\draw [black] (-0.3,1) to (-0.3,-0.3);
\draw [black] (0.3,1) to (0.3,-0.3);
\draw [black] (-0.3,1) to (-0.3,2.3);
\draw [black] (0.3,1) to (0.3,2.3);
\draw (0,1) node[rec] {};
}

\DeclareSymbol{CX_cerc}{0}{
\draw [black] (0,1) to (0,-0.3);
\draw (0,1) node[cerc] {};
}

\DeclareMathAlphabet{\mathpzc}{OT1}{pzc}{m}{it}

\def\eqref#1{(\ref{#1})}

\makeatletter 
\newcommand*{\bigcdot}{}
\DeclareRobustCommand*{\bigcdot}{%
  \mathbin{\mathpalette\bigcdot@{}}%
}
\newcommand*{\bigcdot@scalefactor}{.5}
\newcommand*{\bigcdot@widthfactor}{1.15}
\newcommand*{\bigcdot@}[2]{%
  \sbox0{$#1\vcenter{}$}
  \sbox2{$#1\cdot\m@th$}%
  \hbox to \bigcdot@widthfactor\wd2{%
    \hfil
    \raise\ht0\hbox{%
      \scalebox{\bigcdot@scalefactor}{%
        \lower\ht0\hbox{$#1\pprod\m@th$}%
      }%
    }%
    \hfil
  }%
}
\makeatother

\tcbset
{colframe=boxcolor,colback=symbols!7!pagebackground,coltext=pageforeground,
fonttitle=\bfseries,nobeforeafter,center title,size=fbox,boxsep=1.5pt,
top=0mm,bottom=0mm,boxsep=0mm,tcbox raise base}

\def\two{{\<generic>\kern0.05em\<genericb>}}
\def\twoI{{\<Ito>\kern0.05em\<Itob>}}

\def\mail#1{\burlalt{#1}{mailto:#1}}

\newcommand{\cop}{\textnormal{cop}\,}

\newcommand{\rmU}{{\rm U}}

\newcommand{\lnh}{(\hspace{-5pt}(\hspace{1pt}}
\newcommand{\rnh}{\hspace{1pt})\hspace{-5pt})\hspace{1pt}}
\newcommand{\lspan}{\textnormal{span}\,}

\newcommand{\populated}{\mathcal{T}} 

\newcommand{\length}[1]{\mathcal{l}({#1})}
\newcommand{\alphamin}{\underline{\alpha}}

\newcommand{\mbfX}{\mathbf{x}}
\newcommand{\lk}{{(\mathfrak{l},k)}}
\newcommand{\btl}{\blacktriangleright_{\mfl}}

\newcommand{\renlie}{\mathcal{R}}
\newcommand{\trir}{\triangleright}
\newcommand{\pprod}{\star}
\newcommand{\tripos}{D}
\newcommand{\trineg}{\blacktriangleright}
\newcommand{\gr}{\textnormal{deg}}
\newcommand{\grad}[1]{\gr({#1})}
\newcommand{\bff}{\mathbf{f}}
\newcommand{\bfg}{\mathbf{g}}

\newcommand{\renlietwo}{{\renlie_{\geq 2}}}
\newcommand{\Rentwo}{R_{\geq 2}}
\newcommand{\rhotil}{{\tilde{\rho}}}
\newcommand{\subL}{{\tilde{\mfL}}}
\newcommand{\mbfXbranched}{\mbfX^{\textnormal{Br}}}
\newcommand{\tmbfXbranched}{\tilde{\mbfX}^{\textnormal{Br}}}
\newcommand{\triins}{{\triangleright_{\rm ins}}}
\newcommand{\triinsl}{{\triangleright_{{\rm ins},\mfl}}}
\newcommand{\triinslp}{{\triangleright_{{\rm ins},\mfl'}}}
\newcommand{\indexset}{\mathscr{I}}
\newcommand{\indexsettwo}{\mathscr{J}}
\newcommand{\tricom}{{\triangleright_{\rm null}}}
\newcommand{\kappamin}{\underline{\kappa}}

\usepackage{thmtools}

\begin{document}

\title{Insertion pre-Lie products and translation of rough paths based on multi-indices}
\author{Pablo Linares}
\institute{ Imperial College London \\
Email:\ \begin{minipage}[t]{\linewidth}
\mail{p.linares-ballesteros@imperial.ac.uk}.
\end{minipage}}


\maketitle 

\begin{abstract}
	We use the diagram-free approach to regularity structures introduced by Otto et. al. to build rough paths based on multi-indices. We identify the analogue of the insertion pre-Lie algebra of trees and use it to build the corresponding group of translations of rough paths. We make this identification precise by showing that the natural dictionary between trees and multi-indices is a pre-Lie morphism under insertion, which in turn yields a Hopf algebra morphism to the rooted tree Hopf algebra equipped with the extraction-contraction coproduct.
\end{abstract}

\setcounter{tocdepth}{2}
\setcounter{secnumdepth}{4}
\tableofcontents

\section{Introduction}
\label{Introduction}

In this work we implement the algebraic machinery of the diagram-free approach to regularity structures introduced in \cite{OSSW} and systematized in \cite{LOT,BL23} to rough differential equations (RDEs) of the form
\begin{align}\label{rde01}
	dY_t = \sum_{\mfl\in\mfL} a_\mfl(Y_t) dX^\mfl_t,\quad Y_0 = y_0\in\R,
\end{align}	
where $Y:[0,1] \to \R$, $\mfL$ is a finite set and, for each $\mfl\in \mfL$, $a_\mfl : \R \to \R$ is smooth and $X^\mfl : [0,1]\to \R$ is continuous. In applications, $\{X^\mfl\}_{\mfl\in\mfL}$ typically is the realization of a multi-dimensional stochastic process with almost surely continuous but non-differentiable trajectories. The pathwise analysis of such equations is possible as long as we postulate an interpretation of multiple integrals of $\{X^\mfl\}_{\mfl\in\mfL}$; this is the core of the theory of rough paths \cite{Lyons98, Gub04}. Assuming an algebraic integration by parts formula, it is enough to postulate only the \textit{iterated integrals} of $\{X^\mfl\}_{\mfl\in\mfL}$, in what have been called \textit{weakly geometric rough paths}. These can be interpreted as rough paths built upon the Hopf algebra of words over the alphabet $\mfL$; cf. e.~g. \cite[Subsection 1.1]{HK} for a concise exposition. Later, Gubinelli \cite{Gub06}, inspired by Butcher series \cite{Butcher72,CHV}, introduced the notion of \textit{branched rough paths}, i.~e. rough paths built upon the Hopf algebra of Connes-Kreimer \cite{CK1}, which now encodes all possible products of iterated integrals of $\{X^\mfl\}_{\mfl\in\mfL}$ without having to assume any integration by parts. Other examples of Hopf algebras have been explored, see \cite{Bel20,CEFMMK} to name only a few. 

\medskip

In affine spaces, vector fields carry a natural pre-Lie algebra structure (cf. e.~g. \cite{Manchon}): This pre-Lie algebra can be linked to the Connes-Kreimer Hopf algebra (and thus to branched rough paths) because Connes-Kreimer is dual to the universal envelope of the free pre-Lie algebra over the alphabet $\mfL$, also called Grossman-Larson Hopf algebra (cf. \cite{CL,Hoffman}). Therefore, pre-Lie algebras become a crucial structure in rough paths, 
as first observed in \cite{BCFP}, in order to connect certain transformations (translations) of the rough path with transformations of the equation via pre-Lie morphisms. Later, pre-Lie algebras were incorporated to the algebraic theory of regularity structures \cite{BCCH} with the same purpose, allowing to describe renormalized equations. More recently, pre-Lie algebras have been used for describing both recentering and algebraic renormalization of regularity structures in \cite{BM22}. Also \cite{LOT,BL23}, in the multi-index setup introduced in \cite{OSSW}, are based on pre-Lie algebras, but only make use of them in the space of solutions, in the form of a \textit{composition law} in the B-series jargon (cf. e.~g. \cite{CHV}): The current paper completes the picture, in the easier rough path context, providing the pre-Lie algebra associated to algebraic renormalization (\textit{substitution law}) in the multi-index approach.

\medskip

As a first goal, interpreting the theory of rough paths as a particular case of regularity structures \cite{reg,BHZ,CH16,BCCH}, in Section \ref{sec::3} we implement the program of \cite{BL23} to provide a notion of rough path based on multi-indices. In particular, we formulate the (smooth) rough path equations in terms of an exponential map (Lemma \ref{lem:3.5}), unveiling the algebraic structure required for Definition \ref{def:rp}. The rough path equations embed our approach in the construction of Hopf-algebraic smooth rough paths in \cite[Subsection 4.1]{BFPP}, although our focus is on the Lie-algebra side rather than the Lie-group side (see Remark \ref{rem:srp01} for the precise connection). In Section \ref{sec::4}, we introduce pre-Lie products for multi-indices which allow us to build the corresponding groups of translations (algebraic renormalization) in the sense of \cite{BCFP}. These pre-Lie products arise from the infinitesimal generators of a class of affine transformations of the equations, associated to introducing \textit{admissible counterterms} in the language of \cite[Subsection 3.4]{BL23}; note that \cite{BL23} does not build a group structure for algebraic renormalization in the associated regularity structures, since this would require extended decorations as noted in \cite{BHZ} (these are only needed in the presence of polynomials, and therefore do not play a role in the rough path case). Our group of transformations leads to a notion of translated rough paths, obtaining the analogue of \cite[Theorem 30]{BCFP} in our case, cf. Theorem \ref{th:trp} below. As a difference with the tree-based case, and the perspective described in \cite[Remark 4.13]{BFPP}, we highlight that our construction does not rely on having an underlying free structure: Even if we may identify our Lie algebra with the free Novikov algebra \cite[Section 7]{DL02}, cf. Lemma \ref{lem:nov01} below, the translation maps are built independently, and the pre-Lie morphism property \eqref{tra22} is shown a posteriori. Having the group of transformations is a sign that a multi-index approach with extended decorations is indeed conceivable to obtain the desired cointeraction property of \cite{BHZ} in the case of regularity structures (which we obtain for rough paths in its dual pre-Lie formulation as the generalized Leibniz rule \eqref{pre02} below). Finally, Section \ref{sec::5} connects multi-indices to trees, following \cite[Section 6]{LOT} and \cite[Subsection 2.5]{BL23}. In both cases, a ``dictionary" between trees and multi-indices is constructed, showing in addition that it is a pre-Lie morphism with respect to tree grafting (see \cite[Proposition 2.21]{BL23} for the formulation which is closer to our application here). We show that the same map is also a multi pre-Lie morphism with respect to a family of insertion pre-Lie products, which are a slight generalization of the one introduced in \cite[Section 4]{CEFM}; cf. Subsection \ref{subsec::5.2}. As a consequence, a Hopf algebra morphism with respect to the rooted tree Hopf algebra, equipped with the extraction-contraction coproduct, is established (cf. Corollary \ref{cor:hopf02}).  These results are of independent interest, as they could potentially provide information about the combinatorics of the insertion pre-Lie product.
\subsection*{Acknowledgements}

{\small
	The author sincerely thanks Yvain Bruned, Víctor Carmona, Kurusch Ebrahimi-Fard, Peter Friz and Dominique Manchon for helpful discussions, and Xue-Mei Li for financial support via the EPSRC grant EP/V026100/1.
}   
\section{Algebraic preliminaries}\label{sec::2}
In this section we collect some algebraic tools used in \cite[Sections 4 and 5]{LOT} and \cite[Section 3]{BL23}, formulated in enough generality to cover the situations considered in later sections. For completeness we include all the arguments, which will be well-known for experts on pre-Lie algebras, and which non-experts may skip at first read.
\subsection{The explicit Guin-Oudom procedure}\label{subsec::2.1}
Let $L$ be a vector space with a countable basis $\{a_i\}_{i\in\indexset}$, where $\indexset$ is a countable set. We endow $L$ with a pre-Lie product $\trir$.
\begin{definition} A pre-Lie product $\trir : L \times L \to L$ is a bilinear map such that for all $a,b,c\in L$
	\begin{align}\label{prelie}
		a\trir (b \trir c) - (a \trir b) \trir c = b \trir (a \trir c) - (b \trir a) \trir c.
	\end{align}
$(L,\trir)$ is called \textit{pre-Lie algebra}.
\end{definition}
This notion is stronger than that of Lie algebra, since one can always define a bracket
\begin{align*}
	[a,a']_\trir := a\trir a' - a'\trir a
\end{align*}
under which $(L,[\cdot,\cdot]_\trir)$ is a Lie algebra. See \cite{Manchon} for a survey on pre-Lie algebras.

\medskip

A pre-Lie product $\trir$ carries a canonical action of $L$ onto itself, which we can write as a map 
\begin{align}\label{go01}
	L\ni a \mapsto \rho_\trir(a) := a\trir \in \textnormal{End}(L).
\end{align}
This map is a Lie algebra morphism with respect to composition of endomorphisms.
\begin{lemma}
	\begin{align*}
		\rho_\trir ([a,a']_\trir) = \rho_\trir(a)\rho_\trir(a') - \rho_\trir(a') \rho_\trir (a).
	\end{align*}
\end{lemma}	
\begin{proof}
	Take $b\in L$. Using the definition \eqref{go01} and the pre-Lie identity \eqref{prelie},
	\begin{align*}
		\rho_\trir ([a,a']_\trir) b &= \rho_\trir (a\trir a')b - \rho_\trir(a'\trir a) b\\
		&= (a\trir a')\trir b - (a'\trir a)\trir b\\
		&= a\trir (a'\trir b) - a'\trir (a\trir b)\\
		&= \rho_\trir(a)\rho_\trir(a') b - \rho_\trir(a') \rho_\trir(a) b.
	\end{align*}
\end{proof}	
We may write the structure constants of $(L,\trir)$ with respect to the basis $\{a_i\}_{i\in\indexset}$ in terms of $\rho_\trir$, namely\footnote{ The placement of the indices seems unnatural at this stage, but will become clearer later when we identify the basis with a set of monomials (i.~e. with the indices up).}
\begin{align}\label{go18}
	a_i \trir a_j = \sum_{k\in\indexset} \big(\rho_\trir(a_i)\big)_k^j a_k.
\end{align}

\medskip

Guin and Oudom \cite{GO1,GO2} established that the universal envelope of a pre-Lie algebra is isomorphic to its symmetric algebra endowed with a non-commutative product built from the corresponding pre-Lie product. What follows is a reproduction of their proof up to a small twist in line with the strategy of \cite[Section 4]{LOT}: Instead of constructing a non-commutative product (Grossman-Larson) in the symmetric algebra, we build a symmetric basis  (which we will call Guin-Oudom basis, cf. Lemma \ref{lem:go01} below) of the universal enveloping algebra. For this we will not need the full Guin-Oudom construction of the product, but only ``half" of it, namely a right symmetric action of $L$ on its universal envelope.

\medskip

Denote by $\rmU (L,\trir)$ the universal enveloping algebra of $(L,[\cdot,\cdot]_\trir)$, i.~e. the tensor algebra over $L$ quotiented by the ideal generated by the bracket $[\cdot,\cdot]_\trir$, cf. \cite[p.28]{Abe}.  $\rmU(L,\trir)$ is a Hopf algebra (\cite[Examples 2.5 and 2.8]{Abe}), with the product given by concatenation and the coproduct given by the algebra morphism extension of the map
\begin{align}\label{go16}
	a \mapsto a\otimes 1 + 1 \otimes a.
\end{align}
By the universality property \cite[p. 29]{Abe}, $\rho_\trir$ uniquely extends to an algebra morphism $\rho_\trir:\rmU (L,\trir)\to \textnormal{End}(L)$, i.~e. for $U,U'\in \rmU(L,\trir)$
\begin{align*}
	\rho_\trir (U U') = \rho_\trir (U) \rho_\trir (U').
\end{align*}
As mentioned above, our goal is to build a symmetric basis in $\rmU(L,\trir)$. Recall that, by the Poincaré-Birkhoff-Witt theorem \cite[Theorem 1.9.6]{HGK}, given a total order $\leq$ in $\indexset$, the set of ordered concatenated elements $\{a_{i_1}\cdots a_{i_N}\,|\, N\in \N_0,\, i_1\leq...\leq i_N\}$ is a basis, but it crucially depends on the ordering. We will obtain a symmetric basis by means of constructing a right symmetric action of $L$ on $\rmU (L,\trir)$.  We start considering $\pprod_\trir:L\times L \to \rmU(L,\trir)$ defined as
\begin{align}\label{go02}
	a\pprod_\trir a' := aa' - a\trir a'.
\end{align}
Note at this stage that \eqref{go02} is commutative, since
\begin{align}\label{go03}
	a\pprod_\trir a' - a' \pprod_\trir a = aa' - a'a - [a,a']_\trir,
\end{align}
which is $0$ in $\rmU(L,\trir)$. We extend this operation to $\pprod_\trir:\rmU(L,\trir)\times L \to \rmU(L,\trir)$ inductively: For $a,a'\in L$ and $U\in\rmU(L,\trir)$,
\begin{align}\label{go06}
	\left\{\begin{array}{l}
		1\pprod_\trir a = a\\
		a'U \pprod_\trir a = a'(U \pprod_\trir  a) - U \pprod_\trir (a'\trir a),
	\end{array}\right.
\end{align}	
cf. \cite[Proposition 2.7 (i) and (ii)]{GO2}. We have the following properties.
\begin{lemma}\label{lem:go01}
	\mbox{}
	\begin{enumerate}[label=(\roman*)]
		\item Commutativity: For all $a,a'\in L$ and $U\in \rmU(L,\trir)$,
		\begin{align}\label{go04}
			\big(U \pprod_\trir a\big)\pprod_\trir a'  = \big(U \pprod_\trir a'\big) \pprod_\trir a.
		\end{align}
		\item Coalgebra morphism: For all $a\in L$ and $U\in \rmU(L,\trir)$,
		\begin{align}\label{go05}
			\cop U \pprod_\trir a = \sum_{(U)} U_{(1)} \pprod_\trir a \otimes U_{(2)} + U_{(1)}\otimes U_{(2)} \pprod_\trir a,
		\end{align}
		where we used Sweedler's notation for the coproduct
		\begin{equation*}
			\cop U = \sum_{(U)} U_{(1)}\otimes U_{(2)}.
		\end{equation*}
		\item Generalized Leibniz rule: With the same notation, for all $a\in L$ and $U\in \rmU(L,\trir)$,
		\begin{align}\label{go07}
			Ua = \sum_{(U)} U_{(2)}\pprod_\trir (\rho_\trir(U_{(1)})a).
		\end{align}
	\end{enumerate}
\end{lemma}	
\begin{proof}
	For all three identities, it is enough to show them for $U=1$ and, assuming it true for a generic $U$, to prove them for $bU$ with $b\in L$.
	\begin{enumerate}[label=(\roman*)]
		\item The case $U=1$ is \eqref{go03}. Moreover,
		\begin{align*}
			&(bU \pprod_\trir a)\pprod_\trir a'\\
			&\quad= \big(b (U\pprod_\trir a)\big)\trir a' - \big(U \pprod_\trir (b\trir a)\big)\pprod_\trir a' \\
			&\quad= b\big((U\pprod_\trir a) \pprod_\trir a'\big) - (U\pprod_\trir a) \pprod_\trir (b\trir a') -\big(U \pprod_\trir (b\trir a)\big)\pprod_\trir a'.
		\end{align*}
		Assuming \eqref{go04} holds for $U$, this is equal to
		\begin{align*}
			&b\big((U\pprod_\trir a') \pprod_\trir a\big) -\big(U \pprod_\trir (b\trir a')\big)\pprod_\trir a
			- (U\pprod_\trir a') \pprod_\trir (b\trir a)\\
			&\quad = (bU\pprod_\trir a')\pprod_\trir a.
		\end{align*}
		\item The case $U=1$ is \eqref{go16}. Moreover, assuming \eqref{go05} holds for $U$,
		\begin{align*}
			\cop bU \pprod_\trir a &= \cop b(U\pprod_\trir a) - \cop U \pprod_\trir (b\trir a)\\
			 &= \sum_{(U)} b(U_{(1)} \pprod_\trir a)\otimes U_{(2)} + b U_{(1)} \otimes U_{(2)} \pprod_\trir a \\
			 &\quad\quad+ U_{(1)}\pprod_\trir a \otimes b U_{(2)} + U_{(1)}\otimes b(U_{(2)}\pprod_\trir a) \\
			 &\quad \quad- U_{(1)}\pprod_\trir (b\trir a)\otimes U_{(2)} - U_{(1)}\otimes U_{(2)}\pprod_\trir(b\trir a).
		\end{align*}
		Rearranging terms via \eqref{go06}, and using
		\begin{align}\label{go08}
			\cop bU = \sum_{(U)} bU_{(1)}\otimes U_{(2)} + U_{(1)}\otimes b U_{(2)}, 
		\end{align}
		we obtain the desired \eqref{go05}.
		\item The case $U=1$ is trivial since $\rho_\trir (1) = \id$. In addition, assuming \eqref{go07} holds for $U$ and using \eqref{go06},
		\begin{align*}
			bUa &= b \sum_{(U)} U_{(2)}\pprod_\trir (\rho_\trir(U_{(1)})a) \\
			&= \sum_{(U)} (bU_{(2)})\pprod_\trir \big(\rho_\trir(U_{(1)})a\big) + U_{(2)} \pprod_\trir\big(b\trir \rho_\trir(U_{(1)})a\big) \\
			&= \sum_{(U)} (bU_{(2)})\pprod_\trir \big(\rho_\trir(U_{(1)})a\big) + U_{(2)} \pprod_\trir\big(\rho_\trir(b) \rho_\trir(U_{(1)})a\big);
		\end{align*}
		the claim follows by the algebra morphism property of $\rho_\trir$ and \eqref{go08}.	
	\end{enumerate}
\end{proof}	
\begin{remark}\label{rem:glp}
	Equation \eqref{go07} is a particular case of \cite[Definition 2.9]{GO2}, which connects the concatenation product in $\rmU (L,\triangleright)$ with the Grossman-Larson product in $\mathrm{S}(L)$.
\end{remark}	

\medskip

We now introduce some notation. A multi-index over $\indexset$ is a function $I:\indexset \to \N_0$ such that $I(i) = 0$ for all but finitely many $i\in \indexset$. We denote by $M(\indexset)$ the set of multi-indices over $\indexset$. For every $i\in \indexset$, $e_i\in M(\indexset)$ denotes the length-one multi-index $e_i (i') = \delta_i^{i'}$. We sometimes write multi-indices in the following way:
\begin{equation}\label{multi01}
	I = \sum_{i\in\indexset} I(i) e_i = \sum_{k=1}^{\length{I}} e_{i_k},
\end{equation}
where
\begin{equation*}
	\length{I} := \sum_{i\in\indexset} I(i)
\end{equation*}
is the length of $I$. The second sum allows for repeated indices (i.~e. $i_k = i_{k'}$ for $k\neq k'$ is allowed), while the first one does not. For a multi-index of the form \eqref{multi01} we define the multi-index factorial 
\begin{equation*}
	I! := \prod_{i\in\indexset} I(i)!
\end{equation*}
For $I$ as in \eqref{multi01} we define the element $A_I \in \rmU(L,\triangleright)$ as
\begin{align}\label{go30}
	A_I := \frac{1}{I!} \big( \cdots \big(\big( 1 \pprod_\trir a_{i_1}\big) \cdots \big) \pprod_\trir a_{i_{\length{I}}} \big),
\end{align}
i.~e. the iterative application of $\pprod_\trir a_{i_k}$ to $1\in \rmU(L,\trir)$. The right symmetry property \eqref{go04} implies that \eqref{go30} does not depend on the order of application, so the multi-index notation is meaningful. In the sequel, we will omit the brackets and the $1$ and simply write $a_{i_1} \pprod_\trir\cdots \pprod_\trir a_{i_k}$.
\begin{lemma}\label{lem:go02}
	The set $\{A_I\}_{I\in M(\indexset)}$ is a basis of $\rmU(L,\trir)$. We call it the \textbf{Guin-Oudom basis} of $\rmU(L,\trir)$ over the basis $\{a_i\}_{i\in\indexset}$.
\end{lemma}	
\begin{proof}
	Give an arbitrary ordering to $\indexset$, and consider its associated Poincaré-Birkhoff-Witt basis. We write every $A_I$ using this basis applying \eqref{go06} iteratively:  This representation is strictly triangular with respect to $\length{\cdot}$, and thus invertible.
\end{proof}	
We now study some properties of the representation of the product and the coproduct of $\rmU (L,\trir)$ with respect to the Guin-Oudom basis.
\begin{lemma}\label{lem:2.6}
	\mbox{}
	\begin{enumerate}[label=(\roman*)]
		\item For every $I\in M(\indexset)$,
		\begin{align}\label{go11}
			\cop A_I = \sum_{I'+I''=I} A_{I'}\otimes A_{I''}.
		\end{align}
		\item Let $\proj_L : \rmU(L,\trir)\to L$ be the canonical projection given by the basis\footnote{ i.~e. the projection $\rmU(L,\trir)\to\lspan_{i\in \indexset} \{A_{e_i}\}$ composed with the natural $A_{e_i}\mapsto a_i$.}, and let $\epsilon : \rmU(L,\trir)$ $\to$ $\R$ be the counit\footnote{ i.~e. $\epsilon (A_I) = \delta_I^0$.} of $\rmU(L,\trir)$. Then for all $U,U'\in \rmU(L,\trir)$
		\begin{align}\label{go12}
			\proj_L (U U') = \proj_L (U) \epsilon(U') + \rho_\trir(U)\proj_L (U').
		\end{align}
	\end{enumerate}
\end{lemma}	
\begin{proof}
	Equation \eqref{go11} is a straightforward consequence of \eqref{go05} via the summation lemma \cite[Lemma A.1]{LOT}. For \eqref{go12}, it is enough to show it for two basis elements $A_{I_1},A_{I_2}$. We argue by induction in the length of $I_2$. If $I_2=0$, \eqref{go12} is trivial since $\proj_L(1)=0$ and $\epsilon(1)=1$. We now assume it for all multi-indices $I'$ of length smaller or equal than that of $I_2$ and give ourselves a basis element $a_i$ of $L$. We use \eqref{go07} in combination with \eqref{go11} to obtain
	\begin{align}
		&\proj_L (A_{I_1} (A_{I_2}\pprod_\trir a_i))\nonumber \\
		&\quad = \proj_L(A_{I_1} A_{I_2} a_i) - \proj_L\Big(A_{I_1} \sum_{\substack{I_2' + I_2'' = I_2\\ I_2'\neq 0}} A_{I_2''} \pprod_\trir\big(\rho_\trir (A_{I_2'})a_i\big) \Big).\label{go14}
	\end{align}
	For the first summand, we note that by the algebra morphism property of $\cop$ combined with \eqref{go11},
	\begin{align*}
		\cop A_{I_1}A_{I_2} = \sum_{\substack{I_1' + I_1'' = I_1\\ I_2'+I_2'' = I_2}} A_{I_1'}A_{I_2'}\otimes A_{I_1''}A_{I_2''},
	\end{align*}
	therefore \eqref{go07} implies
	\begin{align*}
		A_{I_1} A_{I_2} a_i = \sum_{\substack{I_1' + I_1'' = I_1\\ I_2'+I_2'' = I_2}} A_{I_1''}A_{I_2''}\pprod_\trir\big(\rho_\trir (A_{I_1'}A_{I_2'})a_i\big) .
	\end{align*}
	Since $\rho_\trir (A_{I_1'}A_{I_2}') a_i$ $\in$ $L$, the above yields
	\begin{align}\label{go15}
		\proj_L \big(A_{I_1} A_{I_2} a_i\big) =  \rho_\trir(A_{I_1}A_{I_2})a_i.
	\end{align}
	In the second summand in \eqref{go14}, we first note that if $I_2=0$ the condition under the sum is empty, so the whole term does not contribute to \eqref{go14} and \eqref{go15} already gives \eqref{go12}. For $I_2\neq 0$,  we use the induction hypothesis and write
	\begin{align*}
		{}&\proj_L\Big(A_{I_1} \sum_{\substack{I_2' + I_2'' = I_2\\ I_2'\neq 0}} A_{I_2''} \pprod_\trir\rho_\trir (A_{I_2'})a_i \Big)\\
		 &\quad= \rho_\trir(A_{I_1})\sum_{\substack{I_2' + I_2'' = I_2\\ I_2'\neq 0}} \proj_L \Big(A_{I_2''} \pprod_\trir\rho_\trir (A_{I_2'})a_i \Big),
	 \end{align*}
 which as before reduces to $\rho_\trir(A_{I_1})\rho_\trir (A_{I_2})a_i$.	This cancels out with \eqref{go15} by the algebra morphism property of $\rho_\trir$.
\end{proof}	
\begin{remark}
	See \cite[Proposition 3.2]{GO2} for a version of \eqref{go12} in the context of tree grafting.
\end{remark}	
\subsection{Grading and transposition}\label{subsec::2.2}
	Property \eqref{go12} means that, as a coalgebra, $\rmU(L,\trir)$ is dual to the free commutative algebra over $\{a_i\}_{i\in\indexset}$, which we will henceforth denote by $\R[\indexset]$, with the canonical pairing given in terms of the Guin-Oudom basis, i.~e.
	\begin{align}\label{dua02}
		\langle A_{I'}, a^I\rangle = \delta_{I'}^{I},
	\end{align}
	where $a^I := \prod_{i\in \mathcal{I}} a_i^{I(i)}$. The question now is whether the Hopf algebra structure, and not just the coalgebra, can be transposed to $\R[\indexset]$. This is in general not true, but in this subsection we give sufficient conditions under which transposition is possible that fit our purposes.
	\begin{definition}
		A pre-Lie algebra $(L,\trir)$ is \textit{graded} if it can be decomposed as $L = \bigoplus_{\kappa\in \mathsf{K}} L_\kappa$ for some countable subset $\mathsf{K}\subset \R$ and
		\begin{align*}
			L_\kappa \trir L_{\kappa'}\subset L_{\kappa + \kappa'}.
		\end{align*}
		We will say that $(L,\trir)$ is $\kappamin$-connected, $\kappamin\in \R$, if $\min_{\kappa\in\mathsf{K}} = \kappamin$. If $\kappamin = 0$, we simply say $(L,\trir)$ is connected.
	\end{definition}
Without loss of generality, we assume that the basis $\{a_i\}_{i\in \indexset}$ is homogeneous, and we consider the map $\gr : \indexset\to \mathsf{K}$ that associates to every $i\in\indexset$ its degree, i.~e.
\begin{align*}
	\grad{i} = \kappa \iff a_i \in L_\kappa.
\end{align*}
Note that this implies for the structure constants of the pre-Lie algebra
\begin{align}\label{gra01}
	\big(\rho_\trir(a_i)\big)_k^j \neq 0\,\implies \grad{k} = \grad{i} + \grad{j}.
\end{align}
If $(L,\trir)$ is $\kappamin$-connected for some $\kappamin >0$, then $\grad{\indexset}$ is bounded from below away from $0$.
\begin{proposition}\label{prop:tra01}
	Let $(L,\trir)$ be a $\kappamin$-connected graded pre-Lie algebra with $\kappamin >0$. Assume that the structure constants \eqref{go18} satisfy the following finiteness property:
	\begin{align}\label{fin01}
		\mbox{for all }k\in\indexset,\quad\#\{i,j\in\indexset\,|\,\big(\rho_\trir(a_i)\big)_k^j \neq 0\}<\infty.
	\end{align}
	Then the free commutative algebra $\R[\indexset]$ is a connected graded Hopf algebra, where the coproduct $\Delta_\trir : \R[\indexset] \to \R[\indexset]\otimes \R[\indexset]$ is the transpose of the concatenation product in $\rmU(L,\trir)$ with respect to the Guin-Oudom basis. More precisely, $\Delta_\trir$ is defined on monomials as
	\begin{align}\label{tra01}
		\Delta_\trir a^{I} = \sum_{I',I''} (\Delta_\trir)_{I',I''}^I a^{I'}\otimes a^{I''},
	\end{align}
	where $(\Delta_\trir)_{I',I''}^I$ are the structure constants of the concatenation product of $\rmU(L,\trir)$ with respect to the Guin-Oudom basis, i.~e.
	\begin{align*}
		A_{I'}A_{I''} = \sum_{I} (\Delta_\trir)_{I',I''}^I A_I.
	\end{align*}	
\end{proposition}	
\begin{remark}
	A sufficient condition for \eqref{fin01} is that for every $\kappa\in \mathsf{K}$ the pre-image of $\kappa$ under $\grad{\cdot}$ is a finite set; equivalently, that for every $\kappa$ $\in$ $\mathsf{K}$, $L_\kappa$ is finite-dimensional.
\end{remark}	
\begin{proof}
	Since the grading is strictly positive, the universal enveloping algebra $\rmU(L,\trir)$ is a connected graded Hopf algebra, i.~e.
	\begin{align*}
		\rmU = \R \oplus \bigoplus_{\nu>0} \rmU_\nu.
	\end{align*}
	As a consequence, if $\R[\indexset]$ is a bialgebra, it will be a connected graded bialgebra, thus automatically a Hopf algebra. Our goal is then to show that $\R[\indexset]$ is indeed a bialgebra, which in turn reduces to showing that \eqref{tra01} is well-defined. This will be a consequence of the finiteness property
	\begin{align}\label{tra02}
		\mbox{for all }I\in M(\indexset),\quad \#\{I',I''\in M(\indexset)\,|\, (\Delta_\trir)_{I',I''}^{I}\neq 0\} < \infty.
	\end{align}
	By the algebra morphism property of the coproduct of $\rmU(L,\trir)$, which thanks to \eqref{go11} is expressed in coordinates as
	\begin{align*}
		(\Delta_\trir)_{I',I''}^{I_1+I_2} = \sum_{\substack{I_1' + I_2' = I'\\ I_1''+I_2'' = I''}} (\Delta_\trir)_{I_1',I_1''}^{I_1} (\Delta_\trir)_{I_2',I_2''}^{I_2},
	\end{align*}
	it is enough to show \eqref{tra02} for $I$ of length one, i.~e. $I=e_k$ for some $k\in\indexset$, which we now fix. Moreover, \eqref{go12} in coordinates implies that
	\begin{align}\label{tra04}
		(\Delta_\trir)_{I',I''}^{e_k} = \delta_{I'}^{e_k} \delta_{I''}^0 + \sum_{j\in\indexset}\big(\rho_\trir (A_{I'})\big)_k^j \delta_{I''}^{e_j}.
	\end{align} 
	The first summand obviously produces finitely many $I', I''$, so we focus on the second one. We need to show the strengthening of \eqref{fin01}
	\begin{align}\label{fin03}
		\mbox{for all }k\in\indexset,\quad \#\{I'\in M(\indexset),\,j\in\indexset\,|\,\big(\rho_\trir(A_{I'})\big)_k^j \neq 0\}<\infty.
	\end{align}
	If $I'$ is of a fixed length, we apply \eqref{go07} to express $A_{I'}$ as a concatenation of basis elements (i.~e. express $A_{I'}$ in a Poincaré-Birkhoff-Wittt basis); we then use the algebra morphism property of $\rho_\trir$ and apply assumption \eqref{fin01} iteratively (we need only $\length{I'}$ applications), yielding finitely many $j$ and $I'$. We still need to bound the length of $I'$ in terms of $k$. To this end, we note that by definition \eqref{go02}, \eqref{go06} the action $\pprod_\trir$ respects the grading, i.~e.
	\begin{align*}
		\rmU_\nu \pprod_\trir \rmU_{\nu'} \subset \rmU_{\nu + \nu'}.
	\end{align*}
	 As a consequence, the basis elements $\{A_{I}\}_{I\in M(\indexset)}$ are homogeneous of degree
	\begin{align*}
		\grad{I} := \sum_{i\in\indexset} I(i)\grad{i}.
	\end{align*}
	Note that the connectedness assumption implies
	\begin{align}\label{coe01}
		\grad{I} \geq \kappamin\length{I},
	\end{align}
	and $\kappamin>0$, so it is enough to obtain a bound on $\grad{I}$. By \eqref{gra01} iteratively via the algebra morphism property, it holds
	\begin{align}\label{gra01bis}
		\big(\rho_\trir(A_{I'})\big)_k^j \neq 0\,\implies \grad{k} = \grad{I} + \grad{j}.
	\end{align}
	The claim then follows since $\grad{I} \leq \grad{k} - \kappamin$.
\end{proof}	
\subsection{The group of characters and the exponential map}\label{subsec::2.3}
The set of multiplicative functionals of a Hopf algebra, also called \textit{characters}, carries a group structure, cf. \cite[Theorem 2.1.5]{Abe}: We denote it by $\mcG(L,\trir)$ $:=\textnormal{Alg}(\R[\indexset], \R)$, equipped with the convolution product and the inverse
\begin{align}\label{gro01}
	F*G := (F\otimes G)\Delta_\trir,\,\,\,F^{-1} = F\mathcal{A},
\end{align}
where $\mathcal{A}$ is the antipode of $\R[\indexset]$ (thus dual to that of $\rmU(L,\trir)$). Formally, under the duality pairing \eqref{dua02} we may write $F\in \mcG(L,\trir)$ as
\begin{align}\label{exp01}
	\sum_{I\in M(\indexset)} \bff^I A_I,
\end{align}
where $\bff^I \in \R$ are characterized by the multiplicativity property, i.~e.
\begin{align}\label{exp20}
	\bff^{e_i} = \bff_i\in\R,\,\,\,\bff^I = \prod_{i\in M(\mathcal{I})} \bff_i^{I(i)}.
\end{align}
When seen as an element of the dual Hopf algebra $\R[\indexset]^*$, to which the concatenation product and the coproduct $\cop$ of $\rmU(L,\trir)$ are extended, $F$ $\in$ $\mcG(L,\trir)$ is grouplike, i.~e.
\begin{equation}\label{grouplike}
	\cop F = F \otimes F.
\end{equation}

\medskip

According to \eqref{exp20}, each element of $\mcG(L,\trir)$ is given by a Lie series $\bff\in L^*$; moreover, by definition \eqref{go30} and after resummation (cf. e.~g. \cite[Lemma A.2]{LOT}) we have\footnote{ The last equality should be understood as $\bff^{\pprod_\trir l}$ $=$ $\bff \pprod_\trir \cdots \pprod_\trir \bff$, with $\pprod_\trir$ extended to $L^*$.}
\begin{align}\label{formal}
	\sum_{I\in M(\indexset)} \bff^I A_I &= \sum_{l\geq 0} \frac{1}{l!} \sum_{i_1,...,i_l\in\indexset} \bff_{i_1}\cdots \bff_{i_l} a_{i_1}\pprod_\trir \cdots \pprod_\trir a_{i_l} = \sum_{l\geq 0} \frac{1}{l!}\bff^{\pprod_\trir l}.
\end{align}
We thus have an exponential map associated to the pre-Lie product $\trir$,
\begin{align*}
	\begin{array}{rccl}
		\exp_{\trir} : & L^* &\longrightarrow &\mcG(L,\trir)\\
		 &\bff & \longmapsto & \exp_{\trir}(\bff) = \sum_{I\in M(\indexset)} \bff^I A_I.
	\end{array}
\end{align*}
Note that
\begin{align*}
	\prescript{}{\R[\indexset]^*}{\langle} \exp_{\trir}(\bff), a_i \rangle_{\R[\indexset]} = \bff_i = \prescript{}{L^*}{\langle} \bff , a_i \rangle_L.
\end{align*}
As a consequence, $\exp_{\trir}(\bff)$ is the multiplicative extension of $\bff\in L^*$ to the free commutative algebra $\R[\indexset]$. Thanks to $\rho_\trir$, $\exp_\trir(\bff)$ acts on Lie series as well.
\begin{lemma}\label{lem:2.9}
	For all $\bff\in L^*$, $\rho_\trir(\exp_\trir (\bff))$ $\in$ $\textnormal{End}(L^*)$.
\end{lemma}	
\begin{proof}
	From representation \eqref{exp01} we formally have
	\begin{align*}
		\rho_\trir(\exp_\trir (\bff)) = \sum_{I\in M(\indexset)} \bff^I \rho_\trir(A_I).
	\end{align*}
	This expression is actually meaningful as an endomorphism of $L^*$ thanks to the strong finiteness property \eqref{fin03}.
\end{proof}	

\medskip

By the algebra morphism property of $\rho_\trir$, now extended as in the previous lemma, 
\begin{align}\label{com11}
	\rho_\trir \big(\exp_{\trir}(\bff) * \exp_{\trir}(\bfg)\big) = \rho_\trir \big(\exp_{\trir}(\bff)\big) \rho_\trir \big(\exp_{\trir}(\bfg)\big).
\end{align}
Furthermore, the composition rule \eqref{gro01} can be expressed in terms of a transformation of the argument of $\exp_{\trir}$, via the following Baker-Campbell-Hausdorff formula.
\begin{lemma}
	For all $\bff,\bfg\in L^*$,
	\begin{align}\label{com01}
		\exp_{\trir}(\bff) * \exp_{\trir}(\bfg) = \exp_{\trir}\Big(\bff + \rho_\trir\big(\exp_{\trir} (\bff)\big)\bfg\Big).
	\end{align}
\end{lemma}	
\begin{proof}
	This is an easy consequence of \eqref{tra04}:
	\begin{align*}
		(\bff*\bfg)_i &= \langle (\bff\otimes\bfg), \Delta_\trir a_i\rangle \\ &= \langle \bff \otimes \bfg, a_i \otimes 1\rangle + \sum_{\substack{I'\in M(\indexset)\\ j\in\indexset}} \big(\rho_\trir (A_{I'})\big)_j^i \langle \bff \otimes \bfg, A^{I'}\otimes a_j \rangle \\ &= \bff_i + \sum_{\substack{I'\in M(\indexset)\\ j\in\indexset}} \bff^{I'} \big(\rho_\trir (A_{I'})\big)_j^i \bfg_j  \\ &= \Big( \bff + \rho_\trir\big(\exp_{\trir} (\bff)\big)\bfg \Big)_i.
	\end{align*}
\end{proof}	
\begin{remark}\label{rem:trivial}
	A trivial special case of this whole construction, which will play a role in a specific type of translation of rough paths (cf. Remark \ref{rem:4.14} below), is given by the trivial pre-Lie product $ a \tricom a' = 0$. In this situation, the Lie algebra is trivial and the Guin-Oudom basis corresponds exactly to the concatenation basis of $L$, as $\rmU(L,\tricom)$ $\simeq$ $\R[\indexset]$. The construction of $\rho_{\tricom}$ is now trivial, and yields $\rho_{\tricom} = \epsilon$ as maps in $\rmU(L,\tricom)$. Consequently, \eqref{go12} now takes the form
	\begin{align*}
		\proj_L (U U') = \proj_L (U) \epsilon(U') + \epsilon(U) \proj_L (U') ,
	\end{align*}	
	which lifts to the composition rule 
	\begin{align*}
		\exp_{\tricom}(\bff) * \exp_{\tricom}(\bfg) = \exp_{\tricom}(\bff + \bfg).
	\end{align*}
\end{remark}	
\subsection{Modules and comodules}
While $\rho_\trir$ associates our pre-Lie and Hopf algebras to a space of endomorphisms of $L$, in applications we sometimes want to interpret these objects as endomorphisms of a \textit{different} vector space. Many of the properties of $\rho_\trir$ also hold and yield modules and comodules.

\medskip

Let $V$ be a vector space with a countable basis $\{v_j\}_{j\in\indexsettwo}$. Let $\psi : L \to \textnormal{End}(V)$ be a Lie algebra morphism with respect to the composition commutator. By the universality property, we extend $\psi$ uniquely to an algebra morphism $\psi : \rmU(L,\triangleright) \to \textnormal{End}(V)$. This defines an action $\psi : \rmU(L,\triangleright) \otimes V \to V$, and $V$ is a left $\rmU(L,\triangleright)$-module. We now give a sufficient condition to transpose this structure based on finiteness properties.
\begin{lemma}\label{lem:com01}
	Assume that
	\begin{align}\label{fin04}
		\mbox{for all }j\in\indexsettwo \quad\#\left\{j'\in \indexsettwo, I\in M(\indexset)\,|\,(\psi(A_I))_j^{j'}\neq 0\right\}<\infty.
	\end{align}
	Then there exists a map $\psi^\dagger : V \to (\R[\indexset],\Delta_\triangleright) \otimes V$ such that $V$ is a left comodule over $(\R[\indexset],\Delta_\triangleright)$. In particular,
	\begin{align}\label{mod01}
		(\Delta_\triangleright \otimes \id)\psi^\dagger = (\id \otimes \psi^\dagger) \psi^\dagger.
	\end{align}
\end{lemma}	
\begin{proof}
	This is trivial defining in terms of the basis of $V$
	\begin{align*}
		\psi^\dagger v_j = \sum_{\substack{j'\in \indexsettwo\\ I\in M(\indexset)}} (\psi(A_I))_j^{j'} a^I \otimes v_{j'}.
	\end{align*}
	The finiteness property implies that this expression is well-defined. The module structure then transposes to a comodule; in particular \eqref{mod01} is the transposition of the algebra morphism property of $\psi$.
\end{proof}	
\begin{remark}
	A particular case of \eqref{mod01} is $\psi = \rho_\trir$:
	\begin{align*}
		(\Delta_\triangleright \otimes \id)\rho_\trir^\dagger = (\id \otimes \rho_\trir^\dagger) \rho_\trir^\dagger.
	\end{align*}	
\end{remark}	
\begin{remark}\label{rem:fin10}
Following the proof of Proposition \ref{prop:tra01}, the finiteness property \eqref{fin04} can be obtained from the case $\length{I} =1$, i.~e.
\begin{align}\label{fin05}
	\mbox{for all }j\in\indexsettwo\quad\#\left\{j'\in \indexsettwo, i\in \indexset\,|\,(\psi(a_i))_j^{j'}\neq 0\right\}<\infty,
\end{align}	
if we also have 
\begin{align}\label{fin06}
	(\psi(A_I))_j^{j'}\neq 0\,\implies \length{I}\leq C_j,
\end{align}
where $C_j$ is a constant depending only on $j\in \mathcal{J}$. Note that \eqref{fin05} means that $\psi\in \textnormal{End}(V^*)$. A sufficient condition for \eqref{fin06} can again be expressed in terms of gradings. Suppose that $V$ is graded after $\gr_V$; that the basis elements $\{v_j\}_{j\in\indexsettwo}$ are homogeneous, i.~e. $v_j \in V_{\gr_V(j)}$; that $\gr_V$ is bounded from below (but not necessarily positive!); and that $\gr_V$ is compatible with $\gr_L$, in the sense that
\begin{align*}
	(\psi(a_i))_j^{j'}\neq 0\,\implies \gr_V(j) = \gr_L(i) + \gr_V(j').
\end{align*}
By the algebra morphism property we have
\begin{align*}
	(\psi(A_I))_j^{j'}\neq 0\,\implies \gr_L(I) = \gr_V(j) - \gr_V(j') \leq \gr_V(j) + C
\end{align*}
for some value $C$ that bounds $\gr_V$ from below. Then \eqref{fin06} follows from \eqref{coe01}.
\end{remark}
	
\medskip

Given $F\in \mathcal{G}(L,\triangleright)$, we now associate the map $(F\otimes \id)\psi^\dagger \in \textnormal{End}(V)$. Together with the exponential map, this defines
\begin{align*}
	L^* \ni \bff \mapsto (\exp_\trir(\bff)\otimes \id)\psi^\dagger \in \textnormal{End}(V).
\end{align*} 
Similar to Lemma \ref{lem:2.9} we have the following.
\begin{lemma}\label{lem:2.13}
	Under \eqref{fin04}, for all $\bff\in L^*$, $\psi(\exp_\trir (\bff))\in \textnormal{End}(V^*)$.
\end{lemma}	
\section{Rough paths based on multi-indices}\label{sec::3}
In this section, we provide a description of integrals of smooth paths in terms of multi-indices, which will unveil the algebraic structure required to construct rough paths based on multi-indices.
\subsection{Multi-indices and integrals of smooth paths}
We follow the approach of \cite{OSSW}, later generalized in \cite{BL23}, and parameterize the space of jets (here \textit{controlled rough paths}, cf. \cite{Gub04}) describing solutions to \eqref{rde01} in terms of multi-indices corresponding to products of the derivatives of the nonlinear terms. What follows is a particular case of \cite{BL23}, with no polynomial contributions\footnote{ This should not be surprising to experts in rough paths and regularity structures, but let us give an argument on why it follows from \cite{BL23} that polynomials are not required. On the one hand, the space of solutions of \eqref{rde01} is parameterized by initial conditions, i.~e. constants, so only constant polynomials are required as a degree of freedom. On the other hand, the solution of \eqref{rde01} is expected to be a Hölder continuous function, so we can mod out constants as in \cite[(2.39)]{BL23} from all multi-indices except the purely polynomial. We completely remove the latter: This in the end is the reason why \eqref{rou03} below has an additional term with respect to the transformation rule \cite[(4.5)]{BL23}.}. 

\medskip

For the sake of this discussion, assume that $a_\mfl$ is analytic. We Taylor-expand it around the initial condition
\begin{align*}
	a_\mfl(y) = \sum_{k\in \N_0} \tfrac{1}{k!} \tfrac{d}{dy^k} a_\mfl(y_0) (y-y_0)^k.
\end{align*}
Rewriting \eqref{rde01} in integral form, we plug this expansion to the effect of
\begin{align}\label{mi04}
	Y_t - y_0 = \sum_{(\mfl,k)\in\mfL\times \N_0} \tfrac{1}{k!}\tfrac{d}{dy^k} a_\mfl(y_0) \int_0^t (Y_u - y_0)^k dX_u^\mfl.
\end{align}
We formally describe the solution $Y_t$ as a function of the coefficients $z_{(\mfl,k)} := \tfrac{1}{k!}\tfrac{d}{dy^k} a_\mfl(y_0)$, or rather as a Taylor-like formal power series
\begin{align}\label{mi02}
	Y_t - y_0 \approx \sum_{\beta\in M(\mfL\times \N_0)} \tfrac{1}{\beta!} \partial^\beta|_{z_{\mfl,k} = 0} (Y_t - y_0)  z^\beta,
\end{align}
where for a multi-index $\beta\in M(\mfL\times \N_0)$
\begin{equation*}
	z^\beta = \prod_{(\mfl,k)} z_{(\mfl,k)}^{\beta(\mfl,k)},\, \partial^\beta = \prod_\lk \partial_{z_\lk}^{\beta\lk},\, \beta! = \prod_\lk \beta\lk !.
\end{equation*}
Applying the Leibniz rule in \eqref{mi04}, the $\beta$-derivative satisfies the identity
\begin{align}
	&\tfrac{1}{\beta!} \partial^\beta|_{z_{\mfl,k} = 0} (Y_t - y_0)\nonumber \\
	&\quad= \sum_{\lk\in \mfL\times \N_0} \sum_{e_{(\mfl,k)} + \beta_1 + ... + \beta_k = \beta} \prod_{j=1}^k \big(\tfrac{1}{\beta_j!} \partial^{\beta_j}|_{z_{\mfl,k} = 0} (Y_u - y_0)\big) dX^\mfl_u,\label{mi05}
\end{align}
which is explicit since the r.~h.~s. depends on multi-indices of strictly smaller length (thanks to the contribution from $e_{(\mfl,k)}$).

\medskip

This expansion of course depends on the initial condition $y_0$ and the nonlinearities $\mathbf{a} = \{a_\mfl\}_{\mfl\in\mfL}$, so we now seek a parameterization for all possible inputs: We consider for each $(\mfl,k)\in\mfL\times \N_0$ the functional
\begin{align}\label{mi07}
	\z_{(\mfl,k)}[\mathbf{a},y] := \tfrac{1}{k!} \partial^k a_\mfl (y),
\end{align}
so we can express \eqref{mi04} as
\begin{align*}
	Y_t - y_0 = \sum_{(\mfl,k)} \z_{(\mfl,k)}[\mathbf{a}, y_0] \int_0^t (Y_u - y_0)^k dX_u^\mfl.
\end{align*}
Note that we may also express local increments of the solution: For $s<t$, we have
\begin{align*}
	Y_t - Y_s = \sum_{(\mfl,k)} \z_{(\mfl,k)}[\mathbf{a},Y_s] \int_s^t (Y_u - Y_s)^k dX^\mfl_u.
\end{align*}
For a fixed nonlinearity $\mathbf{a}$, this more general approach allows us to fix an arbitrary origin in the space of solutions (e.~g. at $0$) and express any $y$-evaluation of $\z_{(\mfl,k)}$ in terms of a shift in the nonlinearity:
\begin{align}\label{mi06}
	\z_{(\mfl,k)} [\mathbf{a},y] = \z_{(\mfl,k)}[\mathbf{a}(\cdot + y), 0].
\end{align}
Therefore, assuming we can express this shift as an algebraic transformation on $\z_{(\mfl,k)}$, we can still use the expansion \eqref{mi02} and the hierarchy \eqref{mi05} as an Ansatz. This will motivate our definition of rough path below.
\begin{lemma}
	Assume that, for all $\mfl\in\mfL$, $X^\mfl:[0,1]\to \R$ is smooth. There exists a unique Let $\mbfX : [0,1]^2 \to \R[[\mfL\times\N_0]]$ satisfying
	\begin{align}\label{rou01}
		\mbfX_{st\,\beta} = \sum_{(\mfl,k)} \sum_{e_{(\mfl,k)} + \beta_1 + ... + \beta_k = \beta} \int_s^t \mbfX_{su\,\beta_1}\cdots \mbfX_{su\,\beta_k} dX^\mfl_u.
	\end{align}
	Moreover, $\mbfX$ satisfies
	\begin{align}\label{rou10}
		\mbfX_{st,\beta} \not\equiv 0 \implies \left[ \beta \right] = 1,
	\end{align}	
	where
	\begin{align}\label{wei01}
		\left[ \beta \right] := \sum_{(\mfl,k)} (1-k)\beta(\mfl,k).
	\end{align}		
	\end{lemma}
\begin{proof}
	As in \eqref{mi05}, \eqref{rou01} defines a hierarchy of explicit equations. Since $X^\mfl$ is smooth, integration makes classical sense, so the r.~h.~s. is well-defined for every $\beta$. Property \eqref{rou10} follows by induction in the length of $\beta$ (cf. e.~g. \cite[Lemma 2.8]{BL23}).
\end{proof}		
\begin{definition}
\begin{align*}
	\populated &:= \{\beta\in M(\mfL\times\N_0)\,|\, \left[\beta\right] = 1\}\\
	T &:= \big\{\pi = \sum_{\beta\in\populated} \pi_\beta \z^\beta \in \R[\mfL\times\N_0]\big\},\\
	T^* &:= \big\{\pi = \sum_{\beta\in\populated} \pi_\beta \z^\beta \in \R[[\mfL\times\N_0]]\big\}.
\end{align*}
\end{definition}
Therefore,by \eqref{mi02}, a local expansion for $Y$ will take the form
\begin{equation}\label{expansion}
	Y_t - Y_s = \sum_{\beta\in\populated} \mbfX_{st\, \beta} \z^\beta \big[\mathbf{a}, Y_s\big].
\end{equation}
\medskip

The first step towards considering \eqref{mi06} as an algebraic transformation is to consider its infinitesimal generator
\begin{align*}
	\tfrac{d}{dh}\big|_{h=0} \z_{(\mfl,k)}[\mathbf{a},y + h] &= \tfrac{1}{k!} \tfrac{d}{dh}\big|_{h=0} \tfrac{d^k a}{dy^k}(y+h)\\
	&=\tfrac{1}{k!} \tfrac{d^{k+1} a}{dy^{k+1}}(y)\\  &=(k+1)\z_{(\mfl,k+1)}[\mathbf{a},y],
\end{align*}
which motivates
\begin{align}\label{defD}
	D := \sum_{(\mfl,k)} (k+1)\z_{(\mfl,k+1)}\partial_{\z_{(\mfl,k)}}\in \textnormal{Der}(\R[\mfL\times \N_0]).
\end{align}
\begin{lemma}
	$D$ $\in$ $\textnormal{Der}(\R[[\mfL\times \N_0]])$.
\end{lemma}	
\begin{proof}
	This follows from the finiteness property
	\begin{align}\label{der02}
		\mbox{for all }\beta\in M(\mfL\times \N_0)\quad\#\{\gamma\in M(\mfL\times \N_0)\,|\, D_\beta^\gamma \neq 0\}<\infty.
	\end{align}
	This property is easy to check writing explicitly
	\begin{align}\label{der01}
		D_\beta^\gamma = \sum_{(\mfl,k)} (k+1)\gamma(\mfl,k) \delta_{\beta + e_{(\mfl,k)}}^{\gamma + e_{(\mfl,k+1)}}.
	\end{align}
\end{proof}	
However, $T$ is not a sub-algebra of $\R[\mfL\times\N_0]$, so it does not make sense to talk about the Leibniz rule there. Furthermore, $D\notin \textnormal{End}(T)$, since
\begin{align}\label{der11}
	D_\beta^\gamma \neq 0\,\implies\,[\beta] = [\gamma]- 1.
\end{align}
This property is interesting though, because it implies that $T$ has a vector field-type structure: Indeed, for every $\beta,\gamma \in \populated$,
\begin{align*}
	\z^\beta D \z^\gamma \in T,
\end{align*}
so we can retain the derivation property in terms of a pre-Lie product, which we still denote by $D$.
\begin{lemma}\label{lem:3.4}
	$(T, D)$ is a $1$-connected graded pre-Lie algebra under $\grad{\beta} = \length{\beta}$.
\end{lemma}	
\begin{proof}
	It only remains to show the grading properties. Note that \eqref{der01} implies that
	\begin{align}\label{der12bis}
		D_\beta^\gamma \neq 0\,\implies\,\sum_{k\in\N_0} \beta(\mfl,k)= \sum_{k\in \N_0} \gamma(\mfl,k)\,\mbox{for all }\mfl\in\mfL,
	\end{align}
	and consequently
	\begin{align}\label{der12}
		D_\beta^\gamma \neq 0\,\implies\,\length{\beta} = \length{\gamma}.
	\end{align}
	Since the length is additive under multiplication, we have that $(T,D)$ is graded under $\length{\cdot}$. Since in addition $\length{\beta}\geq 1$ for all $\beta\in \populated$, the statement follows.	
\end{proof}	

\medskip

In a private communication, Dominique Manchon conjectured that $(T,D)$ is the free (left-)Novikov algebra over $\mfL$, which we will denote by $\mcN(\mfL)$, cf. \cite[Section 7]{DL02}. A Novikov algebra\footnote{Novikov algebras first appeared in the study of Hamiltonian operators of hydrodynamic type, cf. \cite{GD79} and \cite{DN83,BN85}.} is a pre-Lie algebra $(L,\triangleright)$ which satisfies for all $a,b,c$ $\in$ $L$
	\begin{equation}\label{nov01}
		(a\triangleright b) \triangleright c = (a \triangleright c) \triangleright b.
	\end{equation}
Identity \eqref{nov01} in $(T,D)$ is an easy consequence of the commutativity of the product of $\R[\mfL\times \N_0]$, so $(T,D)$ is Novikov. Although we will not make use of this structure, we now show that $(T,D)$ is indeed the free Novikov algebra.
\begin{lemma}\label{lem:nov01}
	$(T,D)$ is isomorphic to $\mcN(\mfL)$. 
\end{lemma}	
\begin{proof}
	We follow \cite[Section 7]{DL02}. We first note that $\mcN\mcP(\mfL)$, cf. \cite[Definition 7.7]{DL02}, is equal to the commutative polynomial algebra $\R[\mfL\times\N_0]$ if one identifies $k! \z_{(\mfl,k)}$ with the symbol $\mfl[k-1]$ in \cite[Example 7.6]{DL02}. A monomial $\z^\beta$, $\beta\in M(\mfL\times \N_0)$, is what we call a \textit{Novikov element}, cf. \cite[Definition 7.2]{DL02}. The \textit{weight} exactly corresponds to $-[\beta]$ in \eqref{wei01}, and therefore $\mcN \mcP_0 (\mfL)$ in \cite[p.~180]{DL02} corresponds to $T$. The derivation $\partial$ defined in \cite[Definition 7.7]{DL02} by $\partial (\mfl[k-1]) = \mfl[k]$ is precisely $D$ in \eqref{defD}, since $D(k! \z_{(\mfl,k)})$ $=$ $(k+1)! \z_{(\mfl,k+1)}$. Denoting by $\circ$ the pre-Lie product $p\circ q = p(\partial q)$, $(T,D)$ is therefore isomorphic to $(\mcN \mcP_0 (\mfL), \circ)$, and the latter, by \cite[Theorem 7.8]{DL02}, is isomorphic to the free Novikov algebra $\mcN(\mfL)$.
\end{proof}	

\medskip
 	
Lemma \ref{lem:3.4}, together with the finiteness property \eqref{der02}, allows us to perform the Guin-Oudom procedure (Subsection \ref{subsec::2.1}) and the transposition (Subsection \ref{subsec::2.2}), leading to the group $\mathcal{G}(T,D)$ of multiplicative functionals of $\R[\populated]$ as described in Subsection \ref{subsec::2.3}. We will see our rough paths as elements of $T^*$, extended multiplicatively via $\exp_{D}$ to the group $\mathcal{G}(T,D)$. The main goal now is to show that \eqref{rou01}, in the smooth case, does indeed define an algebraic smooth rough path; for this we need the composition rule
\begin{align}\label{chen01}
	\exp_D (\mbfX_{sr})* \exp_D (\mbfX_{rt}) = \exp_D (\mbfX_{st}),
\end{align}
also known as Chen's relation. By \eqref{com01}, \eqref{chen01} is reduced to
\begin{align}\label{rou03}
	\mbfX_{st} = \mbfX_{sr} + \rho_D\big(\exp_D(\mbfX_{sr})\big)\mbfX_{rt}.
\end{align}
Let us first write an equivalent formulation of \eqref{rou01} in terms of exponential maps (cf. \cite[(3.50)]{BL23}).
\begin{lemma}\label{lem:3.5}\eqref{rou01} is equivalent to 
	\begin{align}\label{rou02}
		\mbfX_{st} = \sum_{\mfl\in\mfL} \int_s^t \rho_D\big(\exp_{D}(\mbfX_{su})\big)\z_{(\mfl,0)} dX_u^\mfl.
	\end{align}	
\end{lemma}	
\begin{proof}
	We first claim that, as endomorphisms, for every $\beta_1,...,\beta_k\in\populated$
	\begin{align}\label{goD}
		\rho_D (\z^{\beta_1}\pprod_D ... \pprod_D \z^{\beta_k}) = \z^{\beta_1}\cdots\z^{\beta_k} D^k.
	\end{align}
	This in turn follows from the following identity, again as endomorphisms: For $U\in \rmU(T,D)$ and $\beta\in \populated$,
	\begin{equation*}
		\rho_D (U\pprod_D \z^\beta) = \z^\beta \rho_D (U) D.
	\end{equation*}
	The latter is shown inductively in $U$; for $U=1$ it is trivial, whereas assuming it for $U$ and given an arbitrary element $b\in T$ we have from \eqref{go06} and the algebra morphism property of $\rho_D$
	\begin{align*}
		\rho_D(bU\pprod_D \z^\beta) &= \rho_D(b)\z^\beta \rho_D(U) D - (bD\z^\beta) \rho_D (U) D \\
		& = \z^\beta \rho_D(b)\rho_D(U) D\\
		& = \z^\beta \rho_D(bU) D.
	\end{align*}
	As a consequence of \eqref{goD} via definition \eqref{defD}, it holds
	\begin{align*}
		\tfrac{1}{k!}\rho_D (\z^{\beta_1}\pprod_D ... \pprod_D \z^{\beta_k})\z_{(\mfl,0)} = \z^{\beta_1}\cdots \z^{\beta_k} \z_{(\mfl,k)}.
	\end{align*}
	Combining this with the representation \eqref{formal},
	\begin{align}\label{com31}
	\rho_D\big(\exp_{D}(\mbfX_{su})\big)\z_{(\mfl,0)} = \sum_{k\geq 0} \sum_{\beta_1,...,\beta_{k}\in\populated} \mbfX_{su\,\beta_1}\cdots \mbfX_{su\,\beta_k} \z^{\beta_1}\cdots \z^{\beta_k} \z_{(\mfl,k)}.
	\end{align}
	The $\beta$-contribution of this sum gives the expression on the r.~h.~s. of \eqref{rou01}.
\end{proof}	
\begin{lemma}\label{lem:3.7}
	$\mbfX : [0,1]^2 \to T^*$ defined via \eqref{rou02} satisfies \eqref{rou03}.
\end{lemma}	
\begin{proof}
	We write \eqref{rou02} as
	\begin{align*}
		\mbfX_{st} = \mbfX_{sr} + \sum_{\mfl\in\mfL} \int_r^t \rho_D\big(\exp_{D}(\mbfX_{su})\big)\z_{(\mfl,0)} dX_u^\mfl.
	\end{align*}
	Therefore, \eqref{rou03} follows from
	\begin{align*}
		\rho_D\big(\exp_D(\mbfX_{sr})\big)\rho_D\big(\exp_{D}(\mbfX_{ru})\big)\z_{(\mfl,0)}
		=  \rho_D\big(\exp_{D}(\mbfX_{su})\big)\z_{(\mfl,0)}.
	\end{align*}
	In turn, \eqref{com11} reduces this expression to
	\begin{align*}
		\rho_D\big(\exp_D(\mbfX_{sr})*\exp_{D}(\mbfX_{ru})\big)\z_{(\mfl,0)}
		= \rho_D\big(\exp_{D}(\mbfX_{su})\big)\z_{(\mfl,0)},
	\end{align*}
	which by \eqref{com01} further reduces to
	\begin{align}
		\rho_D\big(\exp_D\big(\mbfX_{sr} + \exp_{D}(\mbfX_{sr})\mbfX_{ru}\big)\big)\z_{(\mfl,0)}=   \rho_D\big(\exp_{D}(\mbfX_{su})\big)\z_{(\mfl,0)}.\label{com21}
	\end{align}
	By the previous proof and the graded structure, we note that $\big(\rho_D\big(\exp_{D}(\bff)\big)\z_{(\mfl,0)}\big)_\beta$ only depends on $\bff_{\beta'}$ with $\length{\beta'} \leq \length{\beta} - 1$ (since $\length{e_{(\mfl,0)}} = 1$). Therefore, \eqref{com21} for a fixed $\beta$ follows by induction in  $\length{\cdot}$. The base case $\length{\beta} = 1$ is trivial since \eqref{com31} implies $\big(\rho_D (\exp_D (\bff))\z_{(\mfl,0)}\big)_{e_{(\mfl',0)}}$ $=$ $\delta_\mfl^{\mfl'}$, and thus the $e_{(\mfl,0)}$-projection of \eqref{com21} is simply
	\begin{align*}
	X_t^\mfl - X_r^\mfl = X_t^\mfl - X_r^\mfl.
	\end{align*}
 \end{proof}	
\begin{remark}\label{rem:srp01}
	Defining the set of iterated integrals as the solution to a differential equation on a Lie algebra/group is not a novelty, and can be traced back to signatures in the geometric case\footnote{ i.~e. when integration by parts holds.} (cf. e.~g. \cite[(2.17)]{Lyons98}). In the general Hopf-algebraic setup, we shall compare \eqref{rou02} with \cite[Subsection 4.1]{BFPP}, and more precisely \cite[Theorem 4.2]{BFPP}. To this end let us interpret the driver $X$ as a Lie-valued path
	\begin{equation*}
		\begin{array}{rccl}
			X : &[0,1] &\longrightarrow &T \\
			 & t & \longmapsto & X_t := \sum_{\mfl\in \mfL} X_t^\mfl \z_{(\mfl,0)}.
		\end{array} 
	\end{equation*}
	Let us also denote $\mathbf{X} := \exp_{D}(\mbfX)$. By \eqref{formal} and the Leibniz rule, combined with \eqref{rou02} (in differential form)
	\begin{equation}\label{shrp01}
		d \mathbf{X}_{st} = \mathbf{X}_{st} \pprod_D d \mbfX_{st} = \mathbf{X}_{st} \pprod_D \rho_D (\mathbf{X}_{st})d X_t.
	\end{equation} 
	Now since $\mathbf{X}$ is grouplike, cf. \eqref{grouplike}, property \eqref{go07} allows to rewrite \eqref{shrp01} as
	\begin{equation}\label{shrp02}
		d \mathbf{X}_{st} = \mathbf{X}_{st} dX_t,
	\end{equation}
	where the r.~h.~s. is interpreted as the concatenation product in $\R[\populated]^*$, which is extended from $\rmU (T,D)$. Since the Guin-Oudom procedure identifies the concatenation product with the Grossman-Larson product (cf. Remark \ref{rem:glp}), $\mathbf{X}$ solves the differential equation postulated in \cite[proof of Theorem 4.2]{BFPP}.
\end{remark}	

\subsection{Rough paths over the multi-index Hopf algebra}\label{subsec:3.2}
We are now ready to provide a definition of rough path based on multi-indices. To this end, we first introduce some truncated structures. Recall from Lemma \ref{lem:3.4} that $(T,D)$ is graded and $1$-connected: By Proposition \eqref{tra01}, $(\R[\mathcal{T}],\Delta_D)$ is a connected graded Hopf algebra and the dual of the Guin-Oudom basis over $\{\z^\beta\}_{\beta\in\mathcal{T}}$ consists of homogeneous elements, that is,
\begin{equation*}
	\R[\mathcal{T}] = \bigoplus_{n=0}^\infty \bigoplus_{\substack{I\in M(\mathcal{T})\\ \deg(I)=n}} \lspan\{z^I\};
\end{equation*}
here $z^I = \prod_{\beta\in\mathcal{T}} z_\beta^{I(\beta)}$ is dual to the $I$-th Guin-Oudom basis element, which we will denote by $\mathsf{Z}_I$, and $\deg(I) = \sum_{\beta\in \populated} I(\beta) \length{\beta}$. 
\begin{definition}
	Let $N\in \N_0$ and let
		\begin{equation*}
			\R[\mathcal{T}]^{(N)} = \bigoplus_{n=0}^N \bigoplus_{\substack{I\in M(\mathcal{T})\\ \deg(I)=n}} \lspan\{z^I\}
		\end{equation*}
	We define $\mathcal{G}^{(N)}(T,D)$ as the set of $N$-truncated characters, i.~e. the set of functionals $F\in (\R[\mathcal{T}]^{(N)})^*$ such that for all $z,z'\in \R[\mathcal{T}]^{(N)}$ with $zz'\in \R[\mathcal{T}]^{(N)}$
		\begin{equation*}
			\langle F, zz'\rangle = \langle F,z\rangle \langle F,z'\rangle.
		\end{equation*}
\end{definition}	
The truncated multiplicativity property implies, as in Subsection \ref{subsec::2.3}, that $F$ can be expressed as
\begin{equation}\label{tru03}
	\sum_{\substack{I\in M(\mathcal{T})\\ \deg(I)\leq N}} \mathbf{f}^I \mathsf{Z}_I
\end{equation}
where for all $I\in M(\mathcal{T})$, $\deg(I)\leq N$, 
\begin{equation}\label{tru04}
	\mathbf{f}^{I} = \prod_{\beta\in \mathcal{T}} \mathbf{f}_\beta^{I(\beta)}.
\end{equation}
Thus, $F\in \mathcal{G}^{(N)}(T,D)$ may be identified with the truncated $D$-exponential of an element $\mathbf{f}\in T^{(N)}$, i.~e.
\begin{align*}
	\begin{array}{rccl}
		\exp^{(N)}_{D} : & T^{(N)} &\longrightarrow &\mcG^{(N)}(T,D)\\
		&\bff & \longmapsto & \exp^{(N)}_{D}(\bff) = \sum_{\substack{I\in M(\mathcal{T}) \\ \deg(I)\leq N}} \bff^I \mathsf{Z}_I,
	\end{array}
\end{align*}
where
\begin{equation}\label{TN}
	T^{(N)} = \lspan_{\beta\in\populated} \{\z^\beta\,|\, \length{\beta}\leq N\}.
\end{equation}
\begin{definition}\label{def:rp}
	Let $\alpha\in (0,1)$, and let $N:= \lfloor \alpha^{-1} \rfloor$. A $\alpha$-rough path over $\populated$ is a map $\mbfX: [0,1]^2 \to T^{(N)}$ such that for every $\beta\in \populated$ with $\length{\beta}\leq N$ and every $0\leq s < u < t \leq 1$
	\begin{equation}\label{tru01}
		\mbfX_{st\,\beta} = \mbfX_{su\,\beta} + \sum_{\gamma\in \populated} \big(\rho_D\big(\exp^{(N)}_{D}(\mbfX_{su})\big)\big)_\beta^\gamma \mbfX_{ut\,\gamma} 
	\end{equation}
	and
	\begin{align}\label{tru06}
		\sup_{0\leq s<t\leq 1}\frac{|\mbfX_{st\,\beta}|}{|t-s|^{\alpha\length{\beta}}} <\infty.
	\end{align}
\end{definition}	
\begin{remark}
	\mbox{}
	\begin{enumerate}[label=(\roman*)]
		\item Note from \eqref{gra01bis} that
		\begin{equation}\label{tru08}
			\big(\rho_D (\mathsf{Z}_{I})\big)_\beta^\gamma \neq 0\implies \length{\beta} = \deg(I) + \length{\gamma}.
		\end{equation}
	In particular, this means that $\length{\tilde{\beta}} < \length{\beta}$ for all $\tilde{\beta}$ with $I(\tilde{\beta})\neq 0$, but also $\deg(I) < \length{\beta}$. Therefore, the $\beta$-projection of \eqref{rou03} is equivalent to \eqref{tru01}, and the truncation of the exponential is effectively irrelevant.
		\item The graded structure of $(\R[\mathcal{T}],\Delta_D)$ implies that \eqref{tru01} yields the truncated Chen's relation
		\begin{equation}\label{tru02}
			\exp_D^{(N)}(\mbfX_{su}) * \exp_D^{(N)}(\mbfX_{ut}) = \exp_D^{(N)}(\mbfX_{st}).
		\end{equation}
		In addition, the representation given by \eqref{tru03} and \eqref{tru04} implies
		\begin{equation}\label{tru05}
			\big| \langle \exp_D^{(N)}(\mbfX_{st}), z^I \rangle \big| = \tfrac{1}{I!} \prod_{\beta\in \mathcal{T}} |\mbfX_{st\, \beta}|^{I(\beta)} \lesssim |t-s|^{\alpha \deg(I)}.
		\end{equation}
		The (truncated) multiplicativity property, Chen's relation \eqref{tru02} and the analytic constraint \eqref{tru05} are typically the defining properties of a Hopf-algebraic rough path, cf. e.~g. \cite[Definition 4.3]{BZ21}. Our definition puts the focus on the Lie algebra side, but is of course equivalent.
	\end{enumerate}
\end{remark}	
\begin{remark}
	In line with regularity structures, one could of course use different scalings for different noise components $\mfl\in\mfL$ and impose an analytic condition of the form
	 \begin{align*}
	 	\sup_{0\leq s<t\leq 1}\frac{|\mbfX_{st\,\beta}|}{|t-s|^{|\beta|}} <\infty,
	 \end{align*}
 	where
 	\begin{align*}
 		|\beta| = \sum_{\mfl\in\mfL} \alpha_\mfl \sum_{k\in\N_0} \beta(\mfl,k).
 	\end{align*}
 	For example, if we think of $\mfl = 0$ as a drift term, one chooses
 	\begin{align*}
 		\alpha_\mfl = \left\{\begin{array}{ll}
 			1 & \mbox{if }\mfl = 0,\\
 			\alpha & \mbox{otherwise.}
 		\end{array}\right.
 	\end{align*}	
 	This will become important in Theorem \ref{th:trp} (iii) and Remark \ref{rem:4.18} below.
\end{remark}	

\medskip

An important result for rough paths is the so-called extension theorem \cite[Theorem 2.2.1]{Lyons98}. In our context, it corresponds to a particular case of the extension theorem for Hopf-algebraic rough paths \cite[Proposition 4.5]{BZ21}; for completeness, we provide a short proof which keeps the focus on the Lie-algebraic prespective.
\begin{proposition}
	Let $\mbfX^{(N)} : [0,1]^2 \to T^{(N)}$ be a $\alpha$-rough path, $N=\lfloor \alpha^{-1} \rfloor$. There exists a unique extension $\mbfX:[0,1]^2 \to T^*$ satisfying \eqref{rou03} and \eqref{tru06} for all $\beta\in \mathcal{T}$.
\end{proposition}	
\begin{proof}
	The proof goes by induction. For every $M>N$, we will denote the $M$-th extension $\mbfX^{(M)}:[0,1]^2\to T^{(M)}$ satisfying \eqref{tru01} and \eqref{tru06} for all $\beta$ with $\length{\beta}\leq M$, and accordingly $\mathbf{X}^{(M)}$ $:=$ $\exp_{D}^{(M)}(\mbfX^{(M)})$. Assuming it has been constructed for all $M'\leq M$, we will build it for $M+1$ (note that this covers both the base case $M=N$ and the induction step). Each step is an application of the sewing lemma \cite[Proposition 1]{Gub04}.
	
	\medskip
	
	Fix $\beta$ with $\length{\beta} = M+1$. Define $F:[0,1]^2 \to \R$ by
	\begin{equation*}
		F_{st} := \sum_{\gamma\in \mathcal{T}} \big(\rho_D (\mathbf{X}_{0s}^{(M)}) - \id\big)_\beta^\gamma \mbfX_{st\, \gamma}^{(M)}. 
	\end{equation*}
	Denote $\delta F_{sut} := F_{st} - F_{su} - F_{ut}$. Then by \eqref{tru01}
	\begin{align*}
		&\delta F_{sut} \\
		&\quad= \sum_{\gamma\in \mathcal{T}} \big(\rho_D (\mathbf{X}_{0s}^{(M)}) - \id\big)_\beta^\gamma (\mbfX_{st\, \gamma}^{(M)} - \mbfX_{ut\,\gamma}^{(M)}) - \big(\rho_D  (\mathbf{X}_{0u}^{(M)}) - \id\big)_\beta^\gamma \mbfX_{ut\, \gamma}^{(M)}\\
		&\quad = \sum_{\gamma\in \mathcal{T}} \big(\rho_D (\mathbf{X}_{0s}^{(M)})\rho_D (\mathbf{X}_{su}^{(M)}) - \rho_D (\mathbf{X}_{su}^{(M)}) -\rho_D (\mathbf{X}_{0u}^{(M)}) + \id\big)_\beta^\gamma \mbfX_{ut\, \gamma}^{(M)}\\
		&\quad = - \sum_{\gamma\in \mathcal{T}} \big(\rho_D (\mathbf{X}_{su}^{(M)}) - \id\big)_\beta^\gamma \mbfX_{ut\, \gamma}^{(M)}.
	\end{align*}
	By strict triangularity inherited from \eqref{tru08} (note that $-\id$ is removing the $I=0$ contribution of $\rho_D(\mathbf{X}_{su}^{(M)})$), $\length{\gamma}< \length{\beta}$, so the induction hypothesis yields $| \mbfX_{ut\, \gamma}^{(M)}|\lesssim |t-u|^{\alpha \length{\gamma}}$. On the other hand, representation \eqref{tru03} next to the estimate \eqref{tru05}, which holds by induction, yields for every $\gamma \in \mathcal{T}$
	\begin{equation}
		|\big(\rho_D (\mathbf{X}_{su}^{(M)}) - \id\big)_\beta^\gamma| \leq \sum_{I\in M(\mathcal{T})} \tfrac{1}{I!} (\rho_D(\mathsf{Z}_I))_\beta^\gamma |u-s|^{\alpha \deg(I)},
	\end{equation}
	(the sum over $I$ is effectively finite), which in turn implies by \eqref{tru08}
	\begin{equation*}
		|\big(\rho_D (\mathbf{X}_{su}^{(M)}) - \id\big)_\beta^\gamma| \leq \big(\sum_{I\in M(\mathcal{T})} \tfrac{1}{I!} (D^{\length{I}})_\beta^\gamma\big) |u-s|^{\alpha\length{\beta} - \alpha\length{\gamma}}.
	\end{equation*}
	Since $\alpha\length{\beta} = \alpha (M+1) \geq \alpha (N + 1) >1$, by the sewing lemma \cite[Proposition 1]{Gub04} there exists a unique $\mbfX_{\beta}^{(M+1)} :[0,1]^2 \to \R$ such that 
	\begin{equation*}
		\mbfX_{st\,\beta}^{(M+1)} = \mbfX_{su\,\beta}^{(M+1)} + \mbfX_{ut\,\beta}^{(M+1)} - \delta F_{sut},
	\end{equation*}
	and 
	\begin{equation*}
		\sup_{0\leq s < t \leq 1} \frac{|\mbfX_{st\,\beta}^{(M+1)}|}{|t-s|^{\alpha \length{\beta}}} <\infty.
	\end{equation*}	
	Setting $\mbfX_{\beta'}^{(M+1)} = \mbfX_{\beta'}^{(M)}$ for all $\length{\beta}\leq M$, \eqref{tru01} holds for the new $\beta$ and thus $(M+1)$-th extension is determined.
\end{proof}	
\begin{remark}
	The more general Lyons-Victoir extension theorem \cite[Theorem 1]{LV07} also holds, cf. \cite[Theorem 4.4]{BZ21}.
\end{remark}	
In the sequel, we will always assume that a rough path has been uniquely extended to all levels, and work with the non-truncated object $\mbfX : [0,1]^2 \to T^*$ satisfying \eqref{rou03} and \eqref{tru06} for all $\beta\in \populated$.

\section{Algebraic renormalization of rough paths based on multi-indices}\label{sec::4}
In this section we study transformations of the equation related to admissible counterterms in the sense of \cite[Subsection 3.4]{BL23}. Like with the shifts in the space of solutions, we study their infinitesimal generators (Subsection \ref{subsec::4.1}) and use them to build groups of endomorphisms of $T^*$ (Subsection \ref{subsec::tran02}). These allow to define translated rough paths in the sense of \cite{BCFP} (Subsection \ref{subsec::4.3}).
\subsection{Infinitesimal generators and a multi pre-Lie algebra}\label{subsec::4.1}
The type of transformations we are interested in take the form of local counterterms
\begin{align}\label{ren01}
	a_\mfl(y) \mapsto a_\mfl(y) + c_\mfl[\mathbf{a},y],
\end{align}
where $c\in T$. It follows from \cite[Subsection 3.4]{BL23} that, via \eqref{rou02}, one can express the rough path equations for \eqref{ren01} as
\begin{align*}
	\mbfX_{st} = \sum_{\mfl\in\mfL} \int_s^t \rho_D(\exp_D(\mbfX_{su})) (\z_{(\mfl, 0)} + c_\mfl) dX_u^\mfl.
\end{align*}
Thus, these transformations of $T^*$ take the form of translations
\begin{align*}
	\z_{(\mfl,0)} \mapsto \z_{(\mfl,0)} + c_\mfl,
\end{align*}	
which coincide with the tree-based case \cite[Definition 14]{BCFP}. However, unlike \cite[Remark 4.13]{BFPP} and the references therein, we will not appeal to the universal extension from the free structure (in this case the free Novikov algebra); instead, we will study the algebraic properties of this new transformation and construct translation maps ``explicitly".

\medskip

Let us first study the infinitesimal generators of shifts in the (affine) space of equations. We express $c_\mfl$ as its Taylor expansion around a point $y$. Fix a monomial contribution $k\in\N_0$ and look at
\begin{align*}
	a_\mfl \mapsto a_\mfl + t(\cdot - y)^k.
\end{align*}
Then
\begin{align*}
	\z_{(\mfl',k')} [\mathbf{a} + t e_\mfl(\cdot-y)^k, y] = \z_{(\mfl',k')}[\mathbf{a},y] + t \delta_{(\mfl',k')}^{(\mfl,k)},
\end{align*}
and therefore the infinitesimal generator takes the form $\partial_{(\mfl,k)}$ $:=$ $\partial_{\z_{(\mfl,k)}}$. It is easy to see that $\partial_{(\mfl,k)}$ is not closed in $T$; indeed,
\begin{align*}
	(\partial_{(\mfl,k)})_\beta^\gamma \neq 0\,\implies\, [\beta] = [\gamma] + k - 1.
\end{align*}	
We are however not interested in general shifts, but only those \textit{generated by the equation}. A natural approach would be, as with the shifts in solution space, to compose this derivation with a multiplication with an element $\pi\in \R[\mfL\times \N_0]$, which then takes the form of the infinitesimal generator of
\begin{align*}
	a_\mfl \mapsto a_\mfl + t\pi(\cdot - y)^k.
\end{align*}	
Still, this is too general, as what we would like to see is a \textit{local counterterm}, i.~e. a functional of $\mathbf{a}$ evaluated at the point $y$. It should therefore take the form
\begin{align*}
	t\sum_{k\in\N_0} \pi^{(k)}(\cdot-y)^k,
\end{align*}
where
\begin{align*}
	\pi^{(k)} = \tfrac{1}{k!} \partial_y^k \pi (\mathbf{a},y).
\end{align*}
We finally restrict to $\pi(\mathbf{a},y) = \z^\gamma[\mathbf{a},y]$, using the structure we have already defined in \eqref{mi07}, and note that $\partial_y^k \z^\gamma[\mathbf{a},y] = D^k \z^\gamma[\mathbf{a},y]$. We thus have the transformation
\begin{align}\label{neg01}
	\sum_{k\in\N_0} \big(\tfrac{1}{k!}D^k \z^\gamma\big) \partial_{(\mfl,k)}.
\end{align}
Note that this map is well-defined as an endomorphism of $\R[\mfL\times \N_0]$, since when applied to $\z^\beta$ the sum over $k$ is effectively finite (as there are only finitely many $(\mfl,k)$ such that $\beta(\mfl,k)\neq 0$).
\begin{lemma}\label{lem:ren01}
	For every $\mfl\in \mfL$ and every $\gamma\in M(\mfL\times \N_0)$, $\sum_{k\in\N_0} \big(\tfrac{1}{k!}D^k \z^\gamma\big) \partial_{(\mfl,k)}$ $\in$ $\textnormal{End}(\R[[\mfL\times \N_0]])$.
\end{lemma}	
\begin{proof}
	It is enough to show the finiteness property
	\begin{align*}
		& \text{for all } \beta'\in M(\mfL \times \N_0),\nonumber\\
		& \# \Big\{\gamma'\in M(\mfL\times \N_0)\,|\Big(\sum_{k\in\N_0} \big(\tfrac{1}{k!}D^{k} \z^{\gamma}\big) \partial_{(\mfl,k)}\Big)_{\beta'}^{\gamma'}\neq 0\Big\}<\infty.
	\end{align*}
	Let us fix $k$ in the sum. Note that
	\begin{align}\label{tra10}
		\Big(\big(\tfrac{1}{k!}D^{k} \z^{\gamma}\big) \partial_{(\mfl,k)}\Big)_{\beta'}^{\gamma'} = \sum_{\beta''}\big(\tfrac{1}{k!}D^{k} \z^{\gamma}\big)_{\beta'-\beta''} \gamma'(\mfl,k) \delta_{\beta'' + e_{(\mfl,k)}}^{\gamma'}.
	\end{align}
	Fixing $\beta'$, there are finitely many $\beta''\leq \beta$. Following \eqref{der11} and \eqref{der12}, it holds
\begin{align*}
	(D^k)_\beta^\gamma \neq 0\,\implies \length{\beta} - [\beta] = \length{\gamma} - [\gamma] + k,
\end{align*}
	so for each $\beta' - \beta''$ and a fixed $\gamma$, $k$ is completely determined. Finitely many $(\beta'',k)$ yield finitely many $\gamma'$.
\end{proof}	
Actually, from \eqref{tra10} and using in addition \eqref{der12bis} we obtain 
\begin{align}
	&\Big(\sum_{k\in\N_0} \big(\tfrac{1}{k'!}D^{k'} \z^{\gamma'}\big) \partial_{(\mfl',k')}\Big)_\beta^\gamma \neq 0\,\implies\nonumber\\
	&\left\{\begin{array}{l}
		\sum_{k\in\N_0} \beta(\mfl,k) = \sum_{k\in\N_0} \gamma(\mfl,k) + \sum_{k\in\N_0} \gamma'(\mfl,k) - \delta_\mfl^{\mfl'},\,\mfl\in\mfL,\\
		\length{\beta} = \length{\gamma'} -1 + \length{\gamma},\\
		\left[\beta\right] = \left[\gamma'\right] - 1 + \left[\gamma\right].
	\end{array}\right.\label{tra11}
\end{align}
In particular, further restricting to $\gamma\in T$, we get $[\cdot]$-preserving maps. Restricting to these multi-indices, on the other hand, is natural from the viewpoint of applications, since the translations that we are interested in are the ones generated by components of the rough path (in analogy with substitution of B-series).

\medskip

Since \eqref{neg01} has a vector field-type structure (derivation and multiplication), the pre-Lie property is guaranteed for the map
\begin{align*}
	&\Big(\sum_{k\in\N_0} \big(\tfrac{1}{k!}D^k \z^\gamma\big) \partial_{(\mfl,k)}, \sum_{k'\in\N_0} \big(\tfrac{1}{k'!}D^{k'} \z^{\gamma'}\big) \partial_{(\mfl',k')}\Big) \\
	&\mapsto \sum_{k'\in\N_0} \Big(\sum_{k\in\N_0} \big(\tfrac{1}{k!}D^k \z^\gamma\big) \partial_{(\mfl,k)}\tfrac{1}{k'!}D^{k'} \z^{\gamma'}\Big) \partial_{(\mfl',k')}.
\end{align*}	
We actually have the following.
\begin{lemma}
\begin{align}
	&\sum_{k'\in\N_0} \Big(\sum_{k\in\N_0} \big(\tfrac{1}{k!}D^k \z^\gamma\big) \partial_{(\mfl,k)}\tfrac{1}{k'!}D^{k'} \z^{\gamma'}\Big) \partial_{(\mfl',k')} \nonumber\\ &\quad= \sum_{k'\in\N_0} \Big(\tfrac{1}{k'!}D^{k'}\Big(\sum_{k\in\N_0} \big(\tfrac{1}{k!}D^k \z^\gamma\big) \partial_{(\mfl,k)} \z^{\gamma'}\Big)\Big) \partial_{(\mfl',k')}.\label{neg05}
\end{align}
In particular, if $[\gamma]=[\gamma']=1$, the r.~h.~s. is a linear combination of elements \eqref{neg01} with $[\cdot] = 1$, so the structure is closed.
\end{lemma}
\begin{proof}
\eqref{neg05} can be reduced to showing the commutativity as endomorphisms
\begin{align}\label{neg06}
	\sum_{k\in\N_0} \big(\tfrac{1}{k!}D^k \z^\gamma\big) \partial_{(\mfl,k)} D = D \sum_{k\in\N_0} \big(\tfrac{1}{k!}D^k \z^\gamma\big) \partial_{(\mfl,k)}.
\end{align}	
We first note that
\begin{align*}
	\partial_{(\mfl,k)} D = D\partial_{(\mfl,k)} + k \partial_{(\mfl,k-1)},
\end{align*}
thus for any element $\pi\in\R[[\mfL\times \N_0]]$, which acts as an endomorphism by multiplication, it holds by the Leibniz rule
\begin{align*}
	\pi\partial_{(\mfl,k)} D = D \pi \partial_{(\mfl,k)} - (D\pi) \partial_{(\mfl,k)} + k \pi \partial_{(\mfl,k-1)}.
\end{align*}
When choosing $\pi = \tfrac{1}{k!}D^k\z^\gamma$ and summing in $k$, the last two terms cancel out, and thus \eqref{neg06} follows.
\end{proof}
Thanks to \eqref{neg05}, we may endow $T$ with a multi pre-Lie structure.
\begin{definition}
	For every $\mfl\in\mfL$, define $\btl : T\times T \to T$ as
	\begin{align}\label{neg04}
		\z^\beta \btl \z^\gamma := \sum_{\gamma'} \Big(\sum_{k\in\N_0} \big(\tfrac{1}{k!}D^k \z^\beta\big) \partial_{(\mfl,k)}\Big)_{\gamma'}^\gamma \z^{\gamma'}.
	\end{align}	
\end{definition}
\begin{corollary}
	For every $\mfl\in\mfL$, $\btl$ is a pre-Lie product in $T$, and $(T,\{\btl\}_{\mfl\in\mfL})$ is a multi pre-Lie algebra (cf. e.~g. \cite[Proposition 4.21]{BCCH} or \cite[Remark 2.10]{BL23}).
\end{corollary}		
We consider now a new set of symbols indexed by $\renlie :=\populated\times\mfL$, which we will denote by $\rr^{(\gamma,\mfl)}$, and let
\begin{align}\label{basisR}
	R := \lspan_{(\gamma,\mfl)\in\renlie} \{\rr^{(\gamma,\mfl)}\}
\end{align}
be equipped with $\trineg: R\times R \to R$ given by
\begin{align}\label{neg07}
	\rr^{(\gamma,\mfl)}\trineg \rr^{(\gamma',\mfl')} = \sum_{\gamma'} \Big(\sum_{k\in\N_0} \big(\tfrac{1}{k!}D^k \z^\gamma\big) \partial_{(\mfl,k)}\Big)_{\beta'}^{\gamma'} \rr^{(\beta',\mfl')}.
\end{align}
\begin{corollary}
	$(R,\trineg)$ is a pre-Lie algebra.
\end{corollary}	
Apart from the map $\rho_\trineg : R \to \textnormal{End}(R)$ given by \eqref{go01}, we use \eqref{neg04} to build the linear map $\rhotil_\trineg : R \to \textnormal{End}(T)$ as
\begin{align*}
	\rhotil_\trineg (\rr^{(\gamma,\mfl)}) := \z^\gamma \trineg_{\mfl},
\end{align*}
which is obviously a Lie algebra morphism with respect to composition of endomorphisms and uniquely extends to an algebra morphism $\rmU(R,\trineg)$ $\to$ $\textrm{End}(T)$. By Lemma \ref{lem:ren01}, $\rhotil_\trineg (\rr^{(\gamma,\mfl)})$ $\in$ $\textnormal{End}(T^*)$.
\begin{lemma}
	For every $U$ $\in$ $\rmU (R,\trineg)$,
	\begin{align}\label{pre02bis}
		\rhotil_\trineg (U) \z_{(\mfl,0)} = \epsilon(U) \z_{(\mfl,0)} + \proj_\mfl (U),
	\end{align}
	where $\epsilon : \rmU(R,\trineg) \to \R$ is the counit and $\proj_\mfl : \rmU (R,\trineg) \to T$ is given by $\proj_\mfl = \textnormal{proj}_\mfl \circ \proj_R$ with $\proj_R : \rmU (R,\trineg) \to R$ the rank-one projection with respect to the Guin-Oudom basis (cf. Lemma \ref{lem:2.6} (ii)) and $\textnormal{proj}_\mfl : R \to T$ given in coordinates by
	\begin{equation*}
		\textnormal{proj}_\mfl (\rr^{(\gamma',\mfl')}) = \z^{\gamma'} \delta_\mfl^{\mfl'}.
	\end{equation*}  
	Moreover, for every $U \in \rmU(R,\trineg)$,
	\begin{align}\label{pre02}
		\rhotil_\trineg (U) (\pi D \pi') = \sum_{(U)} (\rhotil_\trineg (U_{(1)}) \pi) D (\rhotil_\trineg (U_{(2)})).
	\end{align}
\end{lemma}	
\begin{proof}
	We start showing \eqref{pre02bis}. It is enough to show it for $U = \mathsf{R}_I$, a Guin-Oudom basis element over the basis \eqref{basisR}. Then \eqref{pre02bis} follows from the identity
	\begin{align}
		&\rhotil_\trineg \big(\rr^{(\gamma_1,\mfl_1)}\pprod_\trineg \cdots \pprod_\trineg \rr^{(\gamma_l,\mfl_l)}\big)\nonumber\\
		&\quad=  \sum_{k_1,...,k_l \in \N_0} \big(\tfrac{1}{k_1 !} D^{k_1}\z^{\gamma_1}\big) \cdots  \big(\tfrac{1}{k_l !} D^{k_l}\z^{\gamma_l}\big) \partial_{(\mfl_l,k_l)} \cdots \partial_{(\mfl_1,k_1)}.\label{goD02}
	\end{align}
	As for \eqref{goD}, it is enough to show that for all $U\in \rmU(R,\trineg)$ and $(\gamma,\mfl)\in \renlie$
	\begin{equation*}
		\rhotil_\trineg (U\pprod_\trineg \rr^{(\gamma,\mfl)}) = \sum_{k\in\N_0} \big(\tfrac{1}{k!} D^k \z^\gamma\big) \rhotil_\trineg (U) \partial_{(\mfl,k)}.
	\end{equation*}
	This is shown inductively in $U$; for $U=1$ is trivial, whereas assuming it for $U$ and given an arbitrary $b\in R$, equation \eqref{go06}, the commutation rule \eqref{neg06}, the Leibniz rule and the algebra morphism property of $\rhotil_\trineg$,
	\begin{align*}
		\rhotil_\trineg (bU\pprod_\trineg \rr^{(\gamma,\mfl)}) &= \rhotil_\trineg(b)\sum_{k\in\N_0} \big(\tfrac{1}{k!} D^k \z^\gamma\big) \rhotil_\trineg (U) \partial_{(\mfl,k)} - \sum_{k\in\N_0} \big(\tfrac{1}{k!} D^k (b\trineg \rr^{(\gamma,\mfl)})\big) \partial_{(\mfl,k)}\\
		&=   \sum_{k\in\N_0} \big(\tfrac{1}{k!} D^k \z^\gamma\big) \rhotil_\trineg(b)\rhotil_\trineg (U) \partial_{(\mfl,k)}\\
		&=   \sum_{k\in\N_0} \big(\tfrac{1}{k!} D^k \z^\gamma\big) \rhotil_\trineg(bU) \partial_{(\mfl,k)}.
	\end{align*}
	\medskip
	
	For \eqref{pre02}, we argue by induction as in Lemma \ref{lem:go01} above. The case $U=1$ is trivial. We now assume it for $U$ and take $r\in R$. The Leibniz rule combined with the commutation \eqref{neg06} yields
	\begin{align*}
		\rhotil_\trineg (r) (\pi D \pi') = (\rhotil_\trineg (r)\pi) D \pi' + \pi D (\rhotil_\trineg (r)\pi').
	\end{align*}
	Using this, the algebra morphism property of $\tilde{\rho}_\trineg$ and the induction hypothesis,
	\begin{align*}
		\mbox{} & \rhotil_\trineg (rU)(\pi D \pi')\\
		 &= \rhotil_\trineg(r)\rhotil_\trineg(U)(\pi D \pi')\\
		 &= \rhotil_\trineg(r) \sum_{(U)} (\rhotil_\trineg (U_{(1)}) \pi) D (\rhotil_\trineg (U_{(2)})\pi') \\
		 &= \sum_{(U)} \big(\rhotil_\trineg(r)(\rhotil_\trineg (U_{(1)}) \pi) D (\rhotil_\trineg (U_{(2)})\pi') + (\rhotil_\trineg (U_{(1)}) \pi) D (\rhotil_\trineg(r)\rhotil_\trineg (U_{(2)})\pi')\big)\\
		 &= \sum_{(U)} \big((\rhotil_\trineg (rU_{(1)}) \pi) D (\rhotil_\trineg (U_{(2)})\pi') + (\rhotil_\trineg (U_{(1)}) \pi) D (\rhotil_\trineg (rU_{(2)})\pi')\big).
	\end{align*}
\end{proof}	
\begin{remark}
	Identity \eqref{pre02} is the analogue, on the dual side, of the cointeraction property \cite[Lemma 20]{BCFP}.
\end{remark}
\subsection{Some relevant groups of translations}\label{subsec::tran02}
We now want to study the group of characters generated via exponential maps as in Subsection \ref{subsec::2.3}. These, when well-defined, take the form of translations in $T^*$: Let
\begin{align}\label{tra31}
	\mathbf{c} = \sum_{(\gamma,\mfl)\in \renlie} \mathbf{c}_{(\gamma,\mfl)}\rr^{(\gamma,\mfl)}\in R^*,
\end{align}
and define for every $\mfl\in\mfL$
\begin{align}\label{tra32}
	c_{\mfl} := \sum_{\gamma\in \populated} \mathbf{c}_{(\gamma,\mfl)} \z^\gamma \in T^*.
\end{align}
Then \eqref{pre02bis} and \eqref{pre02} extend to
\begin{align}
	\tilde{\rho}_\trineg (\exp_{\trineg}(\mathbf{c})) \z_{(\mfl,0)} &= \z_{(\mfl,0)} + c_\mfl,\label{tra21}\\
	\tilde{\rho}_\trineg (\exp_{\trineg}(\mathbf{c})) (\pi D \pi') &= \big(\tilde{\rho}_\trineg (\exp_{\trineg}(\mathbf{c})) \pi\big)\, D\, \big(\tilde{\rho}_\trineg (\exp_{\trineg}(\mathbf{c})) \pi'\big).\label{tra22}
\end{align}
Our task now is to find a suitable graded structure that allows us to implement the strategy of Subsections \ref{subsec::2.2} and \ref{subsec::2.3}. Thanks to \eqref{tra11}, the pre-Lie algebra $(R,\trineg)$ is graded with respect to the following gradings:
\begin{align}
	(\gamma',\mfl') & \longmapsto \sum_{k'\in\N_0^d} \gamma'(\mfl,k') - \delta_\mfl^{\mfl'}\label{grad01}\\
	(\gamma',\mfl') & \longmapsto \sum_{(\mfl',k')\in\mfL\times \N_0^d}\hspace*{-10pt} k' \gamma'(\mfl',k') = \length{\gamma'} - [\gamma'] = \length{\gamma'} - 1.\label{grad02}
\end{align}
Note that \eqref{grad01} works as a grading for all $\mfl\in\mfL$; in addition, thanks to $[\gamma']=1$, \eqref{grad02} is obtained from \eqref{grad01} after summation in $\mfl'$. However, none of these gradings is strictly positive: In fact, \eqref{grad01} vanishes for $\gamma'= e_{(\mfl,0)}$. This is of course expected since
\begin{equation*}
	\rho_\trineg (\rr^{(e_{(\mfl,0)}, \mfl)}) = \sum_{k\in \N_0} (\tfrac{1}{k!} D^k \z_{(\mfl,0)}) \partial_{(\mfl,k)} = \sum_{k\in \N_0} \z_{(\mfl,k)} \partial_{(\mfl,k)}
\end{equation*}
which almost acts like an identity map, but makes it impossible to have a strictly positive grading unless we impose conditions under which $(e_{(\mfl,0)}, \mfl)$ is not allowed in the pre-Lie algebra. In this subsection we provide two different ways of doing this.

\medskip

We first consider the set
\begin{align*}
	\renlietwo := \{(\gamma,\mfl)\in\renlie\,|\,\length{\gamma}\geq 2\},
\end{align*}
and accordingly
\begin{align}\label{ren2}
	\Rentwo := \lspan_{(\gamma,\mfl)\in\renlietwo} \{\rr^{(\gamma,\mfl)}\}.
\end{align}
\begin{lemma}
	$(\Rentwo,\trineg)$ is a $1$-connected graded pre-Lie algebra with respect to $\grad{\gamma,\mfl}:=\length{\gamma} - 1$.
\end{lemma}	
\begin{proof}
	The first line in \eqref{tra11} implies that $\Rentwo$ is closed under $\trineg$, and thus it is a pre-Lie subalgebra of $R$. Moreover, for every $\gamma\in \populated$ with $\length{\gamma}\geq 2$
	\begin{align*}
		\length{\gamma}-1 \geq 1.
	\end{align*}
\end{proof}	
We now apply the Guin-Oudom procedure (Subsection \ref{subsec::2.1}) and the transposition (Subsection \ref{subsec::2.2}), which give rise to the group $\mathcal{G}(\Rentwo,\trineg)$ of multiplicative functionals over $\R[\renlietwo]$ (Subsection \ref{subsec::2.3}). Given an element $\mathbf{c}\in \Rentwo^*$, we have $\rho_\trineg(\exp_\trineg(\mathbf{c})) \in \textnormal{End}(\Rentwo^*)$, cf. Lemma \ref{lem:2.9}. In addition, the following holds.
\begin{lemma}\label{cor:tra01}
	For every $\mathbf{c}\in \Rentwo^*$, $\rhotil_\trineg (\exp_\trineg(\mathbf{c}))\in \textnormal{End}(T^*)$ and it satisfies \eqref{tra21} and \eqref{tra22}.
\end{lemma}	
\begin{proof}
	We know from Lemma \ref{lem:ren01} that $\rhotil_\trineg(\rr^{(\gamma',\mfl')}) \in \textnormal{End} (T^*)$. We also know from \eqref{tra11} that
	\begin{align*}
		\big(\rhotil_\trineg (\rr^{(\gamma',\mfl')})\big)_\beta^\gamma \neq 0\,\implies \length{\beta} - [\beta] = \length{\gamma'} - [\gamma'] + \length{\gamma} -[\gamma].
	\end{align*}
	Since $\length{\cdot} - [\cdot] = \length{\cdot} - 1 \geq 0$, we are in the context of Remark \ref{rem:fin10}, so Lemma \ref{lem:2.13} applies.
\end{proof}	

\medskip

A second strategy to prevent the infiniteness at the level of the Hopf algebra is as follows: We choose a subset $\emptyset\neq\subL\subset \mfL$ and define
\begin{align*}
	\lnh \beta \rnh_\subL := \length{\beta} - \sum_{(\mfl,k)\in \subL\times \N_0} \beta(\mfl,k) = \sum_{(\mfl,k)\in (\mfL\setminus \subL)\times \N_0} \beta (\mfl,k).
\end{align*}
Let
\begin{align*}
	\mcR_\subL &:= \{(\gamma,\mfl)\in \populated \times \subL\,|\, \lnh \gamma\rnh_{\subL} >0\},\\
	\hat{\mcR}_\subL &:= \{(\gamma,\mfl)\in \mcR_\subL\,|\, \length{\gamma} = \lnh \gamma \rnh_{\subL}\},
\end{align*}
and define $R_\subL$ and $\hat{R}_\subL$ accordingly. Roughly speaking, $\mcR_\subL$ encodes the transformations of the form \eqref{neg01} which always increase the contributions from $\mfL\setminus\subL$, whereas $\hat{\mcR}_\subL$ represent those which in addition strictly decrease (by $1$) the contributions from $\subL$. Since $\subL\neq \emptyset$, these restrictions prevent $(e_{(\mfl,0)},\mfl)$.
\begin{lemma}
	For every $\emptyset\neq\subL\subset \mfL$, $(R_\subL, \trineg)$ is a $1$-connected graded pre-Lie algebra with respect to $\grad{\gamma,\mfl} = \lnh\gamma\rnh_{\subL}$.
\end{lemma}	
\begin{proof}
As a consequence of \eqref{tra11}, if $(\gamma,\mfl)\in \mcR_\subL$
\begin{align}\label{noi01}
		\Big(\sum_{k\in\N_0} \big(\tfrac{1}{k!}D^{k} \z^{\gamma}\big) \partial_{(\mfl,k)}\Big)_{\beta'}^{\gamma'} \neq 0
	 \implies\,
	\lnh \beta' \rnh_\subL = \lnh \gamma' \rnh_\subL + \lnh \gamma \rnh_\subL.
\end{align}
Since for every $(\gamma,\mfl)\in \mcR_\subL$ it holds $\lnh \gamma \rnh_\subL \geq 1$, the claim follows.
\end{proof}
We apply the Guin-Oudom procedure (Subsection \ref{subsec::2.1}) and transposition (Subsection \ref{subsec::2.2}), producing the group $\mcG(R_\subL,\trineg)$ of multiplicative functionals over $\R[\renlie_\subL]$ (Subsection \ref{subsec::2.3}). Given an element $\mathbf{c}\in R_\subL^*$, we have $\rho_\trineg (\exp_\trineg (\mathbf{c}))$ $\in$ $\textnormal{End}(R_\subL^*)$, cf. Lemma \ref{lem:2.9}. We also have the corresponding extension of $\rhotil_\trineg$.
\begin{lemma}\label{cor:tra02}
	For every $\mathbf{c}\in R_\subL^*$, $\rhotil_\trineg(\exp_{\trineg}(\mathbf{c}))\in \textnormal{End}(T^*)$ and it satisfies \eqref{tra21} and \eqref{tra22}. In particular,
	\begin{align}
		\rhotil_\trineg(\exp_{\trineg}(\mathbf{c})) \z_{(\mfl,0)} &= \z_{(\mfl,0)} + c_\mfl \quad \quad\mbox{for all }\mfl\in\subL;\label{mc01b}\\
		\rhotil_\trineg(\exp_{\trineg}(\mathbf{c})) \z_{(\mfl,0)} &= \z_{(\mfl,0)} \quad \quad \quad \quad\mbox{for all }\mfl\in \mfL\setminus\subL.\label{mc01c}
	\end{align}
\end{lemma}	
\begin{proof}
	We know from \eqref{noi01} that
	\begin{align*}
		\big(\rhotil_\trineg(\rr^{(\gamma',\mfl')})\big)_\beta^\gamma \neq 0 \implies \lnh \beta\rnh_\subL = \lnh \gamma' \rnh_\subL + \lnh \gamma \rnh_\subL.
	\end{align*} 
	Since $\lnh\cdot\rnh_\subL\geq 0$ for all $\emptyset\neq \subL \subset \mfL$, we are in the situation of Remark \ref{rem:fin10}, thus the claim follows from Lemma \ref{lem:2.13}.
\end{proof}	
\begin{remark}\label{rem:4.14}
In the case of $\hat{R}_\subL$, since for $(\gamma,\mfl)\in \hat{\renlie}_\subL$ $\gamma$ does not depend on $\subL$, and $\mfl\in \subL$, we have
\begin{align*}
	\rr^{(\gamma,\mfl)}\trineg \rr^{(\gamma',\mfl')} = 0.
\end{align*}
Therefore, $(\hat{R}_\subL, \trineg)$ is trivial and Remark \ref{rem:trivial} yields the additive group structure
\begin{align*}
	\exp_{\trineg}(\mathbf{c})*\exp_{\trineg}(\mathbf{c}') = \exp_{\trineg}(\mathbf{c} + \mathbf{c}').
\end{align*}
A natural choice is $\subL = \{0\}$, which represents the case of algebraic renormalization in regularity structures: Only the drift term is affected ($\mfl = 0$), and counterterms depend on the rough drivers ($\mfl\neq 0$). In terms of \cite{BL23}, this corresponds to the combination of \cite[Assumption 2]{BL23} and \cite[(3.37)]{BL23}. For later purpose, let us denote
\begin{align*}
	\renlie_\circ := \renlie_{\{0\}},\,\,\,\hat{\renlie}_\circ := \hat{\renlie}_{\{0\}},
\end{align*}	
and similarly for $R_\circ$ and $\hat{R}_\circ$. Note that
\begin{align*}
	(\hat{R}_\circ,\trineg) \subset (R_\circ,\trineg) \subset (T, \trineg_0)
\end{align*}
are pre-Lie subalgebras.
\end{remark}
\subsection{Translated rough paths}\label{subsec::4.3}
In line with \eqref{TN}, we introduce the following notation: For every $N\in\N$, 
\begin{align*}
	R^{(N)} := \lspan_{(\gamma,\mfl)\in \renlie} \{ \rr^{(\gamma,\mfl)}\,|\, \length{\gamma} \leq N\},
\end{align*}
and the analogue for $\Rentwo^{(N)}$, $R_\subL^{(N)}$ and $\hat{R}_\subL^{(N)}$. Consider $\mathbf{c}\in \Rentwo^{(N)}$ and construct $\rhotil_\trineg (\exp_\trineg(\mathbf{c}))\in \textnormal{End}(T^*)$; note that the restriction to finite sums gives $\rhotil_\trineg (\exp_{\trineg}(\mathbf{c}))\in \textnormal{End}(T)$, and therefore by \eqref{tra22} $\rhotil_\trineg (\exp_{\trineg}(\mathbf{c}))$ is a pre-Lie morphism of $(T, D)$. By the universality property, it extends uniquely to an algebra morphism
\begin{align*}
	M_{\mathbf{c}} : \rmU(T, D) \longrightarrow \rmU(T, D).
\end{align*}
The same can be said for $\mathbf{c}\in R_\subL^{(N)}$.
\begin{lemma}
	For every $\mathbf{c}\in \Rentwo^{(N)}$ and every $U\in\rmU(T,D)$, it holds, as endomorphisms of $T$,
	\begin{align*}
		\rhotil_\trineg (\exp_\trineg(\mathbf{c})) \rho_\tripos (U) = \rho_\tripos(M_\mathbf{c} U) \rhotil_\trineg (\exp_\trineg(\mathbf{c})).
	\end{align*}
	The same holds for $\mathbf{c}\in R_\subL^{(N)}$  for some $\emptyset\neq \subL\subset \mfL$.
\end{lemma}	
\begin{proof}
	We show it inductively in $U\in \rmU(T, D)$. If $U$ is primitive, i.~e. $U\in T$, then it is a straightforward consequence of the pre-Lie morphism property \eqref{tra22} and the definition \eqref{go01} of $\rho_\tripos$. For the induction step, we assume it true for $U\in \rmU(T, D)$, take $b\in T$ and show it for $bU$: By the algebra morphism property of $\rho_\tripos$ and $M_\mathbf{c}$,
	\begin{align*}
		\rhotil_\trineg (\exp_{\trineg}(\mathbf{c})) \rho_\tripos (bU) &= \rhotil_\trineg (\exp_{\trineg}(\mathbf{c}))\rho_\tripos(b) \rho_\tripos(U) \\
		& = \rho_\tripos (\rhotil_\trineg (\exp_{\trineg}(\mathbf{c})) b) \rhotil_\trineg (\exp_{\trineg}(\mathbf{c}))\rho_\tripos (U)\\
		& = \rho_\tripos (\rhotil_\trineg (\exp_{\trineg}(\mathbf{c})) b) \rho_\tripos (M_\mathbf{c} U) \rhotil_\trineg (\exp_{\trineg}(\mathbf{c}))\\
		& = \rho_\tripos \big(\rhotil_\trineg (\exp_{\trineg}(\mathbf{c})) b(M_\mathbf{c} U)\big)  \rhotil_\trineg (\exp_{\trineg}(\mathbf{c})) \\
		& = \rho_\tripos(M_\mathbf{c}(bU)) \rhotil_\trineg (\exp_{\trineg}(\mathbf{c})).
		\end{align*}
\end{proof}	
Thanks to the finiteness properties, via Lemmas \ref{lem:2.9} and \ref{lem:2.13}, we have the following extension:
\begin{corollary}
	For all $\bff\in T^*$, and $\mathbf{c}$ $\in$ $\Rentwo^{(N)}$, it holds as endomorphisms of $T^*$
	\begin{align}\label{ren50}
		\rhotil_\trineg (\exp_{\trineg}(\mathbf{c})) \rho_\tripos (\exp_{D}(\bff)) = \rho_\tripos (\exp_{D}(\rhotil_\trineg (\exp_{\trineg}(\mathbf{c})) \bff)) \rhotil_\trineg (\exp_{\trineg}(\mathbf{c})).
	\end{align}
	The same holds if $\mathbf{c}\in R_\subL^{(N)}$ for some $\emptyset\neq \subL\subset \mfL$.
\end{corollary}	
As a straightforward consequence of \eqref{tra21}, \eqref{mc01b}, \eqref{mc01c} and \eqref{ren50}, we observe that the transformed solution of the hierarchy \eqref{rou01}, seen as \eqref{rou02}, satisfies a transformed hierarchy of equations.
\begin{corollary}
	Let $N\in \N$ and let $\mbfX:[0,1]^2 \to T^*$ satisfy \eqref{rou02}. For $\mathbf{c}\in R^{(N)}$, we use the shorthand notation $T_\mathbf{c} := \rhotil_\trineg(\exp_\trineg (\mathbf{c}))$.
	\begin{enumerate}[label=(\roman*)]
		\item Let $\mathbf{c}\in \Rentwo^{(N)}$. Then $T_\mathbf{c}\mbfX$ satisfies
		\begin{align}\label{ren06}
			T_\mathbf{c}\mbfX_{st} = \sum_{\mfl\in\mfL} \int_s^t \rho_\tripos (\exp_{D}(T_\mathbf{c}\mbfX_{su}))(\z_{(\mfl,0)} + c_\mfl )d X_u^\mfl. 
		\end{align}
		\item Let $\mathbf{c}\in R_\subL^{(N)}$ for some $\emptyset\neq \subL \subset \mfL$. Then $T_\mathbf{c}\mbfX$ satisfies 
		\begin{align*}
			T_\mathbf{c}\mbfX_{st} &= \sum_{\mfl\in\mfL} \int_s^t \rho_\tripos (\exp_{D}(T_\mathbf{c}\mbfX_{su}))\z_{(\mfl,0)}d X_u^\mfl \\
			&\quad + \sum_{\mfl\in \subL} \int_s^t \rho_\tripos (\exp_{D}(T_\mathbf{c}\mbfX_{su}))c_\mfl dX_u^{\mfl}. 
		\end{align*}
		In particular, if $\mathbf{c}\in R_\circ^{(N)}$, identifying $c = \sum_{\gamma\in \populated}\mathbf{c}_{(\gamma,0)}\z^\gamma$,
		\begin{align}
			T_\mathbf{c}\mbfX_{st} &= \sum_{\mfl\in\mfL} \int_s^t \rho_\tripos (\exp_{D}(T_\mathbf{c}\mbfX_{su}))\z_{(\mfl,0)}d X_u^\mfl\nonumber \\
			&\quad+ \int_s^t \rho_\tripos (\exp_{D}(T_\mathbf{c}\mbfX_{su}))c \,\, du. \label{ren20}
		\end{align}
	\end{enumerate}
\end{corollary}	
Note that the hierarchy of equations \eqref{ren06} is associated with the transformed RDE
\begin{align*}
	dY_t = \sum_{\mfl\in\mfL} (a_\mfl (Y_t) + c_\mfl[\mathbf{a}, Y_t]) dX_t^\mfl.
\end{align*}

\medskip

We finally show that this transformation, which is meaningful in the smooth case via the hierarchy \eqref{rou02}, is also a transformation in the space of rough paths.
\begin{theorem}\label{th:trp}
	Let $\mbfX: [0,1]^2\to T^*$ be a $\alpha$-rough path.
	\begin{enumerate}[label=(\roman*)]
		\item Let $\mathbf{c}\in \Rentwo^{(N)}$. Then $T_{\mathbf{c}}\mbfX $ is a $\alpha/N$-rough path.
		\item Let $\mathbf{c}\in R_\subL^{(N)}$. Then $T_{\mathbf{c}}\mbfX$ is a $\alpha/N$-rough path.  
		\item Let $\mathbf{c}\in R_\subL^{(N)}$. Assume in addition that for every $\mfl\in\subL$ there exists $\alpha_\mfl\geq \alpha$ such that for every $\beta\in\populated$
		\begin{align*}
			\sup_{0\leq s<t \leq 1} \frac{|\mbfX_{st \beta}|}{|t-s|^{|\beta|_\subL}} < \infty,
		\end{align*}
		where $|\beta|_\subL = \sum_{\mfl\in\subL}\alpha_\mfl\sum_{k\in\N_0} \beta(\mfl,k) + \alpha \sum_{\mfl\in\mfL\setminus \subL} \sum_{k\in\N_0} \beta(\mfl,k)$. Denote $\alphamin := \min_{\mfl\in\subL} \alpha_\mfl$. Then $T_{\mathbf{c}}\mbfX$ is a $(\alpha\wedge \alphamin/N)$-rough path.
	\end{enumerate}
\end{theorem}	
\begin{proof}
	In all cases, property \eqref{rou03} follows from \eqref{ren50}, so we focus on the analytic bounds. Recall \eqref{tra11}, which can be rewritten for $\beta,\gamma\in \populated$, $(\gamma',\mfl')$ $\in$ $\renlie$ and all $\mfl\in \mfL$
	\begin{align*}
		\big(\rhotil_\trineg (\rr^{(\gamma',\mfl')})\big)_\beta^\gamma \neq 0
		\implies \sum_{k\in\N_0} \beta(\mfl,k) = \sum_{k\in\N_0} \gamma(\mfl,k) + \sum_{k\in\N_0} \gamma'(\mfl,k) - \delta_\mfl^{\mfl'}.
	\end{align*}
	Then the grading \eqref{grad01}, iterated via \eqref{go06}, implies for every $\beta,\gamma\in\populated$ and $I\in M(\mcR)$
		\begin{align}
			&\big(\rhotil_\trineg(R_I)\big)_\beta^\gamma \neq 0\nonumber\\
			&\quad \implies
			\sum_{k\in\N_0} \beta(\mfl,k) = \sum_{k\in \N_0} \gamma(\mfl,k) + \hspace*{-5pt}\sum_{(\gamma',\mfl')\in \mcR} I(\gamma',\mfl') \big(\sum_{k\in\N_0}\hspace*{-5pt}\gamma'(\mfl,k) - \delta_\mfl^{\mfl'}\big). \label{grad03} 
	\end{align}	
	As a consequence, we have
	\begin{align*}
		\big(\rhotil_\trineg(R_I)\big)_\beta^\gamma \neq 0\,\implies \alpha\length{\beta} = \alpha\length{\gamma} + \alpha \sum_{(\gamma',\mfl')\in\mcR} I(\gamma',\mfl') (\length{\gamma'} - 1).
	\end{align*}
	Assuming $\mathbf{c}\in \Rentwo^{(N)}$, we have the upper bound
	\begin{align}\label{ineq}
		\alpha\length{\beta} \leq \alpha\length{\gamma} + \alpha(N - 1) \length{I}.
	\end{align}
	Recall now \eqref{goD02}, which implies for all $\mfl\in \mfL$
	\begin{equation}\label{*p33}
		\big(\rhotil_\trineg(R_I)\big)_\beta^\gamma \neq 0 \implies \sum_{\gamma'\in\populated} I(\gamma',\mfl) \leq \sum_{k\in\N_0} \gamma(\mfl,k).
	\end{equation}
	In particular $\length{I}\leq \length{\gamma}$, which combined with \eqref{ineq} yields
	\begin{align*}
		\alpha\length{\beta} \leq N\alpha\length{\gamma}.
	\end{align*}
	Thus,
	\begin{align*}
		|T_\mathbf{c} \mbfX_{st \beta}| &\leq \sum_{I\in M(\renlietwo)} \tfrac{1}{I!} |\mathbf{c}^I| \sum_{\gamma\in\populated} \big|\big(\rhotil_\trineg(R_I)\big)_\beta^\gamma\big| |\mbfX_{st\,\gamma}|\\
		& \lesssim  \sum_{I\in M(\renlietwo)} \tfrac{1}{I!} |\mathbf{c}^I| \sum_{\gamma\in\populated} \big|\big(\rhotil_\trineg(R_I)\big)_\beta^\gamma\big| |t-s|^{\alpha \length{\gamma}}\\
		&\lesssim |t-s|^{\frac{\alpha}{N}\length{\beta}};
	\end{align*}
	here all sums are finite and inequalities are up to a constant depending on $\mathbf{c}$, $\mbfX$ and $\beta$. This shows (i); the proof of (ii) is the same, assuming $\mathbf{c}\in R_\subL^{(N)}$.
	
	\medskip
	
	Turning our attention to (iii), we repeat the argument isolating the contributions of $\mfl\in\subL$ in \eqref{grad03} and \eqref{*p33}, yielding
	\begin{align*}
		\alpha\length{\beta} \leq \alpha N \sum_{\mfl\in\subL}\sum_{k\in \N_0} \gamma(\mfl,l) + \alpha \sum_{\mfl\in\mfL\setminus\subL} \sum_{k\in \N_0} \gamma(\mfl,k).
	\end{align*}	
	If $\alpha N < \alphamin$, then we simply get $\alpha\length{\beta} \leq \alpha\length{\gamma}$. If $\alpha N > \alphamin$, we divide by $\alpha N$ and obtain
	\begin{align*}
		N^{-1} \length{\beta} &\leq \sum_{\mfl\in\subL}\sum_{k\in \N_0} \gamma(\mfl,l) + N^{-1} \sum_{\mfl\in\mfL\setminus\subL} \sum_{k\in \N_0} \gamma(\mfl,k)\\
		& \leq \sum_{\mfl\in\subL}\sum_{k\in \N_0} \gamma(\mfl,l) + \alpha \alphamin^{-1} \sum_{\mfl\in\mfL\setminus\subL} \sum_{k\in \N_0} \gamma(\mfl,k)\\
		& \leq \alphamin^{-1} \Big(\sum_{\mfl\in\subL}\alpha_\mfl\sum_{k\in \N_0} \gamma(\mfl,l) + \alpha \sum_{\mfl\in\mfL\setminus\subL} \sum_{k\in \N_0} \gamma(\mfl,k)\Big).
	\end{align*}	
	The claim follows as before.
\end{proof}	
\begin{remark}\label{rem:4.18}
	The reader is invited to compare Theorem \ref{th:trp} with \cite[Theorem 30]{BCFP}, taking into account that our allowed translations are restricted to the sets $\Rentwo$ and $R_\subL$. Items (i) and (ii) are consistent with the Hölder regularity of the translated rough paths obtained in \cite[Theorem 30 (i)]{BCFP}. Item (iii) is a slight generalization of \cite[Theorem 30 (ii)]{BCFP}: In our framework, the latter corresponds to $\mathbf{c}\in R_\circ^{(N)}$, taking $\alpha_0 = 1$. 
\end{remark}	
\section{Connection to trees}\label{sec::5}
We conclude this work studying the relation between multi-indices and decorated trees. For this we first introduce some notation. Given $\mfL$, we denote the set of $\mfL$-decorated trees by $\mcH_\mfL$. These trees can be inductively constructed as follows: We start with a set of nodes
\begin{align*}
	\{\Xi_\mfl\}_{\mfl\in\mfL},
\end{align*}
then any tree $\tau\in\mcH_\mfL$ can be written as
\begin{align}\label{tree03}
	\tau = \Xi_\mfl \prod_{j=1}^k \mcI(\tau_j),
\end{align}
where $\tau_j\in\mcH_\mfL$ and $\mcI$ grows an incoming edge at the root of its argument. We identify $\Xi_\mfl$ with the root of $\tau$. We denote by $\mcB_\mfL$ the linear span of $\mcH_\mfL$.

\subsection{Branched rough paths and their translations}
Branched rough paths \cite{Gub06} are rough paths indexed by decorated trees $\mcH_\mfL$. In the spirit of Butcher series \cite{Butcher72,CHV}, for a smooth $\{X^{\mfl}\}_{\mfl\in\mfL}$ we can build a collection of branched integrals indexed by trees following the recursion \cite[(10)]{Gub06}: Namely, we fix
\begin{align*}
	 \mbfXbranched_{st} [\Xi_\mfl] = \int_s^t d X_r^\mfl = X_t^\mfl - X_s^\mfl
\end{align*}	
and then iteratively for $\tau$ as in \eqref{tree03}
\begin{align}\label{brp53}
	\mbfXbranched_{st}[\tau] = \int_s^t \prod_{j=1}^k \mbfXbranched_{sr}[\tau_j]  dX_r^\mfl.
\end{align}
We also associate to each tree a functional (called \textit{elementary differential}, cf. e.~g. \cite[Subsection 2.1]{CHV}) of $(\mathbf{a},y)$ fixing
\begin{equation}\label{brp56}
	\Xi_\mfl [\mathbf{a},y] = a_\mfl(y)
\end{equation}
and then recursively for $\tau$ as in \eqref{tree03} (recall we are in the one-dimensional case)
\begin{equation}\label{brp55}
	\tau [\mathbf{a},y] = \tfrac{d^k a_\mfl}{dy^k} (y) \prod_{j=1}^k \tau_j [\mathbf{a},y].
\end{equation}
We also define the symmetry factor of a tree $\sigma : \mcH_\mfL \to \N$ fixing $\sigma(\Xi_\mfl) = 1$ and for $\tau$ as in \eqref{tree03} by 
\begin{equation}\label{brp52}
	\sigma(\tau) = I! \prod_{j=1}^k \sigma(\tau_j),\,\,\,\, I=\sum_{j=1}^k e_{\tau_j} \in M(\mcH_\mfL)
\end{equation}
(cf. e.~g. \cite[p. 430]{Brouder}). For simplicity, let us define $\tmbfXbranched$ by
\begin{equation}\label{brp54}
	\tmbfXbranched_{st} [\tau] = \tfrac{1}{\sigma(\tau)} \mbfXbranched_{st} [\tau].
\end{equation}
With this notation, \cite[Theorem 5.1]{Gub06} implies that the solution of \eqref{rde01} admits the representation
\begin{equation}\label{brp58}
	Y_t - Y_s = \sum_{\tau\in \mcH_\mfL} \tmbfXbranched_{st} [\tau] \tau[\mathbf{a},Y_s].
\end{equation}

\medskip

As in the discussion leading to \eqref{defD}, there is a natural shift structure associated to the functionals $\tau$, namely the one given by the composition law of B-series cf. \cite[(3)]{CHV}. In order to describe it, for $\tau,\tau'\in \mcH_\mfL$ we set
\begin{equation*}
	(\tau'\curvearrowright \tau)[\mathbf{a},y] := \tfrac{d}{dh}|_{h=0} \tau [\mathbf{a}, y + h\tau'[\mathbf{a},y]].
\end{equation*}
It is easy to derive a recursive construction of $\curvearrowright$: We have
\begin{equation*}
	\big(\tau' \curvearrowright \Xi_\mfl\big)[\mathbf{a},y] =  \tfrac{da_\mfl}{dy}(y) \tau'[\mathbf{a},y] = \big(\Xi_\mfl \mcI(\tau')\big)[\mathbf{a},y]
\end{equation*} 
and for $\tau$ as in \eqref{tree03} by the chain rule
\begin{align*}
	&\big(\tau' \curvearrowright \tau \big)[\mathbf{a},y]\\
	&\quad = \tfrac{d^{k+1}a_\mfl}{dy^{k+1}}(y) \tau'[\mathbf{a},y] \prod_{j=1}^k \tau_j [\mathbf{a},y] + \tfrac{d^{k}a_\mfl}{dy^{k}}(y) \sum_{j=1}^k (\tau' \curvearrowright\tau_j) [\mathbf{a}, y] \prod_{i\neq j} \tau_i [\mathbf{a},y]\\
	&\quad = \Big(\Xi_\mfl \mcI(\tau') \prod_{j=1}^k \tau_j + \Xi_\mfl \sum_{j=1}^k (\tau' \curvearrowright\tau_j) \prod_{i\neq j} \tau_i\Big)[\mathbf{a},y].
\end{align*}
Now obviously $\curvearrowright : \mcB_\mfL \times \mcB_\mfL \to \mcB_\mfL$ is the grafting pre-Lie product, defined recursively by
\begin{align*}
	\tau' \curvearrowright \Xi_\mfl &= \Xi_\mfl \mcI(\tau'),\\
	\tau'\curvearrowright \tau &= \Xi_\mfl \mcI(\tau') \prod_{j=1}^k \mcI(\tau_j) + \Xi_\mfl \sum_{j=1}^k \mcI(\tau'\curvearrowright \tau_j)\prod_{i\neq j} \mcI(\tau_i).
\end{align*}
$(\mcB_\mfL,\curvearrowright)$ is the free pre-Lie algebra over $\R^\mfL$, cf. \cite{CL}. It is clearly graded under $N:\mcH_\mfL \to \N$, the number of nodes, recursively given by
\begin{equation*}
	N(\Xi_\mfl) = 1,\,\, N(\tau) = 1 + \sum_{j=1}^k N(\tau_j),
\end{equation*}
and is furthermore $1$-connected ($N(\tau)$ $\geq$ $1$ for every rooted tree $\tau$). The universal enveloping algebra $\rmU(\mcB_\mfL,\curvearrowright)$ is the so-called Grossman-Larson Hopf algebra, where the Grossman-Larson product is, in our representation, the concatenation product. If instead of the canonical duality pairing one imposes\footnote{ In the choice of the pairing \eqref{tree04}, we are following \cite[Section 4]{BCCH}.}
\begin{align}\label{tree04}
	\langle \tau',\tau \rangle_\sigma := \delta_\tau^{\tau'}\sigma(\tau),
\end{align}
the concatenation product is dual to the Connes-Kreimer coproduct, and therefore $(\R[\mcH_\mfL],\Delta_\curvearrowright)$ given in Proposition \ref{prop:tra01} is the Connes-Kreimer Hopf algebra \cite{CK1}. We refer the reader to \cite{Hoffman} for a complete exposition of the duality between Connes-Kreimer and Grossman-Larson.

\medskip

As in Lemma \ref{lem:3.5}, we now characterize $\tmbfXbranched$ as the solution of a differential equation.
\begin{lemma}\label{lem:5.1}
	$\tmbfXbranched_{st} = \sum_{\tau\in\mcH_\mfL} \tmbfXbranched_{st}[\tau] \tau$ $\in$ $\mcB_\mfL^*$ satisfies
	\begin{equation}\label{brp51}
		\tmbfXbranched_{st} = \sum_{\mfl\in\mfL} \int_s^t \rho_\curvearrowright(\exp_\curvearrowright(\tmbfXbranched_{sr})) \Xi_\mfl dX_r^\mfl.
	\end{equation}
\end{lemma}	
\begin{proof}
	First note that the r.~h.~s. of \eqref{brp51} is well-defined by Section \ref{sec::2} since $(\mcB_\mfL,\curvearrowright)$ is graded and $1$-connected. In addition, by \cite[Subsection 3.1]{GO2}, $\rho_\curvearrowright(\tau_1 \pprod_\curvearrowright\cdots \pprod_\curvearrowright \tau_k)$ is the simultaneous grafting of $\tau_1$,...,$\tau_k$: In particular, $\tau$ in \eqref{tree03} can be expressed as
	\begin{equation*}
		\tau = \rho_\curvearrowright(\tau_1 \pprod_\curvearrowright\cdots \pprod_\curvearrowright \tau_k) \Xi_\mfl.
	\end{equation*}
	Since $(\mcB_\mfL,\curvearrowright)$ is free, this decomposition is unique. We now let $I = \sum_{j=1}^k e_{\tau_j}$ $\in$ $M(\mcH_\mfL)$ and consider the corresponding Guin-Oudom basis element; then by the recursive definition \eqref{brp52} it holds
	\begin{equation*}
		\tfrac{\sigma(\tau_1)\cdots \sigma(\tau_k)}{\sigma(\tau)}\tau = \rho_\curvearrowright\big(\tfrac{1}{I!}\tau_1 \pprod_\curvearrowright\cdots \pprod_\curvearrowright \tau_k\big) \Xi_\mfl.
	\end{equation*}
	Now, appealing to \eqref{exp01}, we look at \eqref{brp51} componentwise, so that \eqref{brp51} is equivalent to showing for every $\tau$ of the form \eqref{tree03} 
	\begin{equation*}
		\tmbfXbranched_{st}[\tau] = \tfrac{\sigma(\tau_1)\cdots \sigma(\tau_k)}{\sigma(\tau)} \int_s^t \prod_{j=1}^k \tmbfXbranched_{st}[\tau_j] dX_r^\mfl.
	\end{equation*}
	This in turn follows from \eqref{brp53} and \eqref{brp54}.
\end{proof}	
\begin{remark}
	In the spirit of Remark \ref{rem:srp01}, equation \eqref{brp51} implies that $\tilde{\mathbf{X}}^{\textnormal{Br}} := \exp_\curvearrowright(\tmbfXbranched)$ solves
	\begin{equation*}
		d\tilde{\mathbf{X}}_{st}^{\textnormal{Br}} = \tilde{\mathbf{X}}^{\textnormal{Br}}_{st} dX_t
	\end{equation*} 
	where concatenation is understood in $\rmU (\mathcal{B}_\mfL,\curvearrowright)$; since the concatenation product is identified with the Grossman-Larson product, this is the ODE in \cite[proof of Theorem 4.2]{BFPP}.
\end{remark}	

\medskip

Translation maps $M_\mathbf{v}$, $\mathbf{v} = \{v_\mfl\}_{\mfl\in\mfL}\subset \mcB_\mfL^*$ in \cite[Subsection 3.2.3]{BCFP} are defined setting
\begin{align}\label{tree07}
	M_\mathbf{v} \Xi_\mfl = \Xi_\mfl + v_\mfl
\end{align}
and uniquely extending to a pre-Lie morphism in\footnote{ Actually, \cite{BCFP} chooses the standard pairing instead of \eqref{tree04}, so the pre-Lie product is rather the one in \cite[(99)]{CK1}. The two perspectives are of course equivalent.} $(\mcB_\mfL^*,\curvearrowright)$. One then recovers the effect on \eqref{rde01} of \cite[Theorem 38 (i)]{BCFP} taking $\mathbf{v} = \{v_\mfl\}_{\mfl\in\mfL}\subset \mcB_\mfL$ simply as
\begin{equation*}
	d Y_t = \sum_{\mfl\in\mfL} \big(a_\mfl(Y_t) + v_\mfl[\mathbf{a},Y_t]\big) d X^\mfl_t.
\end{equation*}
\subsection{Trees vs. multi-indices}
To establish the algebraic and combinatorial relation between trees and multi-indices, we first observe the following identity, already stated in \cite[Lemma 6.1]{LOT} (we omit the proof, as it is essentially the same).
\begin{lemma}\label{lem::tree1}
	Let $X^\mfl$ be smooth for all $\mfl\in \mfL$. Then
	\begin{align}\label{tree01}
		\mbfX_{st\,\beta} = \sigma(\beta) \sum_{\tau \in {\rm T}_\beta} \tmbfXbranched_{st}[\tau],
	\end{align}
	where
	\begin{align}
		{\rm T}_\beta := \big\{ \tau \in \mcH_\mfL\,|\, &\tau\,\mbox{contains }\beta(\mfl,k)\,\mbox{nodes of the form }\Xi_\mfl\nonumber\\
		&\mbox{with }k\,\mbox{outgoing edges}\big\}\label{brp59}
	\end{align}
	and
	\begin{equation*}
		\sigma(\beta):=\prod_{(\mfl,k)} (k!)^{\beta(\mfl,k)}
	\end{equation*}
	 is the size of the group of permutations of the descendents of every node of trees $\tau\in {\rm T}_\beta$.
\end{lemma}	
\begin{remark} A crucial observation from this construction is that for every $\tau\in {\rm T}_\beta$
\begin{itemize}
	\item $\length{\beta} = N(\tau)$, the number of nodes of $\tau$;
	\item $\sum_{(\mfl,k)} k\beta(\mfl,k)$ is the number of edges of $\tau$;
	\item the difference for a rooted tree is always $1$, which reflects $[\beta] =1$, cf. \eqref{rou10}.
\end{itemize}
In particular, this shows that multi-indices are more efficient than trees. We may even deduce a combinatorial interpretation: In the undecorated case, the number of trees with $n$ nodes is the Catalan number $C_{n-1} := \frac{(2n-2)!}{n! (n-1)!}$ (cf. e.~g. \cite{Stanley}), whereas the number of multi-indices $\beta\in\populated$ with $\length{\beta} = n$, which can be identified with the number of fertility configurations of $n-1$ edges, is the number $p(n-1)$ of partitions of $(n-1)$ in a sum of positive integers (cf. e.~g. \cite{Andrews}).
\end{remark}
\begin{remark}\label{rem:nov02}
	As a followup to Lemma \ref{lem:nov01}, the ideal defined in \cite[Proposition 7.9]{DL02} identifies trees with the same fertility configurations as \eqref{brp59}.  
\end{remark}	

\medskip

Next to \eqref{tree01}, the functionals \eqref{mi07} and \eqref{brp55} are related as follows.
\begin{lemma}\label{lem:6.5}
	For all $\tau\in \rmT_\beta$,
	\begin{equation}\label{brp57}
		\tau [\mathbf{a},y] = \sigma(\beta) \z^\beta [\mathbf{a},y].
	\end{equation}
\end{lemma}	
\begin{proof}
	For $\tau = \Xi_\mfl$, which is the unique element of $\rmT_{e_{(\mfl,0)}}$, it follows from \eqref{mi07} and \eqref{brp56}. For $\tau$ of the form \eqref{tree03}, let us assume that for $j=1,...,k$ we have $\tau_j \in \rmT_{\beta_j}$. Then $\tau \in \rmT_\beta$ with $\beta = e_{(\mfl,k)} + \beta_{1} +... + \beta_k$, so \eqref{brp57} follows recursively from \eqref{mi07} and \eqref{brp55}.
\end{proof}	
\begin{remark}
	The combination of \eqref{tree01} and \eqref{brp57} implies the equality between the expansions \eqref{expansion} and \eqref{brp58}: Indeed,
	\begin{equation*}
		\sum_\beta \mbfX_{st\,\beta}\z^\beta[\mathbf{a},Y_s] = \sum_\beta \sum_{\tau\in \rmT_\beta} \tmbfXbranched_{st}[\tau] \sigma(\beta)\z^\beta[\mathbf{a},Y_s] = \sum_\tau \tmbfXbranched_{st}[\tau] \tau[\mathbf{a},Y_s].
	\end{equation*}
\end{remark}	

\medskip

Consider the map $\Psi : \mcB_\mfL \to T$ generated by \eqref{brp57}
\begin{align*}
	\Psi[\tau] := \sum_{\beta}\sigma(\beta) \z^\beta  \delta_{\tau \in {\rm T}_\beta}.
\end{align*}
Of course this sum only involves one term, because $\beta$ is uniquely determined by $\tau$. Actually, this defines a map $\Psi : \mcB_\mfL^* \to T^*$, because ${\rm T}_\beta$ is a finite set for every $\beta$. The proof of Lemma \ref{lem:6.5} establishes the following recursion: $\Psi$ is defined by
\begin{align*}
	\Psi[\Xi_\mfl] = \z_{(\mfl,0)}
\end{align*}
and then for $\tau$ as in \eqref{tree03}
\begin{align}\label{tree05}
	\Psi[\tau] = k! \z_{(\mfl,k)} \prod_{j=1}^k \Psi[\tau_j],
\end{align}
which is precisely \cite[(2.48),(2.49)]{BL23} in our simpler situation. This map is a pre-Lie morphism; the following lemma is a particular case of \cite[Proposition 2.21]{BL23}.
\begin{lemma}
	$\Psi : (\mcB_\mfL,\curvearrowright)\to(T,D)$ is a pre-Lie morphism, i.~e.
	\begin{align}\label{tree02}
		\Psi [\tau \curvearrowright \tau'] = \Psi[\tau] D \Psi[\tau'].
	\end{align}
\end{lemma}	
From this property, $\Psi$ lifts to a Hopf algebra morphism
\begin{align*}
	\Psi : \rmU(\mcB_\mfL, \curvearrowright) \longrightarrow \rmU(T,D).
\end{align*}
The transposition, combined with the duality of Grossman-Larson and Connes-Kreimer \cite{Hoffman}, allows us to conclude the following.
\begin{corollary}
	Let $\Psi^\dagger$ be the transposition of $\Psi$ with respect to the pairing \eqref{tree04}. Then $\Psi^\dagger$ is a Hopf algebra morphism from $(\R[\populated],\Delta_D)$ to the Hopf algebra of Connes-Kreimer \cite{CK1}.
\end{corollary}	

\medskip

Regarding the translations \eqref{tree07}, the following lemma is the analogue of \cite[Lemma 6.6]{LOT}.
\begin{lemma}
	Whenever the r.~h.~s. makes sense,
	\begin{align}\label{tree06}
		\Psi M_\mathbf{v} = \rhotil_\trineg(\exp_{\trineg}(\Psi[\mathbf{v}])) \Psi,
	\end{align} 
	where $\Psi[\mathbf{v}] := \{\Psi[v_\mfl]\}_{\mfl\in\mfL}\subset T^*$ is identified with an element in $R^*$ as in \eqref{tra31} and \eqref{tra32}.
\end{lemma}	
\begin{proof}
	Since our exponential maps $\rhotil_\trineg(\exp_{\trineg}(\mathbf{c}))$, whenever well-defined, are pre-Lie morphisms with respect to $D$, cf. \eqref{tra22}, and since $\Psi$ is also a pre-Lie morphism by \eqref{tree02}, both sides of the equality are pre-Lie morphisms over $(\mcB_\mfL,\curvearrowright)$. It is then enough to show \eqref{tree06} when applied to elements of the form $\Xi_\mfl$. In such a case,
	\begin{align*}
		\Psi M_\mathbf{v} \Xi_\mfl = \Psi[\Xi_\mfl] + \Psi[v_\mfl] = \z_{(\mfl,0)} + \Psi[v_\mfl],
	\end{align*}	
	which coincides by \eqref{tra21} with
	\begin{align*}
		\rhotil_\trineg(\exp_{\trineg}(\Psi[\mathbf{v}])) \z_{(\mfl,0)} = \rhotil_\trineg(\exp_{\trineg}(\Psi[\mathbf{v}])) \Psi[\Xi_\mfl].
	\end{align*}
\end{proof}	
The cases for which \eqref{tree06} is meaningful covered in this paper can be identified with the following:
\begin{itemize}
	\item The set $\Rentwo^*$ corresponds in trees to translations of the form \eqref{tree07} for which the Lie series $\{v_\mfl\}_{\mfl\in \mfL}$ only contain trees with at least two nodes (equivalently, at least one edge).
	\item Given $\emptyset \neq \subL \subset \mfL$, the set $R_\subL^*$ corresponds in trees to translations \eqref{tree07} which leave $\{\Xi_\mfl\}_{\mfl\in\mfL\setminus\subL}$ invariant and such that $\{v_\mfl\}_{\mfl\in \mfL}$ only contains trees which have at least one node in $\mfL\setminus \subL$. The set $\hat{R}_\subL^*$ further restricts to Lie series $\{v_\mfl\}_{\mfl\in \mfL}$ $\subset$ $\mcB_{\mfL\setminus \subL}^*$.
	\item In particular, $R_\circ$ only shifts $\Xi_0$ by Lie series which cannot contain trees in $\mcH_{\{0\}}$, whereas $\hat{R}_\circ$ only allows trees in $\mcH_{\mfL\setminus \{0\}}$. The latter correspond to the transformations of the equation included in \cite[Theorem 38 (ii)]{BCFP}, as reflected in the hierarchy \eqref{ren20}.
\end{itemize}
\subsection{The insertion multi pre-Lie algebra}\label{subsec::5.2}
For every $\mfl\in\mfL$, consider the pre-Lie product $\triinsl : \mcB_\mfL \times \mcB_\mfL \to \mcB_\mfL$ given by
\begin{align}\label{ins01}
	\tau_1 \triinsl \tau_2 = \sum_\tau M_\mfl (\tau_1,\tau_2; \tau) \tau,
\end{align}
where $M_\mfl (\tau_1,\tau_2;\tau)$ counts the number of ways of inserting $\tau_1$ in a node of $\tau_2$ decorated by $\mfl$. Note that this is a generalization of the pre-Lie product defined in \cite[Subsection 4.4]{CEFM} in the undecorated setting and related to the substitution law of B-series \cite[(4)]{CHV}. Following \cite[(3.37)]{BM22}, we may describe this pre-Lie product in terms of simultaneous grafting. Recall from the proof of Lemma \ref{lem:5.1} that by \cite[Subsection 3.1]{GO2} simultaneous grafting is given in terms of $\rho_\curvearrowright$, so \cite[(3.37)]{BM22} implies that \eqref{ins01} may be rewritten as
\begin{align*}
	\tau' \triinsl \tau = \sum_{v\in N_{\tau;\mfl}} \big(\rho_\curvearrowright(\tau_1^{v}\pprod_\curvearrowright ... \pprod_\curvearrowright \tau_k^{v})\tau'\big) \curvearrowright_v (\tau\setminus \tau^v),
\end{align*}
where
\begin{itemize}
	\item the sum is over all nodes $v$ of $\tau$ decorated by $\mfl$;
	\item given a node $v$ we denote by $\tau^v$ the subtree of $\tau^v$ with $v$ as its root, and $\tau^v = \Xi_\mfl \prod_j \mcI(\tau_j^v)$;
	\item $\curvearrowright_v (\tau\setminus \tau^v)$ grafts exactly at the place where $v$ was, i.~e. grafts on $w$ where $w$ is such that there is an outgoing edge of $w$ which is the incoming edge of $v$.
\end{itemize}
This implies a recursive definition of $\triinsl$, namely
\begin{align}\label{ins03}
	\tau' \triinsl \Xi_{\mfl'} = \delta_\mfl^{\mfl'} \tau' 
\end{align}
and then recursively for $\tau$ as in \eqref{tree03}
\begin{align}\label{ins02}
	\tau' \triangleright_{{\rm ins},\mfl'} \tau = \delta_\mfl^{\mfl'} \rho_\curvearrowright (\tau_1 \pprod_\curvearrowright ... \pprod_\curvearrowright \tau_k) \tau' + \sum_{j=1}^k \Xi_\mfl \mcI (\sigma \triangleright_{{\rm ins},\mfl'} \tau_j) \prod_{i\neq j} \mcI (\tau_{i}).
\end{align}
\begin{lemma}
	$\Psi$ is a multi pre-Lie algebra morphism; in particular, for every $\mfl\in\mfL$,
	\begin{align}\label{ins07}
		\Psi [\tau \triinsl \tau'] = \Psi [\tau] \btl \Psi [\tau'].
	\end{align}
\end{lemma}	
\begin{proof}
	Note that the pre-Lie morphism \eqref{tree02} extends to the property
	\begin{align}\label{ins05}
		\Psi [ 	\rho_\curvearrowright (\tau_1 \pprod_\curvearrowright ... \pprod_\curvearrowright \tau_k) \tau' ] = \rho_D \big( \Psi[\tau_1]\pprod_D ... \pprod_D \Psi[\tau_k]\big) \Psi[\tau'].
	\end{align}
	We use this and the recursive definition \eqref{ins03}, \eqref{ins02} to prove our claim. First we note by \eqref{ins03}
	\begin{align*}
		\Psi [\tau' \triinslp \Xi_\mfl] = \delta_\mfl^{\mfl'} \Psi [\tau'] = \Psi [\tau'] \trineg_{\mfl'} \z_{(\mfl,0)} = \Psi [\tau'] \trineg_{\mfl'} \Psi[\Xi_\mfl]. 
	\end{align*}	
	Now for a tree $\tau$ of the form \eqref{tree03} it follows from \eqref{ins02} and \eqref{ins05} combined with \eqref{tree05} that
	\begin{align*}
		\Psi [\tau' \triinslp \tau] 
		&= \delta_\mfl^{\mfl'}\rho_D (\Psi[\tau_1]\pprod_D...\pprod_D \Psi[\tau_k]) \Psi[\tau'] \\
		& \quad + \sum_{j=1}^k k! \z_{(\mfl,k)} \Psi[\tau'\triinslp \tau_j] \prod_{i\neq j} \Psi[\tau_{i}] \\
		& = \delta_\mfl^{\mfl'}\rho_D (\Psi[\tau_1]\pprod_D...\pprod_D \Psi[\tau_k]) \Psi[\tau'] \\
		& \quad + \sum_{j=1}^k k! \z_{(\mfl,k)} (\Psi[\tau']\trineg_{\mfl'} \Psi[\tau_j]) \prod_{i\neq j} \Psi[\tau_{i}].
	\end{align*}
	As a consequence of \eqref{goD},
	\begin{align*}
		\mbox{}&\delta_\mfl^{\mfl'}\rho_D (\Psi[\tau_1]\pprod_D...\pprod_D \Psi[\tau_k]) \Psi[\tau']\\
		&\quad =  \big(\big(\tfrac{1}{k!}\Psi[\tau_1]\cdots \Psi[\tau_k] D^k\big) \Psi[\tau'] \big) k!\partial_{\z_{(\mfl',k)}} \z_{(\mfl,k)}.
	\end{align*}
	The claim follows from \eqref{tree05} and the Leibniz rule.
\end{proof}	

\medskip

Let us now restrict our construction to the undecorated case, i.~e. when $\mfL$ consists of one single element, and therefore we can drop all decorations. In particular, $\renlie = \populated$. In this situation, \eqref{ins07} takes the form of the pre-Lie morphism
\begin{align}\label{ins06}
	\Psi[\tau \triins \tau'] = \Psi[\tau] \trineg \Psi[\tau'],
\end{align}
where $\triins$ is now precisely the pre-Lie product defined in \cite[Subsection 4.4]{CEFM} (and denoted there as $\triangleright_\sigma$). This is the pre-Lie structure associated to the rooted tree Hopf algebra \cite[Subsection 4.1]{CEFM}, i.~e. the algebra of forests of rooted trees \textit{with at least one edge} equipped with the extraction-contraction coproduct. We denote by $\mcH_{\geq 2}$ the set of these trees, and $\mcB_{\geq 2}$ the vector space generated by them. Thanks to Lemma \ref{lem::tree1}, trees with at least one edge are those for which $\length{\beta} \geq 2$. This connects to the pre-Lie algebra $(\Rentwo,\trineg)$ of Subsection \ref{subsec::tran02}, now identified with $(T_{\geq 2},\trineg)$, and to which we applied the Guin-Oudom procedure (Subsection \ref{subsec::2.1}) and transposition (Subsection \ref{subsec::2.2}). By the universality property, \eqref{ins06} lifts $\Psi$ to a Hopf algebra morphism
\begin{align*}
	\Psi : \rmU(\mcB_{\geq 2},\triins) \longrightarrow \rmU(T_{\geq 2}, \trineg).
\end{align*} 
The transposition allows us to conclude the following.
\begin{corollary}\label{cor:hopf02}
	Let $\Psi^\dagger$ be the transposition of $\Psi$ with respect to the pairing \eqref{tree04}. Then $\Psi^\dagger$ is a Hopf algebra morphism from $(\R[\populated_{\geq 2}], \Delta_\trineg)$ to the rooted tree Hopf algebra of \cite[Subsection 4.1]{CEFM}.
\end{corollary}	

\medskip

\end{document}